\documentclass[a4paper,12pt]{article}
\makeatletter
\newcommand*{\rom}[1]{\expandafter\@slowromancap\romannumeral #1@}
\makeatother

\usepackage{lipsum,relsize,graphicx,amsmath,amsthm,amssymb,mathtools,tikz,stmaryrd,comment}
\usepackage[margin=0.8in]{geometry}
\usepackage{hyperref}
\usepackage{marginnote}

\usepackage{subfigure}

\usepackage{titlesec}

\graphicspath{ {c:/users/adjon/onedrive/documents/my papers} }

\newtheorem{theorem}{Theorem}[section]
\newtheorem{lemma}[theorem]{Lemma}
\newtheorem{corollary}[theorem]{Corollary}
\newtheorem{definition}{Definition}[section]
\newtheorem{proposition}[theorem]{Proposition}
\newtheorem{conjecture}[theorem]{Conjecture}

\newcounter{lthm}
\newtheorem{letterthm}[lthm]{Theorem}

\newcounter{lconj}
\newtheorem{letterconj}[lconj]{Conjecture}

\begin{document}

\title{Coefficient systems on the $\widetilde{A}_2$-Bruhat-Tits building}

\author{Adam Jones}

\date{\today}

\maketitle

\begin{abstract}
\noindent We address a conjecture (referred to as \textbf{sur} in \cite{torsion}) in the representation theory of a reductive $p$-adic Lie group $G$ which has important implications for the relationship between mod-$p$ smooth representations and pro-$p$ Iwahori-Hecke modules, and is currently only known for $G$ of rank 1. We prove that \textbf{sur} follows from exactness of the associated oriented chain complex of a coefficient system, when restricted to a local region of the Bruhat-Tits building for $G$. Our main result gives strong evidence towards this exactness in the case where $G=SL_3(K)$ for $K$ a totally ramified extension of $\mathbb{Q}_p$. We also develop new combinatorial techniques for analysing the geometric realisation of the $\widetilde{A}_2$ Bruhat-Tits building, which are fundamental to the proof of our main result, and which we hope will inspire further investigation in Bruhat-Tits theory.
\end{abstract}

\tableofcontents

\section{Background}

Throughout, we will let $p$ be prime, and let $K/\mathbb{Q}_p$ be a finite field extension with valuation ring $\mathcal{O}$, uniformiser $\pi$, residue field $\mathbb{F}_q=\mathcal{O}/\pi\mathcal{O}$. Fix $\mathbb{G}$ a split semisimple, simply connected algebraic group over $K$, and we will set $G:=\mathbb{G}(K)$.

\subsection{Smooth representations and Hecke modules}

\noindent For a field $k$, an ongoing project in number theory is to understand the smooth $k$-linear representation theory of $G$, which is of course essential within the Langlands programme. Indeed, when $k=\mathbb{C}$, the classical local Langlands correspondence yields a bijection between the irreducible, smooth $\mathbb{C}$-linear representations of $GL_n(K)$, and $n$-dimensional representations of the Weyl-Deligne group \cite{Harris-Taylor},\cite{Henniart}. 

A key ingredient in the proof of this correspondence is the relationship between the category Rep$_k^{\infty}(G)$ of smooth, $k$-linear representations of $G$, and the category of modules over the pro-$p$ Iwahori-Hecke algebra.\\ 

\noindent Throughout the paper, we will let $I$ be a pro-$p$ Iwahori subgroup of $G$ and let $\mathbb{X}:=k[G/I]$ be the standard module for $I$, which is of course a smooth $G$-representation. The \emph{pro-$p$ Iwahori-Hecke algebra} $\mathcal{H}$ is defined as $$\mathcal{H}=\mathcal{H}_I(G):=End_{k[G]}(\mathbb{X})^{\text{op}}$$ which is canonically isomorphic to $\mathbb{X}^I$ as a $k$-vector space, and clearly $\mathbb{X}$ has the structure of a right $\mathcal{H}$-module. To describe the important relationship between Rep$_k^\infty(G)$ and Mod$(\mathcal{H})$, consider the canonical adjunction between these categories:

\tikzset{every picture/.style={line width=0.75pt}} 
\begin{equation}\label{adjunction}
\begin{tikzpicture}[x=0.75pt,y=0.75pt,yscale=-1,xscale=1]

\draw    (260.76,163) -- (352,163.59) ;
\draw [shift={(354,163.6)}, rotate = 180.37] [color={rgb, 255:red, 0; green, 0; blue, 0 }  ][line width=0.75]    (10.93,-3.29) .. controls (6.95,-1.4) and (3.31,-0.3) .. (0,0) .. controls (3.31,0.3) and (6.95,1.4) .. (10.93,3.29)   ;
\draw    (353.24,176.6) -- (262,176.01) ;
\draw [shift={(260,176)}, rotate = 0.37] [color={rgb, 255:red, 0; green, 0; blue, 0 }  ][line width=0.75]    (10.93,-3.29) .. controls (6.95,-1.4) and (3.31,-0.3) .. (0,0) .. controls (3.31,0.3) and (6.95,1.4) .. (10.93,3.29)   ;
\draw    (260.48,214.43) -- (358,214.42) ;
\draw [shift={(360,214.41)}, rotate = 179.99] [color={rgb, 255:red, 0; green, 0; blue, 0 }  ][line width=0.75]    (10.93,-3.29) .. controls (6.95,-1.4) and (3.31,-0.3) .. (0,0) .. controls (3.31,0.3) and (6.95,1.4) .. (10.93,3.29)   ;
\draw    (260,209.6) -- (260,219.6) ;
\draw    (359,253.7) -- (262,254.58) ;
\draw [shift={(260,254.6)}, rotate = 359.48] [color={rgb, 255:red, 0; green, 0; blue, 0 }  ][line width=0.75]    (10.93,-3.29) .. controls (6.95,-1.4) and (3.31,-0.3) .. (0,0) .. controls (3.31,0.3) and (6.95,1.4) .. (10.93,3.29)   ;
\draw    (359,248.8) -- (359,258.6) ;

\draw (187,157) node [anchor=north west][inner sep=0.75pt]   [align=left] {$\displaystyle\text{Rep}_{k}^{\infty }( G)$};
\draw (362,161) node [anchor=north west][inner sep=0.75pt]   [align=left] {$\displaystyle\text{ Mod}(\mathcal{H})$};
\draw (161,158) node [anchor=north west][inner sep=0.75pt]   [align=left] {$\displaystyle \mathfrak{h} :$};
\draw (430,162) node [anchor=north west][inner sep=0.75pt]   [align=left] {$\displaystyle :\mathfrak{t}$};
\draw (205,207) node [anchor=north west][inner sep=0.75pt]   [align=left] {$\displaystyle V$};
\draw (382,207) node [anchor=north west][inner sep=0.75pt]   [align=left] {$\displaystyle V^{I}$};
\draw (382,245) node [anchor=north west][inner sep=0.75pt]   [align=left] {$\displaystyle M$};
\draw (170,246) node [anchor=north west][inner sep=0.75pt]   [align=left] {$\displaystyle \mathbb{X} \otimes _{\mathcal{H}} M$};

\end{tikzpicture}
\end{equation}

\noindent In the case where $k$ has characteristic $\ell\neq p$, this pair of functors yields an equivalence between Mod$(\mathcal{H})$ and the category Rep$_k^{\infty}(G)^I$ of representations generated by their pro-$p$ Iwahori fixed vectors. When $k$ has characteristic $p$, this equivalence holds when $G=GL_2(\mathbb{Q}_p)$ \cite{equiv GL_2} or $SL_2(\mathbb{Q}_p)$ \cite{equiv SL_2}, which allows us to recover the classification of smooth, admissible irreducible representations of these groups obtained in \cite{BL} and \cite{Abd'}. 

However, in all other cases where char$(k)=p$, the invariance functor $\mathfrak{h}$ fails to even be right exact \cite{exact}. In \cite{dupre}, the functors were lifted to the associated \emph{model categories} of $Rep_k^{\infty}(G)$ and Mod$(\mathcal{H})$, obtaining an adjunction which is far better behaved homologically, and a derived version of the equivalence does hold \cite[Theorem 9]{derived}. But since we do not yet have a proper understanding of the \emph{d.g. graded pro-$p$ Iwahori-Hecke algebra} \cite[Section 3]{derived}, which is crucial to this derived equivalence, this does not necessarily resolve the problem.\\ 

\noindent Our aim is to develop our understanding of these functors in natural characteristic $p$ on a more explicit level, which we anticipate will advance the mod-$p$ local Langlands programme.

\subsection{The torsionfree category}

The largest obstacle to understanding the relationship between smooth representations and Hecke modules in characteristic $p$, and understanding the mod-$p$ representation theory of $G$ more generally when $G\neq GL_2(\mathbb{Q}_p)$ or $SL_2(\mathbb{Q}_p)$, seems to be the notion of supersingular representations. We will not give a precise definition (see \cite[section 1.2.1]{Herzig}), but in the case when $G=GL_n(K)$, an irreducible, admissible representation $V$ of $G$ is \emph{supersingular} if it cannot be realised as a subquotient of a parabolic induction \cite[Corollary 1.2]{Herzig}.\\ 

\noindent In characteristic $\ell\neq p$, these are the well-studied supercuspidal representations, which can be realised as compact inductions when $k$ is algebraically closed and $p$ does not divide the order of the absolute Weyl group of $G$ \cite[Corollary 3]{Fintzen}. 

In characteristic $p$, supersingular representations are not well understood. They can still be related to compact inductions by \cite[Theorem 5.27]{supersingular module}, but they fail even to be finitely presented when $G\neq GL_2(\mathbb{Q}_p)$ or $SL_2(\mathbb{Q}_p)$ \cite{Schraen},\cite{Wu}.\\

\noindent On the other hand, we say that a finite length module over the pro-$p$ Iwahori-Hecke algebra $\mathcal{H}$ is \emph{supersingular} if it is killed by a power of a canonical ideal $\mathfrak{J}$ of the centre $Z(\mathcal{H})$ of $\mathcal{H}$ \cite[Proposition-Definition 5.10]{supersingular module}. In recovering a version in characteristic $p$ of the equivalence defined by the adjoint functors $\mathfrak{h}$ and $\mathfrak{t}$ in characteristic 0, a very promising approach has been to remove the supersingular objects from both categories.\\

\noindent For example, it was proved by Schneider and Ollivier in \cite[Theorem 0.5]{torsion} that if $G=SL_2(K)$, then the functor $\mathfrak{t}$ restricts to a fully faithful functor from the category Mod$(H_\zeta)$ of finite length $\mathcal{H}$-modules where the canonical generator $\zeta$ of $\mathfrak{J}$ acts invertibly (which of course excludes all supersingular modules). The image of Mod$(H_\zeta)$ under $\mathfrak{t}$ coincides with the category of smooth, finite length representations that arise as subquotients of parabolic inductions (i.e. non-supersingular representations), which is equivalent to Mod$(H_\zeta)$ via the $I$-invariance functor $\mathfrak{h}$. 

This result was generalised by Abe in \cite[Corollary 4.2]{Abe} to any connected, reductive $p$-adic Lie group $G$, allowing us to realise a similar equivalence between a category of Hecke modules and the category of smooth, finite length $G$-representations that arise as subquotients of \emph{maximal} parabolic inductions. But of course, in rank greater than 1, this excludes many examples of non-supersingular modules.\\

\noindent However, the result of Schneider and Ollivier in \cite{torsion} is actually much stronger. Specifically, they exhibit a category $\mathcal{F}$ of $\mathcal{H}$-modules, excluding all supersingular modules, and their main result \cite[Theorem 0.1, Corollary 2.7]{torsion} demonstrates that there is a fully faithful functor $\mathfrak{t}':\mathcal{F}\to Rep_k^{\infty}(G)$ with inverse given by $\mathfrak{h}$, and which restricts to the induction functor $\mathfrak{t}$ on Mod$(H_\zeta)$.\\

\noindent Indeed, the construction of $\mathcal{F}$ and $\mathfrak{t}'$ is very general \cite[section 1.4]{torsion}, and applies to \emph{any} reductive $p$-adic Lie group $G$. To briefly summarise, for each $m\in\mathbb{N}$, let $K_m$ be the $m$'th congruence subgroup, and $\mathbb{X}^{K_m}$ is a $(\mathcal{H},I)$-bimodule, so we can define the $\mathcal{H}$-module $(\mathbb{X}^{K_m})^*_I$ (the $I$-coinvariance of the $\mathcal{H}$-dual), and we obtain an exact sequence $$0\to Z_m\to (\mathbb{X}^{K_m})^*_I\to \mathcal{H}^*$$ where the latter map is given by restriction. 

We now define $\mathcal{F}$ as the category of all $\mathcal{H}$-modules $M$ with Hom$_{\mathcal{H}}(Z_m,M)=0$ for each $m\in\mathbb{N}$. We can think of $\mathcal{F}$ as the the torsionfree part of the torsion pair in Mod$(\mathcal{H})$ defined by $\{Z_m:m\in\mathbb{N}\}$. We can then define $\mathfrak{t}'$ as the functor that sends a module $M\in\mathcal{F}$ to the image of the map of smooth $G$-representations $$\mathbb{X}\otimes_\mathcal{H}M\to Hom_\mathcal{H}(\mathbb{X}^*,M),x\otimes m\mapsto (\lambda\mapsto\lambda(x)m)$$

\noindent The conjecture below was denoted by \textbf{(sur)} in \cite{torsion}, and it stands as the largest obstacle to understanding $\mathcal{F}$ in higher rank:

\begin{letterconj}\label{sur}
The canonical morphism $\mathbb{X}^*\to\mathcal{H}^*$ is a surjection for all $m\in\mathbb{N}$.
\end{letterconj}

\noindent It follows from this conjecture that $0\to Z_m\to(\mathbb{X}^{K_m})^*_I\to \mathcal{H}^*\to 0$ is exact for all $m\in\mathbb{N}$, thus $M\cong Hom_{\mathcal{H}}(M,(\mathbb{X}^{K_m})^*_I)$ for all $M\in\mathcal{F}$. Using the argument in the proof of \cite[Theorem 1.9]{torsion}, Conjecture \ref{sur} is all that is required to prove that $\mathfrak{t}'$ is fully faithful, with inverse $\mathfrak{h}$, and hence that $\mathcal{F}$ embeds faithfully into $Rep_k^\infty(G)$.\\

\noindent Conjecture \ref{sur} holds when char$(k)\neq p$ \cite[Lemma 1.8]{torsion}, or if $G$ has rank 1 \cite[Corollary 2.7]{torsion}, but in general it remains open. Our ultimate aim is to prove Conjecture \ref{sur} in characteristic $p$ for any choice of semisimple, simply connected $p$-adic Lie group $G=\mathbb{G}(K)$, but in this paper, we will focus on the smallest case not currently known, when $G=SL_3(K)$. 

\subsection{Coefficient systems on the Bruhat-Tits building}

The proof of Conjecture \ref{sur} in rank 1 \cite[Corollary 2.7]{torsion} makes use of the Bruhat-Tits tree $T_q$, which can be simply defined as the tree where each vertex has degree $q+1$. However, its vertices can be realised as rank 2 lattices in $K^2$ modulo scaling, so it carries a natural action of $PGL_2(K)$, and hence of $GL_2(K)$ and $SL_2(K)$.

Arguably the most important ingredient of the proof is the \emph{coefficient system} of $\mathbb{X}$, and its associated oriented chain complex, which has the form $$0\to C_1(T_q,\mathbb{X})\to C_0(T_q,\mathbb{X})\to\mathbb{X}\to 0$$ where $C_0(T_q,\mathbb{X})$ (resp. $C_1(T_q,\mathbb{X})$) is a space of functions from the set of vertices (resp. oriented edges) of $T_q$ to $\mathbb{X}$ with finite support. These spaces have the structure of $(\mathcal{H},I)$-bimodules, and the sequence obtained is exact by \cite[Remark 3.2]{gorenstein}. It is straightforward to prove \cite[Lemma 2.2]{torsion} that the sequence remains exact when we restrict to the sequence defined on a finite region of the tree, which is a crucial detail in the argument.\\

\noindent More generally, for any reductive $p$-adic Lie group $G=\mathbb{G}(K)$ of rank $d$, there is a canonically defined Bruhat-Tits building $\Delta=\widetilde{\Delta}(G)$, which can be realised as a polysimplicial complex of dimension $d=\dim(G)$, which coincides with $T_q$ when $G$ has rank 1. This building also carries a transitive action of $G$, indeed the pro-$p$ Iwahori subgroup $I$ can be most easily defined as the Sylow $p$-subgroup of the stabiliser of a maximal simplex (or \emph{chamber}) in $\Delta$.

We can also define a coefficient system on $\Delta$ completely analogously to the rank 1 case (see \cite[Chapter \rom{2}]{coefficient} for details), where we extend the chain complex to include the higher space of functions $C_i(\Delta,\mathbb{X})$ defined on oriented $i$-simplices in $\Delta$, for each $i\leq d$, and the sequence remains exact.\\

\noindent Of course, in rank greater than 1, the Bruhat-Tits building is no longer a tree, and its structure becomes immeasurably more complex. Even in the simplest rank 2 case, where $G$ has type $\widetilde{A}_2$, there is very little material in the literature that deals with the building explicitly (see \cite{gallery} for an overview). Without the assumption that $\Delta$ is a tree, it becomes very difficult to control the local behaviour of the coefficient system, i.e. what happens when we restrict to functions defined on a fixed, bounded region in the building.\\

\noindent More generally, whenever $\mathcal{X}$ is a set of $j$-facets in $\Delta$, for some $j\leq d$, we define $C_i(\mathcal{X},\Delta)$ for each $i\leq d$ to be the space of functions in $C_i(\Delta,\mathbb{X})$ with support in $\mathcal{X}$. Restricting to these spaces, we deduce the following chain complex of $\mathcal{H}$-modules 
\begin{equation}\label{eqn: local coefficient system 1}
0\to C_d(\mathcal{X},\mathbb{X})\to\dots\to C_1(\mathcal{X},\mathbb{X})\to C_0(\mathcal{X},\mathbb{X})\to\mathbb{X}
\end{equation}

\noindent\textbf{Note:} 
\begin{enumerate}
\item If $\mathcal{X}$ is $I$-invariant, then $C_i(\mathcal{X},\mathbb{X})$ is a $(\mathcal{H},I)$-submodule of $C_i(\Delta,\mathbb{X})$ for each $i$. But in general, $C_i(\mathcal{X},\mathbb{X})$ need not carry an $I$-action.\\

\item Of course, if $i>j$ then $C_i(\mathcal{X},\mathbb{X})=0$, since $\mathcal{X}$ contains no $i$-simplices. However, we will usually assume that $\mathcal{X}$ consists of chambers (i.e. $d$-simplices).
\end{enumerate}

\noindent In general, it is not clear whether this restricted sequence is exact, but we do suspect that it is in several important cases. In section \ref{sec: buildings}, we will define the region $\Delta_n$ of $\Delta$, for each $n\geq 0$, consisting of all chambers of distance no more than $n$ from the hyperspecial chamber $C$, and we will spend much of this section exploring the geometric and combinatorial properties of this region in type $\widetilde{A}_2$.

In section \ref{sec: coefficient systems} we will define a \emph{complete region} $\mathcal{X}$ of $\Delta$ to be a set of chambers with $\Delta_n\subseteq\mathcal{X}\subset\Delta_{n+1}$ for some $n\in\mathbb{N}$.

\begin{letterconj}\label{conj: exactness 1}
Suppose $\mathcal{X}$ is a set of facets in $\Delta$ satisfying one of the two properties below:\\

\emph{\textbf{(A)}} $\mathcal{X}$ is an $I$-invariant complete region in $\Delta$.\\

\emph{\textbf{(B)}} $\mathcal{X}$ consists of a single face of the hyperspecial chamber $C$.\\

\noindent Then the restricted chain complex (\ref{eqn: local coefficient system 1}) is exact.
\end{letterconj}





\noindent Conjecture \ref{conj: exactness 1} is easy to prove when $G$ has rank 1 and $\Delta=T_q$ is a tree (\cite[Lemma 2.2]{torsion}). But in higher ranks, without the tree structure, the proof fails and there are no cases when it is known to hold. 

Still, our first main result demonstrates that this conjecture is the only obstacle to a full proof of Conjecture \ref{sur}, and the rest of the proof of \cite[Corollary 2.7]{torsion} can be shown to generalise.

\begin{letterthm}\label{letterthm: implies sur}
Suppose $G=\mathbb{G}(K)$ for $\mathbb{G}$ split semisimple, simply connected. If Conjecture \ref{conj: exactness 1} holds for $\Delta=\Delta(G)$, then $\mathbb{X}^*\to\mathcal{H}^*$ is surjective.
\end{letterthm}

\noindent Let $v_0$ be the hyperspecial vertex in $\Delta$, and define the \emph{star} of $v_0$ (denoted $Star(v_0)$) to be the region consisting of all chambers containing $v_0$ in their closure. We say that a collection of chambers $\mathcal{X}\subseteq \Delta$ is \emph{star-bounded} if

\begin{itemize}

\item $C\in\mathcal{X}$,

\item $\mathcal{X}\subseteq Star(v_0)$, and

\item $\mathcal{X}$ is $I$-invariant.


\end{itemize}

\noindent We believe that proving exactness of (\ref{eqn: local coefficient system 1}) in the case when $\mathcal{X}$ is star-bounded should suffice to complete a proof of Conjecture \ref{conj: exactness 1}, and we explore this idea in the case where $G=SL_3(K)$, and $\Delta$ is the $\widetilde{A}_2$-Bruhat-Tits building.

\begin{letterthm}\label{letterthm: exact for star}
Let $G=SL_3(K)$, for $K/\mathbb{Q}_p$ a totally ramified extension, and suppose that \emph{char}$(k)=p$. Then if the restricted chain complex (\ref{eqn: local coefficient system 1}) is exact for any star-bounded region $\mathcal{X}$, then $\mathbb{X}^*\to\mathcal{H}^*$ is surjective.
\end{letterthm}

\noindent In fact, it suffices only to prove that star-bounded regions satisfy the weaker property of being Star$(v_0)$\emph{-collapsible} (Definition \ref{defn: collapsible}), which we can prove at least in the case where $\mathcal{X}=\{C\}$ (Theorem \ref{thm: final}).


\subsection{Structure of paper}

Proving Conjecture \ref{conj: exactness 1} in type $\widetilde{A}_2$ is very technical, and requires in depth analysis of the combinatorial structure of the Bruhat-Tits building. Section \ref{sec: buildings} will largely be dedicated to revising the relevant theory of buildings, and proving technical results concerning the local structure of the $\widetilde{A}_2$ building, and the canonical pro-$p$ subgroups associated to its facets.

In section \ref{sec: coefficient systems}, we will formally define coefficent systems and local oriented chain complexes, complete the proof of Theorem \ref{letterthm: implies sur}, and begin to explore some avenues by which Conjecture \ref{conj: exactness 1} could be proved in general.\\

\noindent In section \ref{sec: orbits}, we will use the results we obtained in section \ref{sec: buildings} regarding the action of $G$ on $\Delta_n$ to explore the action of $G$ on the space $C_1(\Delta_n,\mathbb{X})$, which will ultimately culminate in a full proof of Theorem \ref{letterthm: exact for star}, modulo some technical details which will be completed in section \ref{sec: small cases}. The central idea of the proof is to regard the surface of $\Delta_n$ as (almost) a disjoint union of regions isometric with $Star(v_0)$, and we shift chains in $C_1(\Delta_n,\mathbb{X})$ to chains defined on these regions.\\



\noindent In light of Theorem \ref{letterthm: exact for star}, the only obstacle to proving Conjecture \ref{sur} for $SL_3(K)$ is to prove that (\ref{eqn: local coefficient system 1}) is exact for appropriate $I$-invariant regions $\mathcal{X}$ contained in the star of $v_0$. In the remainder of section \ref{sec: small cases}, we will examine the star in detail, and explore avenues of proving this exactness, focusing on the case where $\mathcal{X}=\{C\}$.\\


\noindent\textbf{Acknowledgements:} The author is very grateful to Nicolas Dupr\'{e} for many helpful and constructive comments on this work, for assisting in resolving some serious issues, and for providing considerable simplifications to the Bruhat-Tits theory involved. The author would also like to thank the Heilbronn Institute for Mathematical Research for funding and supporting this research.

\section{The Bruhat-Tits building}\label{sec: buildings}

This section serves as a primer for the theory of buildings and Bruhat-Tits theory, but the results we prove here will be essential in our main argument. Very little exists in the literature exploring the structure of the Bruhat-Tits building explicitly in higher ranks. In type $\widetilde{A}_2$, the best resource currently available is \cite{gallery}, which is the only resource which provides an explicit realisation of the building. We will reprove this realisation combinatorially, and develop techniques for working with it practically.

\subsection{Recap on buildings}

Typically, a \emph{building} is realised as a simplicial complex with additional geometric structure defined by an associated Coxeter group, similar to the usual Coxeter complex  \cite{Ronan}, \cite{Brown}. So throughout, we will will fix an irreducible Coxeter system $(W,S)$, with $\vert S\vert=d<\infty$.\\ 

\noindent Formally, there are a number of equivalent definitions. In \cite[Chapter 3.1]{Ronan}, a building over $W$ is defined to be a chamber system $\Delta$ (as defined in \cite[Chapter 1.1]{Ronan}), where the adjacency relations $C\sim_s D$ are defined for \emph{chambers} $C,D\in\Delta$ using the generators $s\in S$, together with a function $\delta:\Delta\times\Delta\to W$ such that for any minimal gallery of chambers $C_0\sim_{s_1} C_1\sim_{s_2}\dots\sim_{s_r} C_r$ in $\Delta$,

\begin{itemize}
    \item $\delta(C_0,C_r)=s_1s_2\dots s_r$.

    \item $\delta(C_0,C_r)$ has length $r$ in $W$.
\end{itemize}

\noindent We define the \emph{distance between chambers} $C,D\in\Delta$ to be $d(C,D):=\ell(\delta(C,D))$, where $\ell$ is the length function in $W$. From the properties of $\delta$, we see that this is equal to the length of any minimal gallery from $C$ to $D$.\\

\noindent On the other hand, in \cite[Chapter \rom{4}.1]{Brown}, a building over $W$ it is explicitly defined as a simplicial complex $\Delta$ of uniform dimension $d$, which arises as a union of a collection of subcomplexes $\{\mathcal{A}_i:i\in I\}$ called \emph{apartments} such that 

\begin{itemize}
    \item Each $\mathcal{A}_i$ is isomorphic to the Coxeter complex of $W$. In particular, the maximal simplices in $\mathcal{A}_i$ have codimension 1 faces indexed by the elements of $S$.

    \item For any two maximal simplices $C,D$, there is an apartment containing $C$ and $D$.

    \item For all apartments $\mathcal{A}_i,\mathcal{A}_j$, there is an isometry $\iota:\mathcal{A}_i\to\mathcal{A}_j$ which fixes $\mathcal{A}_i\cap\mathcal{A}_j$ pointwise.
\end{itemize}

\noindent These definitions are equivalent because if $\Delta$ is defined as a simplicial complex, we can realise it as a chamber system over $S$ by defining the chambers to be maximal simplices (i.e. $d$-simplices), and if $C,D\in\mathcal{A}_i$ with $\tau:\mathcal{A}_i\cong W$, then $\delta(C,D):=\tau(C)^{-1}\tau(D)$ (which is independent of the choice of apartment $\mathcal{A}_i$). 

Conversely, if $\Delta$ is defined as a chamber system, then its geometric realisation will be a simplicial complex of dimension $d$, and the apartments can be defined as all isometric images of the Coxeter complex of $W$ in this realisation.\\

\noindent We will alternate between these definitions quite liberally in this paper, but we can always regard a building $\Delta$ as a simplicial complex of dimension $d$, and its $d$-simplices are called \emph{chambers}. We will also often reference the function $\delta:\Delta\times\Delta\to W$, defined on pairs of chambers, and to apartments in the building.

Of course, as a simplicial complex, a building is defined uniquely by its graph structure, i.e. its vertices and edges. Any complete subgraph $F$ on $n\leq d$ vertices forms a codimension $d-n$ face of a chamber in $\Delta$, and we call this a \emph{facet} in $\Delta$.\\

\noindent \textbf{Convention:} When referring to \emph{subsets} of a building $\Delta$, it is often unclear whether we are referring to sets of chambers, or sets of smaller facets, e.g. vertices, edges, etc. So for clarity, in this paper, when we refer to a \emph{subset} $\mathcal{X}$ of $\Delta$, we mean a set of vertices, but for any facet $F$ in $\Delta$, we write $F\in\mathcal{X}$ to mean that all the vertices of $F$ lie in $\mathcal{X}$. 

More generally, we say that $\mathcal{X}$ is a \emph{set of $i$-facets} if for all vertices $v\in\mathcal{X}$, there is an $i$-facet $F\in\mathcal{X}$ containing $v$.


\begin{definition}
An \emph{automorphism} of a building is defined as a bijective morphism of chamber systems $\sigma:\Delta\to\Delta$ such that for all chambers $C,D$ in $\Delta$ $$\delta(\sigma(C),\sigma(D))=\delta(C,D)$$ Let $Aut(\Delta)$ be the group of automorphisms of $\Delta$, we say that $\Delta$ is \emph{transitive} if $Aut(\Delta)$ acts transitively on the chambers of $\Delta$, and \emph{strongly transitive} if $Aut(\Delta)$ also acts transitively on the apartments of $\Delta$.
\end{definition}

\noindent The inspiration for the theory of buildings lies in study of algebraic groups and groups of $p$-adic type. Classically, if $\mathbb{G}$ is a reductive algebraic group with irreducible Weyl group $W$, and $G=\mathbb{G}(K)$ for $K$ any field, we can construct a transitive building $\Delta(G)$ over $W$ on which $G$ acts by automorphisms. This is called the \emph{spherical building} of $G$.

When $\mathbb{G}$ has type $A_n$, the spherical building of $G=\mathbb{G}(K)$ is the \emph{flag complex} in $n$-dimensional projective space over $K$.\\

\noindent On the other hand, if $\widetilde{W}$ is the affine Weyl group associated to $\mathbb{G}$, and $K$ is a $p$-adic field, we instead want to define a building $\widetilde{\Delta}(G)$ over $\widetilde{W}$, with an action of $G$ by automorphisms. This is known as the \emph{Bruhat-Tits building} (or \emph{semisimple building}) of $G$. There are many ways of defining $\Delta=\widetilde{\Delta}(G)$, most commonly using the root datum of $\mathbb{G}$, but there are various other constructions. \\

\noindent\textbf{Example:} When $\mathbb{G}$ has type $A_n$, we can define $\widetilde{\Delta}(G)$ as the complex whose vertices are full rank $\mathcal{O}$-lattices in $K^{n+1}$ modulo scaling, where two vertices $u=[\mathcal{L}_1]$ and $v=[\mathcal{L}_2]$ are joined by an edge if $\mathcal{L}_1\supseteq\alpha\mathcal{L}_2\supseteq\pi\mathcal{L}_1$ for some $\alpha\in K$.\\

\noindent Later in this section, and in the paper, we will use this description, but we will now record some properties of general buildings that we will cite throughout.

\begin{theorem}\label{thm: minimal gallery}
Suppose $C,D$ are chambers in a strongly transitive building $\Delta$, $d(C,D)=d$, and $C=C_0\sim C_1\sim\dots\sim C_d=D$ is a minimal gallery from $C$ to $D$. Then for any apartment $\mathcal{A}$ containing $C$ and $D$, $\mathcal{A}$ contains $C_0,C_1,\dots,C_d$.
\end{theorem}

\begin{proof}

This is given by \cite[Proposition 2.3.6]{Bruhat-Tits}\end{proof}

\noindent  Now, if $\Delta=\widetilde{\Delta}(G)$ is the Bruhat-Tits building of $G$, then it contains a canonical chamber $C$ known as the \emph{hyperspecial chamber}.

\noindent\textbf{Example:} If $\mathbb{G}$ has type $A_n$, the hyperspecial chamber $C$ consists of the vertices $[\mathcal{O}^n],[\mathcal{O}^{n-1}\oplus\pi\mathcal{O}],[\mathcal{O}^{n-2}\oplus\pi\mathcal{O}^2],\dots,[\mathcal{O}\oplus\pi\mathcal{O}^{n-1}]$, and we call $v=[\mathcal{O}^n]$ the \emph{hyperspecial vertex}.\\ 

\noindent For each facet $F$ in $\Delta$, denote by $d_F$ the integer $$d_F:=\min\{d(D,C):D\text{ a chamber in }\Delta\text{ containing }F\text{ as a face}\}$$

\begin{lemma}\label{lem: chamber for facet}
Let $\Delta=\widetilde{\Delta}(G)$ be the Bruhat-Tits building for $G$, and let $F$ be a facet in $\Delta$. Then there exists a unique chamber $C(F)$ of $\Delta$, containing $F$, with $d(C(F),C)=d_F$. Moreover, 

\begin{enumerate}
\item If $\mathcal{A}$ is an apartment in $\Delta$ containing $C$ and $F$, then $\mathcal{A}$ contains $C(F)$.

\item If $g\in G$ and $g\cdot C=C$ then $C(g\cdot F)=g\cdot C(F)$.
\end{enumerate}
\end{lemma}

\begin{proof}

This follows from \cite[Lemma 1.3]{simplification}.\end{proof}

\noindent So from now on, let $\Delta=\widetilde{\Delta}(G)$ be the Bruhat-Tits building for $G$, and we deduce the following useful corollaries of Theorem \ref{thm: minimal gallery} and Lemma \ref{lem: chamber for facet}.

\begin{corollary}\label{cor: gallery fixing}
If $C,D$ are chambers in $\Delta$, then if $g\in G$ with $g\cdot C=C$ and $g\cdot D=D$, then for any minimal gallery $C=C_0,C_1,\dots,C_m=D$ from $C$ to $D$, $g\cdot C_i=C_i$ for all $i$.
\end{corollary}

\begin{proof}

We know that $C=g\cdot C=g\cdot C_0\sim\dots\sim g\cdot C_m=C_m$ is a minimal gallery from $C$ to $D$, so fixing any apartment $\mathcal{A}$ containing $C$ and $D$, we know it must contain $g\cdot C_i$ for all $i$ by Theorem \ref{thm: minimal gallery}.

Since we know that $g\cdot C_0=C_0$, we will apply induction and assume that $g\cdot C_i=C_i$ for some $i<m$. Then since $C_i\sim_{s_{i+1}} C_{i+1}$, we must have that $C_i=g\cdot C_i\sim_{s_{i+1}} g\cdot C_{i+1}$, and hence $C_i$ is adjacent to $C_{i+1}$ and $g\cdot C_{i+1}$ via the same codimension 1 face. But since $C_i,C_{i+1},g\cdot C_{i+1}$ all lie in the same apartment $\mathcal{A}$, this implies that $g\cdot C_{i+1}=C_{i+1}$ as required.\end{proof}

\begin{corollary}\label{cor: adjacent to facet}
If $C$ is the hyperspecial chamber in $\Delta$, and $F$ is a codimension 1 facet in $\Delta$, then setting $d:=d_F$:

\begin{itemize}
\item $F$ is a face of precisely one chamber of distance $d$ from $C$.

\item All chambers with $F$ as a face have distance $d$ or $d+1$ from $C$.
\end{itemize}
\end{corollary}

\begin{proof}

Using Lemma \ref{lem: chamber for facet}, we know that there exists a unique chamber $C(F)$ containing $F$ as a face such that $d(C(F),C)=d_F$, so we only need to prove that if $D$ contains $F$ and $D\neq C(F)$ then $d(D,C)=d_F+1$.

But $D$ is adjacent to $C(F)$ via $F$, so choose a minimal gallery $C=C_0\sim\dots\sim C_m=D$ from $C$ to $D$, where $C_{m-1}=C(F)$, and using Theorem \ref{thm: minimal gallery} we can choose an apartment $\mathcal{A}$ containing $C,C(F)$ and $D$. Since $d(C,C(F))=d_F$ and $D$ is adjacent to $C(F)$ in $\mathcal{A}$, we must have that $d(D,C)=d_F\pm 1$. So by minimality of $d_F$, we must have that $d(D,C)=d_F+1$.\qedhere






\end{proof}

\noindent In fact, in the case when $G$ has type $\widetilde{A}_n$ for some $n\in\mathbb{N}$, we have a stronger version of Corollary \ref{cor: adjacent to facet}.

\begin{proposition}\label{propn: number of adjacent chambers}
Suppose $G=\mathbb{G}(K)$ for some reductive algebraic group $\mathbb{G}$ of type $A_n$, and the residue field of $K$ has order $q$. Then for every codimension 1 facet $F$ in $\Delta$, and each chamber $C$ in $\Delta$, setting $d:=d_F(C)$:
\begin{itemize}
\item $F$ belongs to precisely $q+1$ chambers.

\item one of these chambers has distance $d$ from $C$, the remaining $q$ have distance $d+1$.
\end{itemize}
\end{proposition}

\begin{proof}

The second statement follows immediately from Corollary \ref{cor: adjacent to facet}, so we only need to prove the first statement.\\

\noindent Realising the vertices of $F$ as lattices in $K^{n+1}$ modulo scaling, we can write $F=\{[\mathcal{L}_1],\dots,[\mathcal{L}_{n-1}]\}$ with $$\mathcal{L}_1\supseteq\mathcal{L}_2\dots\supseteq\mathcal{L}_{n-1}\supseteq \pi\mathcal{L}_1$$ But $\mathcal{L}_1/\pi\mathcal{L}_1$ is a $\mathbb{F}_q$-vector space of dimension $n+1$, so each quotient $\mathcal{L}_i/\mathcal{L}_{i+1}$ has dimension 1 or 2, and only one can have dimension 1.

So if $D$ is a chamber of $\Delta$, and $F$ is a face of $D$, then $D=\{[\mathcal{L}_0],[\mathcal{L}_1],\dots,[\mathcal{L}_{n-1}]\}$, and we may assume that $\mathcal{L}_1\supseteq\mathcal{L}_0\supseteq\pi\mathcal{L}_1$. But we know that for each $i$, $\mathcal{L}_i\supseteq\beta_i\mathcal{L}_0\supseteq\pi\mathcal{L}_i$ for some $\beta_i\in K$, and since $\mathcal{L}_i\subseteq\mathcal{L}_0$ it follows that $\beta_i\in\mathcal{O}$.\\

\noindent If $\beta_i\in \pi\mathcal{O}$ then $\mathcal{L}_0\supseteq\beta_i^{-1}\pi\mathcal{L}_i\supseteq\mathcal{L}_i$, and if $\beta_i\in\mathcal{O}^{\times}$ then $\mathcal{L}_i\supseteq\mathcal{L}_0$. So let $j\geq 1$ be maximal such that $\mathcal{L}_j\supseteq\mathcal{L}_0$, and it follows that $\mathcal{L}_j\supseteq\mathcal{L}_0\supseteq\mathcal{L}_{j+1}$. This implies that $\mathcal{L}_j/\mathcal{L}_{j+1}$ has dimension 2 over $\mathbb{F}_q$, and $\mathcal{L}_0/\mathcal{L}_{j+1}$ has dimension 1. 

Since only a single quotient $\mathcal{L}_i/\mathcal{L}_{i+1}$ has dimension 2, $j$ does not depend on $D$, and since a 2 dimensional $\mathbb{F}_q$-vector space has only $q+1$ 1-dimensional subspaces, it follows that there are only $q+1$ chambers adjacent to $F$.\end{proof}

\subsection{Subgroups associated to facets}\label{subsec: facet subgroups}

Again, let $\mathbb{G}$ be a split semisimple, simply connected algebraic group, let $G=\mathbb{G}(K)$, and for each facet $F$ in the Bruhat-Tits building $\Delta=\widetilde{\Delta}(G)$, define the subgroup $J_F$ of $G$ by $$J_F:=\{g\in G:g\text{ fixes every vertex of }F\}$$ and note that $J_F=Stab_G(F)$ by \cite[Proposition 4.6.32]{Bruhat-Tits2}, and $J_F$ is a compact open subgroup of $G$.\\






\noindent It is proved in \cite{Tits} that for each facet $F$, there exists a connected $\mathcal{O}$-group scheme $\mathcal{G}_F$ with generic fiber $\mathbb{G}$ such that $\mathcal{G}_F(\mathcal{O})=J_F$, and the reduction $\overline{\mathcal{G}_F}$ of $\mathcal{G}_F$ modulo $\pi$ is a connected algebraic group over $\mathbb{F}_q$ with unipotent radical $N_F$. As in \cite{gorenstein} and \cite{torsion}, we define the subgroup $I_F$ of $J_F$ as $$I_F:=\{g\in \mathcal{G}_F(\mathcal{O}):\overline{g}\in N_F(\mathcal{O}/\pi\mathcal{O})\}$$

\noindent\textbf{Note:} If we assume $\mathbb{G}$ is a general split reductive algebraic group, we can still define the groups $J_F,I_F,\mathcal{G}_F$, but $J_F$, $Stab_G(F)$ and $\mathcal{G}_F^o(\mathcal{O})$ do not always coincide in general, which affects many of our subsequent results. So in this paper we will always assume semisimplicity.\\

\noindent It is clear that $I_F$ is a normal subgroup of $J_F$. Moreover, if $D$ is a chamber in $\Delta$, $v\in D$ is a vertex, and $$D=F_d\supseteq F_{d-1}\supseteq\dots\supseteq F_1\supseteq F_0=v$$ where each facet $F_i$ has dimension $i$, then $$I_v=I_{F_0}\subseteq\dots\subseteq I_{F_d}=I_D\subseteq J_D= J_{F_d}\subseteq\dots\subseteq J_{F_0}=J_v$$ If $F=C$ is the hyperspecial chamber in $\Delta$, then $J_C$ is called the \emph{Iwahori subgroup} of $G$, and we call the subgroup $I:=I_C$ the \emph{pro-$p$ Iwahori subgroup}. Note that $I_C$ is a Sylow-$p$ subgroup of $J_C$.

On the other hand, if $F=v$ is a vertex, then $J_v$ is a maximal compact open subgroup of $G$ \cite[Corollaire 3.3.3 and \S 4.4.9]{Bruhat-Tits}, and assuming $v$ is the hyperspecial vertex, we may realise $I_v$ as $\ker(\mathbb{G}(\mathcal{O})\to\mathbb{G}(\mathcal{O}/\pi\mathcal{O}))$.\\

\noindent\textbf{Note:} For each $m\in\mathbb{N}$, we similarly define the subgroup $K_m:=\ker(\mathbb{G}(\mathcal{O})\to\mathbb{G}(\mathcal{O}/\pi^m\mathcal{O}))$, a compact open subgroup of $G$ with $K_1=I_v$.\\

\noindent We define the \emph{standard apartment} in $\Delta$ to be a certain canonical apartment $\mathcal{A}_0$ in $\Delta$ that contains the hyperspecial chamber in $\Delta$. To define it explicitly, consider the BN-pair $(J,N_G(T))$, $J$ is the Iwahori and $T:=\mathbb{T}(K)$ for any torus $\mathbb{T}$ in $\mathbb{G}$. Therefore we can realise the cosets of $G/J$ as chambers in a transitive building using \cite[Theorem \rom{5}.3]{Brown}, where the defining function $\delta$ is given by $$\delta(gJ,hJ)=w\text{ where }JwJ=Jg^{-1}hJ$$ In fact this building coincides with $\widetilde{\Delta}(G)$, and the standard apartment can now be realised as the set of chambers $\{gJ:g\in N_G(T)\}$. The trivial coset $J$ is known as the \emph{hyperspecial chamber}.\\

\noindent\textbf{Example:} If $\mathbb{G}$ has type $A_n$, then $J_v=\mathbb{G}(\mathcal{O})$, $J=J_C$ is the group of matrices in $\mathbb{G}(\mathcal{O})$ that are invertible, upper-triangular modulo $\pi$, and the standard apartment $\mathcal{A}_0$ can be realised as the set vertices the form $[\langle \alpha_1 e_1,\dots,\alpha_{n+1}e_{n+1}\rangle_\mathcal{O}]$, where $\{e_1,\dots,e_{n+1}\}$ is the standard basis for $K^{n+1}$, and $\alpha_1,\dots,\alpha_{n+1}\in K$.

Furthermore, $I_v$ is the first congruence kernel $K_1:=\ker(\mathbb{G}(\mathcal{O})\to\mathbb{G}(\mathbb{F}_p))$, and $I=I_C$ is the group of matrices in $\mathbb{G}(\mathcal{O})$ that are unipotent, upper-triangular modulo $\pi$. The subgroup $K_m$ arises as the stabiliser of all vertices of distance no more than $m$ from $v$.

\begin{lemma}\label{lem: face transitivity}
Any facet $F$ in $\Delta$ is conjugate under the pro-$p$ Iwahori subgroup $I$ to a unique facet in $\mathcal{A}_0$.
\end{lemma}

\begin{proof}

See \cite[Remark 4.17(2)]{gorenstein}.\end{proof}


\begin{proposition}\label{propn: adjacent chambers facet subgroup}
For any facet $F$ in $\Delta$, 
\begin{enumerate}
    \item $I_F$ is the unique, maximal pro-$p$ normal subgroup of $J_F$.

    \item $I_F$ is equal to the set of all $g\in G$ such that 
    \begin{itemize}
    \item $g$ stabilises all chambers containing $F$.

    \item $g^{p^n}\rightarrow 1$ as $n\rightarrow\infty$.
    \end{itemize}

In particular, $S\subseteq I_F$ for any pro-$p$ subgroup $S$ of $G$ fixing all chambers containing $F$.
\end{enumerate}
\end{proposition}

\begin{proof} $ $

\begin{enumerate}
    \item We know that $I_F$ is a pro-$p$ normal subgroup of $J_F$, and $$J_F/I_F\cong(\overline{\mathcal{G}_F}/N_F)(\mathbb{F}_q)$$ So since $\overline{\mathcal{G}_F}/N_F$ is a reductive algebraic group over $\mathbb{F}_q$, it follows that $J_F/I_F$ contains no non-trivial normal $p$-subgroups.

    \item Since $I_F$ is a pro-$p$ subgroup of $G$, it is clear that $g^{p^n}\rightarrow 1$ for all $g\in I_F$. On the other hand, for any chamber $C$ of $\Delta$ containing $F$ as a face, we know that $I_F\subseteq I_C\subseteq J_C$, so clearly every element of $I_F$ stabilises $C$.

\noindent Conversely, assume $g$ fixes every chamber adjacent to $F$, and that $g^{p^n}\rightarrow 1$ as $n\rightarrow\infty$. Then for any chamber $C$ containing $F$, it is clear that $g\in J_C$. But since $I_C$ is an open subgroup of $J_C$, it follows that $g$ maps to a $p$-torsion element of $J_C/I_C\cong (\overline{\mathcal{G}_C}/N_C)(\mathbb{F}_q)$. But $\overline{\mathcal{G}_C}/N_C$ is a split torus, so $J_F/I_F$ can contain no non-trivial $p$-torsion elements, and thus $g\in I_C$.

So if $\{C_1,\dots,C_r\}$ is the set of all chambers containing $F$ as a face, then $g\in N:= I_{C_1}\cap\cdots\cap I_{C_r}$. But since $J_F$ permutes $\{C_1,\dots,C_r\}$, it follows that $N$ is a normal subgroup of $J_F$. and clearly it is a pro-$p$ subgroup, so it follows from part 1 that $N\subseteq I_F$, and hence $g\in I_F$ as required.\qedhere
\end{enumerate}

\end{proof}

\subsection{Cycles and Summits in the building}\label{subsec: summits}

The geometric and combinatorial structure of the Bruhat-Tits building of $G$ is well understood when $G$ has rank $1$, in which case the building is a tree. In higher ranks, this is very false. Indeed, the building is constructed from higher dimensional simplices, all of which contain cycles.

In this section, we will move towards understanding the local structure of the $\widetilde{A}_2$-building, and prove several technical results which will be required in the proof of our main results.\\

\noindent Let $\Delta=\widetilde{\Delta}(G)$ be the Bruhat-Tits building of $G$, and let $C$ be the hyperspecial chamber. A \emph{cycle} in $\Delta$ is defined to be a gallery $$D=D_0\sim D_1\sim\dots\sim D_m=D$$ where $D_i\neq D$ for all $0<i<m$. We call $m$ the \emph{length} of the cycle. For example, there are no cycles of any length in the rank 1 Bruhat-Tits tree.

\begin{lemma}\label{lem: no-cycles}
The $\widetilde{A}_2$-building $\Delta$ contains no cycles of length less than 6.
\end{lemma}

\begin{proof}
Clearly there can be no cycles in $\Delta$ of length 1 or 2, and a cycle of length 3 would constitute a 3-simplex, which cannot exist in the $\widetilde{A}_2$-building. Thus all cycles have length at least 4.\\

\noindent If $D=D_0\sim D_1\sim D_2\sim D_3\sim D_4=D$ is a cycle of length 4, then $d(D,D_2)=2$, because if $d(D,D_2)=0$ or 1, this would give a cycle of length 2 or 3. So choose an apartment $\mathcal{A}$ containing $D$ and $D_2$, and by Theorem \ref{thm: minimal gallery}, $\mathcal{A}$ must contain $D_1,D_2,D_3$, so this is a cycle of length 4 in the apartment, i.e. in the $\widetilde{A}_2$-Coxeter complex, which is impossible.\\

\noindent Therefore, suppose $D=D_0\sim D_1\sim D_2\sim D_3\sim D_4\sim D_5=D$ is a cycle of length 5, and since it is a cycle of minimal possible length, we must have that $d(D,D_1)=d(D,D_4)=1$ and $d(D,D_2)=d(D,D_3)=2$. In particular, $D_2$ and $D_3$ are not adjacent to $D$.\\ 

\noindent Note that $D_1$ and $D_2$ share a vertex $v_1$ in common with $D$, and similarly $D_3,D_4$ and $D$ have a common vertex $v_2$. If $v_1\neq v_2$, then these vertices are joined by an edge (in $D$), and this cannot be an edge of $D_2$ or $D_3$, since they are not adjacent to $D$ by assumption. Thus the vertices of $D_2$ and $D_3$ form a 3-simplex, which again is impossible. 

Therefore, $v_1=v_2=:v$ is a common edge shared by all chambers in the cycle, as illustrated below:\\

\begin{center}

\tikzset{every picture/.style={line width=0.75pt}} 

\begin{tikzpicture}[x=0.75pt,y=0.75pt,yscale=-1,xscale=1]

\draw  [color={rgb, 255:red, 0; green, 0; blue, 0 }  ,draw opacity=1 ] (290.95,310.98) -- (182.01,181.64) -- (316.71,190.37) -- cycle ;
\draw  [color={rgb, 255:red, 0; green, 0; blue, 0 }  ,draw opacity=1 ] (455,149.18) -- (415.29,272.87) -- (317.15,189.85) -- cycle ;
\draw  [color={rgb, 255:red, 0; green, 0; blue, 0 }  ,draw opacity=1 ] (258.64,78) -- (316.73,190.38) -- (181.03,181.53) -- cycle ;
\draw  [color={rgb, 255:red, 0; green, 0; blue, 0 }  ,draw opacity=1 ] (415.28,272.88) -- (290.76,310.94) -- (316.74,190.36) -- cycle ;
\draw    (258.63,78) -- (455,149.17) ;

\draw (325.94,239.7) node [anchor=north west][inner sep=0.75pt]   [align=left] {$\displaystyle D$};
\draw (255.65,222.7) node [anchor=north west][inner sep=0.75pt]   [align=left] {$\displaystyle D_{1}$};
\draw (245.34,133.48) node [anchor=north west][inner sep=0.75pt]   [align=left] {$\displaystyle D_{2}$};
\draw (323.41,136.31) node [anchor=north west][inner sep=0.75pt]   [align=left] {$\displaystyle D_{3}$};
\draw (389.7,192.96) node [anchor=north west][inner sep=0.75pt]   [align=left] {$\displaystyle D_{4}$};
\draw (316,198) node [anchor=north west][inner sep=0.75pt]   [align=left] {$\displaystyle v$};
\draw (345,268) node [anchor=north west][inner sep=0.75pt]   [align=left] {$\displaystyle e_{0}$};
\draw (291,119) node [anchor=north west][inner sep=0.75pt]   [align=left] {$\displaystyle e$};
\draw (304,246) node [anchor=north west][inner sep=0.75pt]   [align=left] {$\displaystyle e_{1}$};
\draw (355,228) node [anchor=north west][inner sep=0.75pt]   [align=left] {$\displaystyle e_{4}$};

\end{tikzpicture}

\end{center}

\noindent Let $e$ be the edge joining $D_2$ and $D_3$, and let $D'$ be a chamber adjacent to $e$ with $D'\neq D_2$ or $D_3$. Then $d(D,D')\leq 3$, and if $d(D,D')=3$ then $D\sim D_1\sim D_2\sim D'$ and $D\sim D_4\sim D_3\sim D'$ are minimal galleries, so fix any apartment $\mathcal{A}$ containing $D$ and $D'$, and it follows from Theorem \ref{thm: minimal gallery} that $\mathcal{A}$ contains $D,D_1,\dots,D_4$, so it contains a 5-cycle, which is impossible in the $\widetilde{A}_2$-Coxeter complex.\\

\noindent On the other hand, if $d(D,D')=0$ or 1, then this gives a 3 or 4-cycle, so we may assume that $d(D,D')=2$, so there exists a chamber $E$ such that $D\sim E\sim D'$. Let $e_1$ (resp. $e_4$) be the edge of $D$ adjacent to $D_1$ (resp. $D_4$). If $E$ is adjacent to $D$ via $e_1$ or $e_4$, then this gives a 4-cycle $E\sim D'\sim D_2\sim D_1\sim E$ or $E\sim D'\sim D_3\sim D_4\sim E$, which is impossible. So $E$ must be adjacent to $D$ via its third edge $e_0$.

\begin{center}

\tikzset{every picture/.style={line width=0.75pt}} 

\begin{tikzpicture}[x=0.75pt,y=0.75pt,yscale=-1,xscale=1]

\draw  [color={rgb, 255:red, 0; green, 0; blue, 0 }  ,draw opacity=1 ] (177.96,271.27) -- (90.82,169.67) -- (198.56,176.53) -- cycle ;
\draw  [color={rgb, 255:red, 0; green, 0; blue, 0 }  ,draw opacity=1 ] (309.2,144.17) -- (277.43,241.34) -- (198.92,176.13) -- cycle ;
\draw  [color={rgb, 255:red, 0; green, 0; blue, 0 }  ,draw opacity=1 ] (152.11,88.26) -- (198.59,176.54) -- (90.04,169.59) -- cycle ;
\draw  [color={rgb, 255:red, 0; green, 0; blue, 0 }  ,draw opacity=1 ] (277.43,241.34) -- (177.81,271.25) -- (198.59,176.52) -- cycle ;
\draw    (152.11,88.27) -- (309.2,144.17) ;
\draw   (249.57,331.27) -- (178.04,270.62) -- (278.3,241.09) -- cycle ;
\draw   (241.75,60.99) -- (198.76,176.5) -- (152.63,88.07) -- cycle ;
\draw  [dash pattern={on 0.84pt off 2.51pt}]  (241.75,60.99) -- (249.58,331.27) ;
\draw  [dash pattern={on 0.84pt off 2.51pt}]  (152.63,88.07) -- (177.96,271.27) ;
\draw  [color={rgb, 255:red, 0; green, 0; blue, 0 }  ,draw opacity=1 ] (446.76,266.56) -- (359.61,164.96) -- (467.36,171.81) -- cycle ;
\draw  [color={rgb, 255:red, 0; green, 0; blue, 0 }  ,draw opacity=1 ] (578,139.46) -- (546.23,236.62) -- (467.72,171.42) -- cycle ;
\draw  [color={rgb, 255:red, 0; green, 0; blue, 0 }  ,draw opacity=1 ] (420.91,83.55) -- (467.38,171.82) -- (358.83,164.88) -- cycle ;
\draw  [color={rgb, 255:red, 0; green, 0; blue, 0 }  ,draw opacity=1 ] (546.23,236.63) -- (446.61,266.53) -- (467.38,171.81) -- cycle ;
\draw    (420.91,83.55) -- (578,139.46) ;
\draw   (518.37,326.56) -- (446.84,265.9) -- (547.09,236.38) -- cycle ;
\draw   (510.55,56.28) -- (467.55,171.79) -- (421.42,83.36) -- cycle ;
\draw  [dash pattern={on 0.84pt off 2.51pt}]  (510.54,56.28) -- (446.84,265.9) ;
\draw  [dash pattern={on 0.84pt off 2.51pt}]  (421.43,83.36) -- (518.37,326.56) ;

\draw (204.45,213.35) node [anchor=north west][inner sep=0.75pt]   [align=left] {$\displaystyle D$};
\draw (147.53,200) node [anchor=north west][inner sep=0.75pt]   [align=left] {$\displaystyle D_{1}$};
\draw (139.28,129.91) node [anchor=north west][inner sep=0.75pt]   [align=left] {$\displaystyle D_{2}$};
\draw (201.73,132.14) node [anchor=north west][inner sep=0.75pt]   [align=left] {$\displaystyle D_{3}$};
\draw (254.76,176.64) node [anchor=north west][inner sep=0.75pt]   [align=left] {$\displaystyle D_{4}$};
\draw (197.01,180.6) node [anchor=north west][inner sep=0.75pt]   [align=left] {$\displaystyle v$};
\draw (227.95,270.69) node [anchor=north west][inner sep=0.75pt]   [align=left] {$\displaystyle E$};
\draw (219.6,235.58) node [anchor=north west][inner sep=0.75pt]   [align=left] {$\displaystyle e_{0}$};
\draw (177.01,118.54) node [anchor=north west][inner sep=0.75pt]   [align=left] {$\displaystyle e$};
\draw (186.81,218.3) node [anchor=north west][inner sep=0.75pt]   [align=left] {$\displaystyle e_{1}$};
\draw (227.6,204.16) node [anchor=north west][inner sep=0.75pt]   [align=left] {$\displaystyle e_{4}$};
\draw (205.6,76.12) node [anchor=north west][inner sep=0.75pt]   [align=left] {$\displaystyle D'$};
\draw (244.1,329.84) node [anchor=north west][inner sep=0.75pt]   [align=left] {$\displaystyle u$};
\draw (473.25,208.64) node [anchor=north west][inner sep=0.75pt]   [align=left] {$\displaystyle D$};
\draw (416.32,195.29) node [anchor=north west][inner sep=0.75pt]   [align=left] {$\displaystyle D_{1}$};
\draw (408.07,125.2) node [anchor=north west][inner sep=0.75pt]   [align=left] {$\displaystyle D_{2}$};
\draw (470.53,127.43) node [anchor=north west][inner sep=0.75pt]   [align=left] {$\displaystyle D_{3}$};
\draw (523.56,171.92) node [anchor=north west][inner sep=0.75pt]   [align=left] {$\displaystyle D_{4}$};
\draw (465.8,175.88) node [anchor=north west][inner sep=0.75pt]   [align=left] {$\displaystyle v$};
\draw (496.75,265.98) node [anchor=north west][inner sep=0.75pt]   [align=left] {$\displaystyle E$};
\draw (488.4,230.87) node [anchor=north west][inner sep=0.75pt]   [align=left] {$\displaystyle e_{0}$};
\draw (445.8,113.83) node [anchor=north west][inner sep=0.75pt]   [align=left] {$\displaystyle e$};
\draw (455.6,213.59) node [anchor=north west][inner sep=0.75pt]   [align=left] {$\displaystyle e_{1}$};
\draw (496.4,199.45) node [anchor=north west][inner sep=0.75pt]   [align=left] {$\displaystyle e_{4}$};
\draw (474.4,71.41) node [anchor=north west][inner sep=0.75pt]   [align=left] {$\displaystyle D'$};
\draw (512.9,325.13) node [anchor=north west][inner sep=0.75pt]   [align=left] {$\displaystyle u$};

\end{tikzpicture}

\end{center}

\noindent In particular, $E$ does not contain $v$, so $E$ is adjacent to $D'$ via the edge of not containing $v$, and thus $D'$ contains an edge connecting $v$ to the edge of $E$ outside $e_0$ (as shown in the diagram, where the dotted line indicates that we identify the vertices). This results in a 3-simplex consisting of the vertices of $D$ and $E$, or else a 3-cycle $D_1\sim D_2\sim D'\sim D_1$, a contradiction in both cases.\end{proof}

\noindent\textbf{Remark:} There are cycles of higher length in the $\widetilde{A}_2$-Bruhat-Tits building.  Indeed, every apartment is composed of hexagonal arrangements of chambers, which are 6-cycles, and these are the only examples of 6-cycles in the building. However, as we will see later, there are examples of cycles of higher length that are not contained in apartments.\\

\noindent Now, for each $n\in\mathbb{N}$, define the following set of vertices in $\Delta$:  $$\Delta_n:=\left\{v\in V(\Delta):v\in D\text{ for some chamber }D\text{ of }\Delta\text{ with }d(C,D)\leq n\right\}$$

\noindent\textbf{Note:} 1. We define this region as a set of vertices, rather than chambers, since we can find chambers $D$ such that $d(D,C)>n$ but all vertices of $D$ lie in $\Delta_n$.\\

\noindent 2. For convenience, we let $\Delta_{-1}:=\varnothing$, so we may always refer to $\Delta_{n-1}$ for any $n\in\mathbb{N}$.\\

\noindent If $\Delta$ is the $\widetilde{A}_1$-building, i.e. the infinite tree where every vertex has degree $q+1$, then we can realise the regions $\Delta_n$ explicitly for any $n$, as illustrated below when $q=2$:

\begin{center}

\tikzset{every picture/.style={line width=0.75pt}} 

\begin{tikzpicture}[x=0.75pt,y=0.75pt,yscale=-1,xscale=1]

\draw    (254.55,190.6) -- (315.64,191.2) ;
\draw  [fill={rgb, 255:red, 0; green, 0; blue, 0 }  ,fill opacity=1 ] (254.55,190.6) .. controls (254.55,189.76) and (253.94,189.08) .. (253.19,189.08) .. controls (252.44,189.08) and (251.83,189.76) .. (251.83,190.6) .. controls (251.83,191.43) and (252.44,192.11) .. (253.19,192.11) .. controls (253.94,192.11) and (254.55,191.43) .. (254.55,190.6) -- cycle ;
\draw  [fill={rgb, 255:red, 0; green, 0; blue, 0 }  ,fill opacity=1 ] (318.36,191.2) .. controls (318.36,190.36) and (317.75,189.69) .. (317,189.69) .. controls (316.25,189.69) and (315.64,190.36) .. (315.64,191.2) .. controls (315.64,192.04) and (316.25,192.72) .. (317,192.72) .. controls (317.75,192.72) and (318.36,192.04) .. (318.36,191.2) -- cycle ;
\draw    (211.1,152.78) -- (253.19,190.6) ;
\draw    (253.19,190.6) -- (214.72,228.61) ;
\draw    (358.18,154.7) -- (318.36,191.2) ;
\draw    (318.36,191.2) -- (360.45,229.01) ;
\draw  [fill={rgb, 255:red, 0; green, 0; blue, 0 }  ,fill opacity=1 ] (361.81,229.01) .. controls (361.81,228.18) and (361.2,227.5) .. (360.45,227.5) .. controls (359.7,227.5) and (359.09,228.18) .. (359.09,229.01) .. controls (359.09,229.85) and (359.7,230.53) .. (360.45,230.53) .. controls (361.2,230.53) and (361.81,229.85) .. (361.81,229.01) -- cycle ;
\draw  [fill={rgb, 255:red, 0; green, 0; blue, 0 }  ,fill opacity=1 ] (359.54,154.7) .. controls (359.54,153.87) and (358.93,153.19) .. (358.18,153.19) .. controls (357.43,153.19) and (356.83,153.87) .. (356.83,154.7) .. controls (356.83,155.54) and (357.43,156.22) .. (358.18,156.22) .. controls (358.93,156.22) and (359.54,155.54) .. (359.54,154.7) -- cycle ;
\draw  [fill={rgb, 255:red, 0; green, 0; blue, 0 }  ,fill opacity=1 ] (216.08,228.61) .. controls (216.08,227.77) and (215.47,227.09) .. (214.72,227.09) .. controls (213.97,227.09) and (213.36,227.77) .. (213.36,228.61) .. controls (213.36,229.45) and (213.97,230.13) .. (214.72,230.13) .. controls (215.47,230.13) and (216.08,229.45) .. (216.08,228.61) -- cycle ;
\draw  [fill={rgb, 255:red, 0; green, 0; blue, 0 }  ,fill opacity=1 ] (212.46,152.78) .. controls (212.46,151.95) and (211.85,151.27) .. (211.1,151.27) .. controls (210.35,151.27) and (209.74,151.95) .. (209.74,152.78) .. controls (209.74,153.62) and (210.35,154.3) .. (211.1,154.3) .. controls (211.85,154.3) and (212.46,153.62) .. (212.46,152.78) -- cycle ;
\draw    (210.16,106.87) -- (211.1,152.78) ;
\draw    (211.1,152.78) -- (166.88,132.55) ;
\draw    (214.72,228.61) -- (161.68,242.09) ;
\draw    (214.72,228.61) -- (205.96,272.76) ;
\draw    (363.39,107.74) -- (358.18,154.7) ;
\draw    (402.69,134.04) -- (358.18,154.7) ;
\draw    (418.49,242.89) -- (360.45,229.01) ;
\draw    (360.45,229.01) -- (378.5,273.19) ;
\draw  [fill={rgb, 255:red, 0; green, 0; blue, 0 }  ,fill opacity=1 ] (167.04,132.41) .. controls (167.04,131.58) and (166.43,130.9) .. (165.68,130.9) .. controls (164.93,130.9) and (164.33,131.58) .. (164.33,132.41) .. controls (164.33,133.25) and (164.93,133.93) .. (165.68,133.93) .. controls (166.43,133.93) and (167.04,133.25) .. (167.04,132.41) -- cycle ;
\draw  [fill={rgb, 255:red, 0; green, 0; blue, 0 }  ,fill opacity=1 ] (211.52,106.87) .. controls (211.52,106.03) and (210.91,105.35) .. (210.16,105.35) .. controls (209.41,105.35) and (208.8,106.03) .. (208.8,106.87) .. controls (208.8,107.71) and (209.41,108.39) .. (210.16,108.39) .. controls (210.91,108.39) and (211.52,107.71) .. (211.52,106.87) -- cycle ;
\draw  [fill={rgb, 255:red, 0; green, 0; blue, 0 }  ,fill opacity=1 ] (364.74,107.74) .. controls (364.74,106.9) and (364.14,106.22) .. (363.39,106.22) .. controls (362.64,106.22) and (362.03,106.9) .. (362.03,107.74) .. controls (362.03,108.57) and (362.64,109.25) .. (363.39,109.25) .. controls (364.14,109.25) and (364.74,108.57) .. (364.74,107.74) -- cycle ;
\draw  [fill={rgb, 255:red, 0; green, 0; blue, 0 }  ,fill opacity=1 ] (404.71,133.75) .. controls (404.71,132.91) and (404.1,132.23) .. (403.35,132.23) .. controls (402.6,132.23) and (401.99,132.91) .. (401.99,133.75) .. controls (401.99,134.58) and (402.6,135.26) .. (403.35,135.26) .. controls (404.1,135.26) and (404.71,134.58) .. (404.71,133.75) -- cycle ;
\draw  [fill={rgb, 255:red, 0; green, 0; blue, 0 }  ,fill opacity=1 ] (419.56,242.18) .. controls (419.56,241.34) and (418.95,240.66) .. (418.2,240.66) .. controls (417.45,240.66) and (416.84,241.34) .. (416.84,242.18) .. controls (416.84,243.01) and (417.45,243.69) .. (418.2,243.69) .. controls (418.95,243.69) and (419.56,243.01) .. (419.56,242.18) -- cycle ;
\draw  [fill={rgb, 255:red, 0; green, 0; blue, 0 }  ,fill opacity=1 ] (379.86,273.19) .. controls (379.86,272.36) and (379.25,271.68) .. (378.5,271.68) .. controls (377.75,271.68) and (377.14,272.36) .. (377.14,273.19) .. controls (377.14,274.03) and (377.75,274.71) .. (378.5,274.71) .. controls (379.25,274.71) and (379.86,274.03) .. (379.86,273.19) -- cycle ;
\draw  [fill={rgb, 255:red, 0; green, 0; blue, 0 }  ,fill opacity=1 ] (161.68,242.09) .. controls (161.68,241.25) and (161.08,240.58) .. (160.33,240.58) .. controls (159.58,240.58) and (158.97,241.25) .. (158.97,242.09) .. controls (158.97,242.93) and (159.58,243.61) .. (160.33,243.61) .. controls (161.08,243.61) and (161.68,242.93) .. (161.68,242.09) -- cycle ;
\draw  [fill={rgb, 255:red, 0; green, 0; blue, 0 }  ,fill opacity=1 ] (207.32,274.28) .. controls (207.32,273.44) and (206.71,272.76) .. (205.96,272.76) .. controls (205.21,272.76) and (204.61,273.44) .. (204.61,274.28) .. controls (204.61,275.12) and (205.21,275.79) .. (205.96,275.79) .. controls (206.71,275.79) and (207.32,275.12) .. (207.32,274.28) -- cycle ;
\draw    (223.49,71.86) -- (209.86,108.2) ;
\draw  [fill={rgb, 255:red, 0; green, 0; blue, 0 }  ,fill opacity=1 ] (224.62,72.27) .. controls (224.89,71.61) and (224.59,70.89) .. (223.97,70.67) .. controls (223.34,70.44) and (222.63,70.79) .. (222.36,71.45) .. controls (222.1,72.11) and (222.39,72.83) .. (223.02,73.05) .. controls (223.64,73.27) and (224.36,72.92) .. (224.62,72.27) -- cycle ;
\draw    (187.58,75.9) -- (210.92,106.85) ;
\draw  [fill={rgb, 255:red, 0; green, 0; blue, 0 }  ,fill opacity=1 ] (188.53,75.15) .. controls (188.11,74.58) and (187.35,74.44) .. (186.83,74.86) .. controls (186.31,75.27) and (186.22,76.07) .. (186.64,76.64) .. controls (187.05,77.22) and (187.81,77.35) .. (188.33,76.94) .. controls (188.85,76.52) and (188.94,75.72) .. (188.53,75.15) -- cycle ;
\draw    (150.43,96.73) -- (165.68,132.41) ;
\draw  [fill={rgb, 255:red, 0; green, 0; blue, 0 }  ,fill opacity=1 ] (151.52,96.23) .. controls (151.26,95.58) and (150.56,95.27) .. (149.95,95.54) .. controls (149.35,95.81) and (149.07,96.57) .. (149.34,97.23) .. controls (149.6,97.88) and (150.3,98.19) .. (150.91,97.92) .. controls (151.51,97.65) and (151.79,96.89) .. (151.52,96.23) -- cycle ;
\draw    (128.99,140.03) -- (166.88,132.55) ;
\draw  [fill={rgb, 255:red, 0; green, 0; blue, 0 }  ,fill opacity=1 ] (128.73,138.85) .. controls (128.04,138.97) and (127.59,139.6) .. (127.73,140.25) .. controls (127.87,140.9) and (128.54,141.32) .. (129.24,141.2) .. controls (129.93,141.08) and (130.38,140.46) .. (130.24,139.81) .. controls (130.1,139.16) and (129.43,138.73) .. (128.73,138.85) -- cycle ;
\draw    (121.78,239.77) -- (160.33,242.09) ;
\draw  [fill={rgb, 255:red, 0; green, 0; blue, 0 }  ,fill opacity=1 ] (121.83,238.57) .. controls (121.12,238.51) and (120.53,239) .. (120.5,239.67) .. controls (120.47,240.33) and (121.02,240.92) .. (121.72,240.98) .. controls (122.43,241.03) and (123.02,240.54) .. (123.05,239.88) .. controls (123.08,239.21) and (122.53,238.62) .. (121.83,238.57) -- cycle ;
\draw    (138.22,279.53) -- (160.08,241.44) ;
\draw  [fill={rgb, 255:red, 0; green, 0; blue, 0 }  ,fill opacity=1 ] (139.58,279.53) .. controls (139.58,278.69) and (138.97,278.01) .. (138.22,278.01) .. controls (137.47,278.01) and (136.86,278.69) .. (136.86,279.53) .. controls (136.86,280.37) and (137.47,281.05) .. (138.22,281.05) .. controls (138.97,281.05) and (139.58,280.37) .. (139.58,279.53) -- cycle ;
\draw    (176.33,306.14) -- (206.32,274.13) ;
\draw  [fill={rgb, 255:red, 0; green, 0; blue, 0 }  ,fill opacity=1 ] (177.65,306.45) .. controls (177.84,305.64) and (177.41,304.83) .. (176.68,304.66) .. controls (175.95,304.49) and (175.2,305.01) .. (175.01,305.82) .. controls (174.82,306.64) and (175.25,307.44) .. (175.98,307.61) .. controls (176.71,307.79) and (177.46,307.27) .. (177.65,306.45) -- cycle ;
\draw    (231.23,307.21) -- (206.45,274.98) ;
\draw  [fill={rgb, 255:red, 0; green, 0; blue, 0 }  ,fill opacity=1 ] (231.55,305.89) .. controls (230.89,305.65) and (230.21,306.05) .. (230.03,306.78) .. controls (229.86,307.51) and (230.25,308.29) .. (230.91,308.53) .. controls (231.56,308.77) and (232.24,308.38) .. (232.42,307.65) .. controls (232.59,306.92) and (232.2,306.13) .. (231.55,305.89) -- cycle ;
\draw    (442.98,274.4) -- (418.2,242.18) ;
\draw  [fill={rgb, 255:red, 0; green, 0; blue, 0 }  ,fill opacity=1 ] (443.3,273.08) .. controls (442.64,272.84) and (441.96,273.24) .. (441.79,273.97) .. controls (441.61,274.7) and (442,275.48) .. (442.66,275.72) .. controls (443.32,275.97) and (443.99,275.57) .. (444.17,274.84) .. controls (444.35,274.11) and (443.96,273.32) .. (443.3,273.08) -- cycle ;
\draw    (366.47,312.08) -- (378.03,273.03) ;
\draw  [fill={rgb, 255:red, 0; green, 0; blue, 0 }  ,fill opacity=1 ] (367.72,311.57) .. controls (367.54,310.89) and (366.82,310.57) .. (366.13,310.85) .. controls (365.44,311.13) and (365.03,311.91) .. (365.22,312.59) .. controls (365.41,313.26) and (366.12,313.59) .. (366.82,313.3) .. controls (367.51,313.02) and (367.91,312.24) .. (367.72,311.57) -- cycle ;
\draw    (403.72,306.1) -- (378.94,273.87) ;
\draw  [fill={rgb, 255:red, 0; green, 0; blue, 0 }  ,fill opacity=1 ] (404.04,304.78) .. controls (403.39,304.54) and (402.71,304.93) .. (402.53,305.66) .. controls (402.36,306.39) and (402.75,307.18) .. (403.41,307.42) .. controls (404.06,307.66) and (404.74,307.27) .. (404.92,306.54) .. controls (405.09,305.81) and (404.7,305.02) .. (404.04,304.78) -- cycle ;
\draw    (457.61,232.33) -- (418.49,242.89) ;
\draw  [fill={rgb, 255:red, 0; green, 0; blue, 0 }  ,fill opacity=1 ] (456.52,231.53) .. controls (456.05,232.05) and (456.15,232.83) .. (456.75,233.27) .. controls (457.35,233.72) and (458.23,233.66) .. (458.7,233.14) .. controls (459.18,232.62) and (459.07,231.84) .. (458.47,231.4) .. controls (457.87,230.95) and (457,231.01) .. (456.52,231.53) -- cycle ;
\draw    (443.78,136.29) -- (403.35,133.75) ;
\draw  [fill={rgb, 255:red, 0; green, 0; blue, 0 }  ,fill opacity=1 ] (443,135.17) .. controls (442.39,135.51) and (442.24,136.28) .. (442.67,136.9) .. controls (443.1,137.51) and (443.94,137.74) .. (444.56,137.4) .. controls (445.17,137.06) and (445.32,136.29) .. (444.89,135.67) .. controls (444.46,135.06) and (443.62,134.84) .. (443,135.17) -- cycle ;
\draw    (424.55,99.74) -- (402.69,134.04) ;
\draw  [fill={rgb, 255:red, 0; green, 0; blue, 0 }  ,fill opacity=1 ] (423.21,99.88) .. controls (423.2,100.59) and (423.8,101.09) .. (424.54,101.02) .. controls (425.29,100.94) and (425.89,100.3) .. (425.9,99.6) .. controls (425.9,98.9) and (425.3,98.39) .. (424.56,98.47) .. controls (423.82,98.54) and (423.21,99.18) .. (423.21,99.88) -- cycle ;
\draw    (391.46,77.53) -- (363.62,107.12) ;
\draw  [fill={rgb, 255:red, 0; green, 0; blue, 0 }  ,fill opacity=1 ] (390.12,77.42) .. controls (389.98,78.11) and (390.47,78.72) .. (391.22,78.78) .. controls (391.96,78.84) and (392.67,78.33) .. (392.81,77.64) .. controls (392.95,76.95) and (392.45,76.34) .. (391.71,76.28) .. controls (390.96,76.21) and (390.25,76.72) .. (390.12,77.42) -- cycle ;
\draw    (341.25,73.51) -- (363.04,107.86) ;
\draw  [fill={rgb, 255:red, 0; green, 0; blue, 0 }  ,fill opacity=1 ] (340.81,74.8) .. controls (341.44,75.1) and (342.15,74.77) .. (342.39,74.06) .. controls (342.64,73.35) and (342.32,72.53) .. (341.68,72.23) .. controls (341.05,71.93) and (340.34,72.26) .. (340.1,72.97) .. controls (339.86,73.68) and (340.18,74.5) .. (340.81,74.8) -- cycle ;
\draw  [dash pattern={on 0.84pt off 2.51pt}] (253.25,190.54) .. controls (252.82,180) and (266.73,171.46) .. (284.31,171.46) .. controls (301.89,171.46) and (316.48,180) .. (316.91,190.54) .. controls (317.34,201.08) and (303.43,209.62) .. (285.85,209.62) .. controls (268.27,209.62) and (253.68,201.08) .. (253.25,190.54) -- cycle ;
\draw  [dash pattern={on 0.84pt off 2.51pt}] (183.26,190.9) .. controls (182.05,160.96) and (226.73,136.68) .. (283.06,136.68) .. controls (339.39,136.68) and (386.03,160.96) .. (387.25,190.9) .. controls (388.46,220.84) and (343.78,245.12) .. (287.45,245.12) .. controls (231.12,245.12) and (184.48,220.84) .. (183.26,190.9) -- cycle ;
\draw  [dash pattern={on 0.84pt off 2.51pt}] (137.87,192.66) .. controls (135.65,137.84) and (200.15,93.4) .. (281.93,93.4) .. controls (363.72,93.4) and (431.82,137.84) .. (434.04,192.66) .. controls (436.26,247.48) and (371.76,291.92) .. (289.98,291.92) .. controls (208.19,291.92) and (140.09,247.48) .. (137.87,192.66) -- cycle ;

\draw (277.65,174.71) node [anchor=north west][inner sep=0.75pt]   [align=left] {$\displaystyle C$};
\draw (126.49,320.14) node [anchor=north west][inner sep=0.75pt]   [align=left] {\textbf{Figure 1: }The region $\displaystyle \Delta _{3}$ of the Bruhat-Tits tree };

\end{tikzpicture}

\end{center} 

\noindent In this paper, however, we are particularly interested in the case where $\Delta$ is the $\widetilde{A}_2$-building. Techniques were developed for visualising $\Delta$ in \cite{gallery}, and using this visualisation we can construct an image of $\Delta_n$ in this building, as illustrated below in Figure 2, an image available in \cite{link} (again when $q=2$).\\

\begin{figure}[h]
\centering
{\includegraphics[scale=.8]{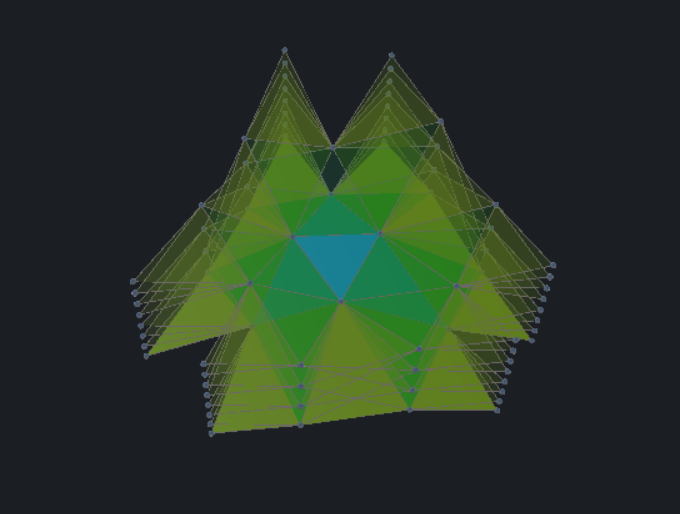}}\\
\textbf{Figure 2:} The region $\Delta_3$ of the $\widetilde{A}_2$-Bruhat-Tits building, where $C$ is the blue chamber.
\end{figure}

\noindent The apartments in $\Delta$ containing $C$ are very visible in this image, since they are all isomorphic to the $\widetilde{A}_2$ Coxeter complex, which is a tiling of the Euclidean plane by 2-simplices. The chambers of this complex are in bijection with elements of the affine Coxeter group $\widetilde{W}$; which in type $\widetilde{A}_2$ we can realise as $$\widetilde{W}:=\langle s_0,s_1,s_2|s_0^2=s_1^2=s_2^2=(s_0s_1)^3=(s_0s_2)^3=(s_1s_2)^3\rangle$$

\begin{center}

\tikzset{every picture/.style={line width=0.75pt}} 

\begin{tikzpicture}[x=0.75pt,y=0.75pt,yscale=-1,xscale=1]

\draw  [color={rgb, 255:red, 0; green, 0; blue, 0 }  ,draw opacity=1 ] (281.9,211.87) -- (346.13,310.27) -- (219.71,311.57) -- cycle ;
\draw  [color={rgb, 255:red, 0; green, 0; blue, 0 }  ,draw opacity=1 ] (150.68,212.89) -- (281.8,211.49) -- (219.34,311.02) -- cycle ;
\draw  [color={rgb, 255:red, 0; green, 0; blue, 0 }  ,draw opacity=1 ] (409.09,210.18) -- (346.03,310.35) -- (281.79,211.47) -- cycle ;
\draw  [color={rgb, 255:red, 0; green, 0; blue, 0 }  ,draw opacity=1 ] (215.22,113.16) -- (281.8,211.49) -- (150.67,212.84) -- cycle ;
\draw  [color={rgb, 255:red, 0; green, 0; blue, 0 }  ,draw opacity=1 ] (344.42,111.83) -- (409.09,210.17) -- (281.8,211.48) -- cycle ;
\draw    (215.22,113.16) -- (344.42,111.83) ;
\draw  [color={rgb, 255:red, 0; green, 0; blue, 0 }  ,draw opacity=1 ] (471.72,110.52) -- (536.38,208.86) -- (409.09,210.17) -- cycle ;
\draw    (342.51,111.85) -- (471.72,110.52) ;
\draw  [color={rgb, 255:red, 0; green, 0; blue, 0 }  ,draw opacity=1 ] (409.09,210.17) -- (473.32,308.58) -- (346.9,309.88) -- cycle ;
\draw  [color={rgb, 255:red, 0; green, 0; blue, 0 }  ,draw opacity=1 ] (536.39,208.86) -- (473.32,309.03) -- (409.08,210.16) -- cycle ;
\draw  [color={rgb, 255:red, 0; green, 0; blue, 0 }  ,draw opacity=1 ] (219.71,311.57) -- (346.64,309.46) -- (285.68,406.6) -- cycle ;
\draw    (285.69,406.6) -- (412.88,404.91) ;
\draw    (346.03,310.35) -- (412.88,404.91) ;
\draw    (412.88,404.91) -- (473.32,309.03) ;
\draw    (194.06,82.53) -- (271.23,195.74) ;
\draw    (131.12,185.18) -- (175.77,248.73) ;
\draw    (355.56,323.88) -- (426.6,425.15) ;
\draw    (266.11,377.8) -- (300.64,429.45) ;
\draw    (469.35,302.71) -- (518.12,377.21) ;
\draw    (120.43,312.3) -- (251.7,311.81) ;
\draw    (210.41,407.38) -- (296.67,406.35) ;
\draw    (388.4,405.54) -- (505.39,404.34) ;
\draw    (422.41,309.19) -- (533.39,307.05) ;
\draw    (118.41,213.31) -- (250.67,211.81) ;
\draw    (443.76,209.89) -- (556.37,207.8) ;
\draw    (182.1,377.67) -- (231.48,291.01) ;
\draw    (124.84,255.25) -- (160.25,197.88) ;
\draw    (334.53,127.08) -- (365.99,74.76) ;
\draw    (524.47,191.39) -- (551.62,231.85) ;
\draw    (203.57,131.43) -- (232.11,87.14) ;
\draw    (350.48,121.92) -- (324.1,85.19) ;
\draw    (268.67,432.78) -- (291.32,397.94) ;
\draw    (399.57,422.43) -- (412.88,404.91) ;
\draw    (536.38,208.86) -- (549.16,186.88) ;
\draw    (465.33,121.51) -- (478.1,99.53) ;
\draw  [dash pattern={on 0.84pt off 2.51pt}]  (386.27,439.95) -- (399.57,422.43) ;
\draw  [dash pattern={on 0.84pt off 2.51pt}]  (254.9,454.92) -- (268.67,432.78) ;
\draw  [dash pattern={on 0.84pt off 2.51pt}]  (151.64,429.98) -- (182.1,377.67) ;
\draw  [dash pattern={on 0.84pt off 2.51pt}]  (97.29,299.54) -- (124.84,255.25) ;
\draw  [dash pattern={on 0.84pt off 2.51pt}]  (174.41,407.75) -- (210.41,407.38) ;
\draw  [dash pattern={on 0.84pt off 2.51pt}]  (505.39,404.34) -- (530.39,404.08) ;
\draw  [dash pattern={on 0.84pt off 2.51pt}]  (518.12,377.21) -- (528.28,393.1) ;
\draw  [dash pattern={on 0.84pt off 2.51pt}]  (313.93,69.29) -- (316.96,74.03) -- (324.1,85.19) ;
\draw  [dash pattern={on 0.84pt off 2.51pt}]  (426.6,425.15) -- (436.77,441.05) ;
\draw  [dash pattern={on 0.84pt off 2.51pt}]  (300.64,429.45) -- (310.81,445.35) ;
\draw  [dash pattern={on 0.84pt off 2.51pt}]  (533.39,307.05) -- (558.39,306.79) ;
\draw  [dash pattern={on 0.84pt off 2.51pt}]  (95.43,312.56) -- (120.43,312.3) ;
\draw  [dash pattern={on 0.84pt off 2.51pt}]  (549.16,186.88) -- (559.96,167.76) ;
\draw  [dash pattern={on 0.84pt off 2.51pt}]  (478.1,99.53) -- (488.91,80.41) ;
\draw  [dash pattern={on 0.84pt off 2.51pt}]  (365.99,74.76) -- (376.8,55.64) ;
\draw  [dash pattern={on 0.84pt off 2.51pt}]  (232.11,87.14) -- (242.92,68.02) ;
\draw  [dash pattern={on 0.84pt off 2.51pt}]  (93.41,213.57) -- (118.41,213.31) ;
\draw  [dash pattern={on 0.84pt off 2.51pt}]  (556.37,207.8) -- (581.37,207.55) ;
\draw  [dash pattern={on 0.84pt off 2.51pt}]  (551.62,231.85) -- (561.78,247.75) ;
\draw  [dash pattern={on 0.84pt off 2.51pt}]  (120.95,169.29) -- (123.98,174.02) -- (131.12,185.18) ;
\draw  [dash pattern={on 0.84pt off 2.51pt}]  (183.9,66.63) -- (186.93,71.36) -- (194.06,82.53) ;

\draw (154.49,463.14) node [anchor=north west][inner sep=0.75pt]   [align=left] {\textbf{Figure 3: }An apartment of the $\displaystyle \tilde{A}_{2}$ building containing $\displaystyle C$};
\draw (338.67,237.45) node [anchor=north west][inner sep=0.75pt]  [rotate=-359.41] [align=left] {$\displaystyle C$};
\draw (273.67,264.45) node [anchor=north west][inner sep=0.75pt]  [rotate=-359.41] [align=left] {$\displaystyle s_{0}$};
\draw (338.67,162.45) node [anchor=north west][inner sep=0.75pt]  [rotate=-359.41] [align=left] {$\displaystyle s_{1}$};
\draw (403.67,262.45) node [anchor=north west][inner sep=0.75pt]  [rotate=-359.41] [align=left] {$\displaystyle s_{2}$};
\draw (393.67,138.45) node [anchor=north west][inner sep=0.75pt]  [rotate=-359.41] [align=left] {$\displaystyle s_{2} s_{1}$};
\draw (270.67,137.45) node [anchor=north west][inner sep=0.75pt]  [rotate=-359.41] [align=left] {$\displaystyle s_{0} s_{1}$};
\draw (205.67,240.45) node [anchor=north west][inner sep=0.75pt]  [rotate=-359.41] [align=left] {$\displaystyle s_{1} s_{0}$};
\draw (456.67,232.45) node [anchor=north west][inner sep=0.75pt]  [rotate=-359.41] [align=left] {$\displaystyle s_{1} s_{2}$};
\draw (267.67,334.45) node [anchor=north west][inner sep=0.75pt]  [rotate=-359.41] [align=left] {$\displaystyle s_{2} s_{0}$};
\draw (452.67,168.45) node [anchor=north west][inner sep=0.75pt]  [rotate=-359.41] [align=left] {$\displaystyle s_{2} s_{1} s_{2}$};
\draw (397.67,332.45) node [anchor=north west][inner sep=0.75pt]  [rotate=-359.41] [align=left] {$\displaystyle s_{0} s_{2}$};
\draw (324.67,360.45) node [anchor=north west][inner sep=0.75pt]  [rotate=-359.41] [align=left] {$\displaystyle s_{2} s_{0} s_{2}$};
\draw (194.67,172.45) node [anchor=north west][inner sep=0.75pt]  [rotate=-359.41] [align=left] {$\displaystyle s_{1} s_{0} s_{1}$};

\end{tikzpicture}

\end{center}

\noindent More generally, for each element $w\in\widetilde{W}$, as in \cite{gallery}, we define the \emph{$w$-sphere} in $\Delta$ as the set of all chambers $D$ in $\Delta$ such that $\delta(D,C)=w$, i.e. in any apartment containing $C$ and $D$ where $C$ corresponds with the identity, $D$ corresponds with $w$. Denote the $w$-sphere by $\mathcal{C}_w$.\\

\noindent Observation of Figure 2 shows that the vertices in the $\widetilde{A}_2$ building that lie on the boundary of $\Delta_{n}$ (i.e. outside $\Delta_{n-1}$) are connected to $\Delta_{n-1}$ via a single chamber, and no other edge joins them to this region. In other words, they can be regarded as an isolated peak of the jagged surface. This prompts the following definition, which we state in full generality:

\begin{definition}\label{defn: summit and peak}
For each $n>0$ and each vertex $v\in \Delta_{n}\backslash \Delta_{n-1}$, we say $v$ is a \emph{peak} of $\Delta_{n}$ if 
\begin{itemize}
\item there is a unique chamber $D_v$ in $\Delta$ containing $v$ with $d(D_v,C)=n$,

\item $F_v:=D_v\backslash\{v\}$ is contained in $\Delta_{n-1}$, and the vertices of $F_v$ are the only vertices in $\Delta_{n-1}$ that are joined by an edge to $v$.

\end{itemize}

\noindent We call $D_v$ the \emph{summit} of $\Delta_{n}$ at $v$, and we call the codimension 1 facet $F_v$ the \emph{base} of the summit.
\end{definition}

\noindent\textbf{Remark:} If $n=0$ then $\Delta_0=C$, and we say that every vertex of $C$ is a peak of $\Delta_0$, with summit $C$.\\

\noindent\textbf{Example:} 1. If $\Delta$ is the $\widetilde{A}_1$ tree, then clearly every vertex in $v\in\Delta_{n}\backslash\Delta_{n-1}$ is a peak of $\Delta_{n}$, and the summit at $v$ is the unique edge adjacent to $v$ that belongs to a path beginning at $v$ and ending at $C$.\\

\noindent 2. If $\Delta$ has rank 2 and $v\in\Delta_{n}\backslash\Delta_{n-1}$ is a peak of $\Delta_{n}$, then the base $F_v$ of the $D_v$ is an edge, and we call it $b_v$. \\

\noindent Figure 2 and intuition suggest that every vertex in $\Delta_{n}\backslash\Delta_{n-1}$ is a peak of $\Delta_{n}$ in the $\widetilde{A}_2$ building. This is true, in fact the following theorem gives us something even stronger.

\begin{theorem}\label{thm: peak}
Suppose $\Delta$ is the $\widetilde{A}_2$ building and $n>0$. Then: 
\begin{enumerate}
\item If $v\in\Delta_{n}\backslash\Delta_{n-1}$, then $v$ is a peak of $\Delta_{n}$.

\item If $u,v\in\Delta_{n}\backslash\Delta_{n-1}$ are joined by an edge $e$, then there is a unique chamber $E$ with $d(E,C)=n+1$, adjacent to $e$, and to the summits $D_u$ and $D_v$.
\end{enumerate}
\end{theorem}

\begin{proof} $ $

\begin{enumerate}

\item For every vertex $v\in\Delta_{n}\backslash\Delta_{n-1}$, we know by Lemma \ref{lem: chamber for facet} that there exists a unique chamber $C(v)$ of $\Delta$ containing $v$ of minimal distance from $C$. It follows that $d(C(v),C)=n$. We will prove that $C(v)$ is the summit of $v$.

\noindent Let $I$ be the pro-$p$ Iwahori subgroup of $G$, and let $\mathcal{A}$ be the standard apartment. Without loss of generality, we will assume that $v\in\mathcal{A}$, and hence $C(v)\in\mathcal{A}$ by Lemma \ref{lem: chamber for facet}(1). Suppose $e$ is an edge of $\Delta$ joining $v$ to an edge in $\Delta_m$, then using Lemma \ref{lem: face transitivity}, there is a unique edge $e'$ in $\mathcal{A}$ such that $e'$ is conjugate to $e$ by an element $g\in I$.

Setting $v':=g\cdot v$, we know that $C(v')=C(g\cdot v)=g\cdot C(v)$ by Lemma \ref{lem: chamber for facet}(2), and thus $d(C(v'),C)=d(g\cdot C(v),g\cdot C)=d(C(v),C)=n$. Moreover, if $v'\in D$ with $d(D,C)\leq n-1$ then $v\in g^{-1}D$ and $d(g^{-1}D,C)\leq n-1$, so $v\in\Delta_{n-1}$. This contradiction implies that $v'\in\Delta_{n}\backslash\Delta_{n-1}$.

\noindent But both vertices of $e'$ lie in $\mathcal{A}$, so $v'$ lies in $\mathcal{A}$, and hence $C(v')$ lies in $\mathcal{A}$ by Lemma \ref{lem: chamber for facet}(1). It remains to show that $e'$ is an edge of $C(v')$, and it will follow that $e:=g^{-1}e'$ is an edge of $C(v)=g^{-1}C(v')$ as required.

\noindent But $\mathcal{A}$ is isomorphic to the $\widetilde{A}_2$ Coxeter complex, so all vertices adjacent to $v'$ in $\mathcal{A}$ form the hexagonal arrangement below (where the number in brackets indicates the distance from $C$).

\begin{center}

\tikzset{every picture/.style={line width=0.75pt}} 

\begin{tikzpicture}[x=0.75pt,y=0.75pt,yscale=-1,xscale=1]

\draw  [color={rgb, 255:red, 0; green, 0; blue, 0 }  ,draw opacity=1 ] (317.86,211.66) -- (387.51,318.52) -- (248.21,318.52) -- cycle ;
\draw  [color={rgb, 255:red, 0; green, 0; blue, 0 }  ,draw opacity=1 ] (173.27,211.31) -- (317.74,211.25) -- (247.82,317.91) -- cycle ;
\draw  [color={rgb, 255:red, 0; green, 0; blue, 0 }  ,draw opacity=1 ] (458,211.25) -- (387.39,318.6) -- (317.75,211.24) -- cycle ;
\draw  [color={rgb, 255:red, 0; green, 0; blue, 0 }  ,draw opacity=1 ] (245.51,104.45) -- (317.74,211.25) -- (173.27,211.25) -- cycle ;
\draw  [color={rgb, 255:red, 0; green, 0; blue, 0 }  ,draw opacity=1 ] (387.87,104.45) -- (458,211.24) -- (317.74,211.24) -- cycle ;
\draw    (245.51,104.45) -- (387.87,104.45) ;

\draw (299,260) node [anchor=north west][inner sep=0.75pt]   [align=left] {$\displaystyle C( v')$};
\draw (309,221) node [anchor=north west][inner sep=0.75pt]   [align=left] {$\displaystyle v'$};
\draw (362,187) node [anchor=north west][inner sep=0.75pt]   [align=left] {$\displaystyle ( n+2)$};
\draw (362,227) node [anchor=north west][inner sep=0.75pt]   [align=left] {$\displaystyle ( n+1)$};
\draw (304,295) node [anchor=north west][inner sep=0.75pt]   [align=left] {$\displaystyle (n)$};
\draw (216,224) node [anchor=north west][inner sep=0.75pt]   [align=left] {$\displaystyle ( n+1)$};
\draw (215,185) node [anchor=north west][inner sep=0.75pt]   [align=left] {$\displaystyle ( n+2)$};
\draw (292,115) node [anchor=north west][inner sep=0.75pt]   [align=left] {$\displaystyle (n+3)$};

\end{tikzpicture}

\end{center}

\noindent But, we know that the second vertex $u'$ of $e'$ lies in $\Delta_{n-1}$, so $d(C(u'),C)\leq n-1$ and $C(u')\in\mathcal{A}$ by Lemma \ref{lem: chamber for facet}(1). So if we assume that $e'$ is not an edge of $C(v')$, then it follows that 
\begin{itemize}
\item $u'$ is a vertex in $\mathcal{A}$,

\item $u'$ is adjacent to $v'$,

\item $u'$ is contained in a chamber of $\mathcal{A}$ of distance no more than $n-1$ from $C$
\end{itemize} 

\noindent But the extended diagram below shows that such a vertex $u'$ cannot exist in $\mathcal{A}$, and it follows that $e'=g\cdot e$ is an edge of $C(v')=g\cdot C(v)$, and hence $e$ is and edge of $C(v)$ as required.\\

\tikzset{every picture/.style={line width=0.75pt}} 

\begin{tikzpicture}[x=0.75pt,y=0.75pt,yscale=-1,xscale=1]

\draw  [color={rgb, 255:red, 0; green, 0; blue, 0 }  ,draw opacity=1 ] (328.88,279.65) -- (393.25,377.52) -- (264.52,377.52) -- cycle ;
\draw  [color={rgb, 255:red, 0; green, 0; blue, 0 }  ,draw opacity=1 ] (195.27,279.32) -- (328.78,279.27) -- (264.16,376.97) -- cycle ;
\draw  [color={rgb, 255:red, 0; green, 0; blue, 0 }  ,draw opacity=1 ] (458.39,279.28) -- (393.14,377.6) -- (328.78,279.26) -- cycle ;
\draw  [color={rgb, 255:red, 0; green, 0; blue, 0 }  ,draw opacity=1 ] (262.02,181.45) -- (328.78,279.27) -- (195.27,279.27) -- cycle ;
\draw  [color={rgb, 255:red, 0; green, 0; blue, 0 }  ,draw opacity=1 ] (393.58,181.45) -- (458.39,279.27) -- (328.78,279.27) -- cycle ;
\draw    (262.02,181.45) -- (393.58,181.45) ;
\draw  [color={rgb, 255:red, 0; green, 0; blue, 0 }  ,draw opacity=1 ] (523.19,181.45) -- (588,279.27) -- (458.39,279.27) -- cycle ;
\draw    (391.64,181.45) -- (523.19,181.45) ;
\draw  [color={rgb, 255:red, 0; green, 0; blue, 0 }  ,draw opacity=1 ] (458.39,279.27) -- (522.75,377.14) -- (394.03,377.14) -- cycle ;
\draw  [color={rgb, 255:red, 0; green, 0; blue, 0 }  ,draw opacity=1 ] (588,279.27) -- (522.75,377.59) -- (458.39,279.25) -- cycle ;
\draw  [color={rgb, 255:red, 0; green, 0; blue, 0 }  ,draw opacity=1 ] (196,278.45) -- (264.16,376.97) -- (129.22,377.28) -- cycle ;
\draw  [color={rgb, 255:red, 0; green, 0; blue, 0 }  ,draw opacity=1 ] (195.27,279.32) -- (130.02,377.65) -- (65.67,279.3) -- cycle ;
\draw  [color={rgb, 255:red, 0; green, 0; blue, 0 }  ,draw opacity=1 ] (130.47,181.48) -- (195.27,279.32) -- (65.67,279.32) -- cycle ;
\draw    (130.47,181.48) -- (262.02,181.45) ;
\draw  [color={rgb, 255:red, 0; green, 0; blue, 0 }  ,draw opacity=1 ] (196.25,83.63) -- (262.02,181.48) -- (130.47,181.48) -- cycle ;
\draw  [color={rgb, 255:red, 0; green, 0; blue, 0 }  ,draw opacity=1 ] (327.8,83.63) -- (393.58,181.45) -- (262.02,181.45) -- cycle ;
\draw  [color={rgb, 255:red, 0; green, 0; blue, 0 }  ,draw opacity=1 ] (458.39,83.63) -- (523.19,181.45) -- (393.58,181.45) -- cycle ;
\draw    (196.25,83.63) -- (458.39,83.63) ;

\draw (309.09,323.08) node [anchor=north west][inner sep=0.75pt]   [align=left] {$\displaystyle C( v')$};
\draw (320.09,289.28) node [anchor=north west][inner sep=0.75pt]   [align=left] {$\displaystyle v'$};
\draw (367.66,256.3) node [anchor=north west][inner sep=0.75pt]   [align=left] {$\displaystyle (n+2)$};
\draw (368.59,284.7) node [anchor=north west][inner sep=0.75pt]   [align=left] {$\displaystyle (n+1)$};
\draw (319.28,352.48) node [anchor=north west][inner sep=0.75pt]   [align=left] {$\displaystyle (n)$};
\draw (451.43,349.73) node [anchor=north west][inner sep=0.75pt]   [align=left] {$\displaystyle (n)$};
\draw (239.21,283.78) node [anchor=north west][inner sep=0.75pt]   [align=left] {$\displaystyle (n+1)$};
\draw (240.14,258.14) node [anchor=north west][inner sep=0.75pt]   [align=left] {$\displaystyle (n+2)$};
\draw (302.98,190.36) node [anchor=north west][inner sep=0.75pt]   [align=left] {$\displaystyle (n+3)$};
\draw (428.65,191.27) node [anchor=north west][inner sep=0.75pt]   [align=left] {$\displaystyle (n+3)$};
\draw (496.11,256.3) node [anchor=north west][inner sep=0.75pt]   [align=left] {$\displaystyle (n+2)$};
\draw (497.04,284.7) node [anchor=north west][inner sep=0.75pt]   [align=left] {$\displaystyle (n+1)$};
\draw (186.28,354.48) node [anchor=north west][inner sep=0.75pt]   [align=left] {$\displaystyle (n)$};
\draw (101.21,283.78) node [anchor=north west][inner sep=0.75pt]   [align=left] {$\displaystyle (n+1)$};
\draw (101.14,257.14) node [anchor=north west][inner sep=0.75pt]   [align=left] {$\displaystyle (n+2)$};
\draw (165.98,189.36) node [anchor=north west][inner sep=0.75pt]   [align=left] {$\displaystyle (n+3)$};
\draw (164.98,157.36) node [anchor=north west][inner sep=0.75pt]   [align=left] {$\displaystyle (n+4)$};
\draw (300.98,159.36) node [anchor=north west][inner sep=0.75pt]   [align=left] {$\displaystyle (n+4)$};
\draw (428.98,159.36) node [anchor=north west][inner sep=0.75pt]   [align=left] {$\displaystyle (n+4)$};
\draw (361.98,92.36) node [anchor=north west][inner sep=0.75pt]   [align=left] {$\displaystyle (n+5)$};
\draw (232.98,90.36) node [anchor=north west][inner sep=0.75pt]   [align=left] {$\displaystyle (n+5)$};

\end{tikzpicture}

\item Suppose $e$ is an edge joining two peaks $u,v\in\Delta_{n}\backslash\Delta_{n-1}$. Then using Lemma \ref{lem: chamber for facet}, we know that there is a unique chamber $C(e)$ of $\Delta$ containing $e$ of minimal distance from $C$, so since $u,v\in C(e)$, we must have that $d(C(e),C)>n-1$.


Let $\mathcal{A}$ be an apartment containing $e$ and $C$. Then since $u,v\in e$, it follows from Lemma \ref{lem: chamber for facet}(1) that $\mathcal{A}$ contains $C(e),C(v)$ and $C(u)$, and they must form the arrangement in $\mathcal{A}$ below.

\begin{center}

\tikzset{every picture/.style={line width=0.75pt}} 

\begin{tikzpicture}[x=0.75pt,y=0.75pt,yscale=-1,xscale=1]

\draw  [color={rgb, 255:red, 0; green, 0; blue, 0 }  ,draw opacity=1 ] (182.47,299.97) -- (252.55,192.27) -- (322.72,299.29) -- cycle ;
\draw  [color={rgb, 255:red, 0; green, 0; blue, 0 }  ,draw opacity=1 ] (322.46,298.97) -- (252.29,192.46) -- (391.59,191.78) -- cycle ;
\draw  [color={rgb, 255:red, 0; green, 0; blue, 0 }  ,draw opacity=1 ] (321.21,299.66) -- (391.3,191.97) -- (461.46,298.99) -- cycle ;
\draw  [color={rgb, 255:red, 0; green, 0; blue, 0 }  ,draw opacity=1 ] (252.18,192.28) -- (321.73,85.23) -- (391.59,191.78) -- cycle ;

\draw (315.21,233.15) node [anchor=north west][inner sep=0.75pt]  [rotate=-0.31] [align=left] {$\displaystyle E$};
\draw (235.31,243.81) node [anchor=north west][inner sep=0.75pt]  [rotate=-0.37] [align=left] {$\displaystyle C( v)$};
\draw (243.35,276.98) node [anchor=north west][inner sep=0.75pt]  [font=\normalsize,rotate=-359.67] [align=left] {$\displaystyle (n)$};
\draw (291.96,197.36) node [anchor=north west][inner sep=0.75pt]  [font=\normalsize,rotate=-359.51] [align=left] {$\displaystyle (n+1)$};
\draw (387.36,276.57) node [anchor=north west][inner sep=0.75pt]  [font=\normalsize,rotate=-359.98] [align=left] {$\displaystyle (n)$};
\draw (215.33,273.37) node [anchor=north west][inner sep=0.75pt]  [rotate=-179.72] [align=left] {$ $};
\draw (374.31,239.81) node [anchor=north west][inner sep=0.75pt]  [rotate=-0.37] [align=left] {$\displaystyle C( u)$};
\draw (243.31,169.81) node [anchor=north west][inner sep=0.75pt]  [rotate=-0.37] [align=left] {$\displaystyle v$};
\draw (387.31,167.81) node [anchor=north west][inner sep=0.75pt]  [rotate=-0.37] [align=left] {$\displaystyle u$};
\draw (291.96,169.36) node [anchor=north west][inner sep=0.75pt]  [font=\normalsize,rotate=-359.51] [align=left] {$\displaystyle (n+2)$};

\end{tikzpicture}

\end{center}

Minimality implies that $C(e)$ is the chamber denoted by $E$ in this diagram, so $E:=C(e)$ is adjacent to the summits $C(v)$ and $C(e)$.

Finally, suppose that $E'$ is another chamber adjacent to $e,C(v)$ and $C(u)$. Then $E'$ consists of $u,v$ and a third vertex $w$ that lies at the base of $C(u)$ and $C(v)$. So if $E'\neq E$ then $C(u)$ and $C(v)$ must share two distinct vertices at their bases, and hence their bases must agree.

But since the peaks $u$ and $v$ are joined by an edge, this implies that the base $b_u=b_v$ and the vertices $u,v$ form a 3-simplex, which is impossible in the $\widetilde{A}_2$ building. This proves that $E'=E$ as required.\qedhere
\end{enumerate}
\end{proof}

\noindent\textbf{Note:} We expect Theorem \ref{thm: peak}(1) to hold in full generality, i.e. for any strongly transitive building $\Delta$, every $v\in\Delta_{n}\backslash\Delta_{n-1}$ is a peak of $\Delta_{n}$, but we will not prove this here.

\begin{corollary}
$\Delta_n$ can be realised as the set of all chambers of distance no more than $n$ from $C$.
\end{corollary}

\begin{proof}

It is clear that if $d(D,C)\leq n$ then all the vertices of $D$ lie in $\Delta_n$ by the definition of $\Delta_n$. Conversely, if $v\in\Delta_0$ then $v$ is a vertex of $C$, so we can proceed by induction on $n$.

If $n>0$ and $v\in\Delta_n$, then we may assume by induction that $v\notin\Delta_{n-1}$, and hence $v$ is a peak of $\Delta_n$ by Theorem \ref{thm: peak}(1). Thus $v\in D_v$ and $d(D_v,C)\leq n$ as required.\end{proof}

\noindent The great advantage of Theorem \ref{thm: peak} is that it demonstrates that when passing from a vertex on the border of $\Delta_{n}$ to an adjacent vertex in $\Delta_{n-1}$, we stay in a fixed apartment. Next, we will show how we can use this to recover a description for the regions $\Delta_n$ in the $\widetilde{A}_2$ Bruhat-Tits building, as illustrated in Figure 2.

\subsection{Crowns of $\Delta_{n}$}\label{subsec: Delta decomposition}

\noindent Until the end of the section, we will assume that $\Delta$ is the $\widetilde{A}_2$ building. We now want to prove a decomposition for the regions $\Delta_n$ associated to the hyperspecial chamber $C$ in terms of the vertices of $C$. In particular, we want to recover a description of the surface of $\Delta_n$.\\

\noindent\textbf{Notation:} 
\begin{itemize}
\item Let $v_0,v_1,v_2$ be the three vertices of the hyperspecial chamber $C$, where $v_0$ is the hyperspecial vertex.

\item For any vertex $v\in V(\Delta)$, $n\geq 0$, let $B(v,n)$ be the set of vertices in $\Delta$ of graph theoretic distance no more than $n$ from $v$, i.e. the ball of radius $n$ centered at $v$.

 \item For each $n\geq 0$, let $P(n):=\Delta_n\backslash\Delta_{n-1}$ (where $\Delta_{-1}:=\varnothing$), which is the set of all peaks of $\Delta_n$ by Theorem \ref{thm: peak}. 

\item Let $S(n):=\{D_v:v\in P(n)\}$ be the associated set of summits.
\end{itemize}

\noindent For now, we will fix a single apartment $\mathcal{A}$ in $\Delta$ containing $C$, and analogously to $\Delta_n$, we define $$\mathcal{A}_n:=\left\{v\in V(\mathcal{A}):v\in D\text{ for some chamber }D\text{ of }\mathcal{A}\text{ with }d(C,D)\leq n\right\}.$$

\begin{proposition}\label{propn: peaks in apartments}
$\mathcal{A}_n=\Delta_n\cap\mathcal{A}$, and $\mathcal{A}_n\backslash\mathcal{A}_{n-1}=P(n)\cap\mathcal{A}$. Moreover, for all $v\in P(n)\cap\mathcal{A}$, the summit of $\Delta_n$ at $v$ lies in $\mathcal{A}$.
\end{proposition}

\begin{proof}

Using Lemma \ref{lem: chamber for facet}, we know that there exists a unique chamber $C(v)$ in $\Delta$, containing $v$, of minimal distance from $C$, and that $C(v)\in\mathcal{A}$. The proof of Theorem \ref{thm: peak}(1) shows that $C(v)$ is the summit of $\Delta_n$ at $v$, and Lemma \ref{lem: chamber for facet}(1) shows that $C(v)\in\mathcal{A}$.

This implies that $v\in\mathcal{A}_n$, and since $v\notin\Delta_{n-1}$, it is clear that $v\notin\mathcal{A}_{n-1}$ as required.\end{proof}

\noindent In light of this result, structural statements regarding $\Delta_n$ can be reduced to statements involving a single apartment, which is isometric with the $\widetilde{A}_2$ Coxeter complex. The following fact regarding the complex is intuitively obvious from observation of Figure 3, and follows from induction on $n$.

\begin{lemma}\label{lem: distance}
Let $M:=\lceil\frac{n}{2}\rceil$. Then given $v\in\mathcal{A}_n\backslash\mathcal{A}_{n-1}$:
\begin{itemize}
\item If $n$ is even, there exists $i\in\{0,1,2\}$ such that $v$ has graph theoretic distance $M$ from $v_i$, and $v$ has distance $m+1$ from $v_{i-1}$ and $v_{i+1}$ (subscripts modulo 3).

\item If $n$ is odd, there exists $i\in\{0,1,2\}$ such that $v$ has graph theoretic distance $m$ from $v_{i-1}$ and $v_{i+1}$, and distance $M+1$ from $v_i$.
\end{itemize}
\end{lemma}\qedhere

\noindent Motivated by this result, we define for each $n\in\mathbb{N}$, $i=0,1,2$ the following subset $X_{i,n}$ of vertices in $\Delta$, where $M:=\lceil\frac{n}{2}\rceil$ (taking subscripts modulo $3$):

\begin{equation}\label{eqn: X set}
X_{i,n}:=\begin{cases} B\left(v_i,M\right), & n\text{ even}\\
                                   B(v_{i-1},M)\cap B(v_{i+1},M) & n\text{ odd}\end{cases}
\end{equation}

\noindent Using Lemma \ref{lem: distance} and Proposition \ref{propn: peaks in apartments} we deduce that $\Delta_n=X_{0,n}\cup X_{1,n}\cup X_{2,n}$, and $\Delta_{n-1}\subseteq X_{i,n}$ for each $i$. In particular, if $n>0$, then for every vertex $v$ in $X_{i,n}$, there exists a chamber $D$ in $\Delta$ entirely contained in $X_{i,n}$, so we may regard $X_{i,n}$ as a set of chambers. Note that $X_{0,n},X_{1,n},X_{2,n}$ are all isometric.\\

\noindent\textbf{Examples:} 1. If $n=0$ then $X_{0,n}=\{v_0\}$.\\

\noindent 2. If $n=1$ then $X_{0,n}$ is the set of all vertices adjacent to $v_1,v_2$. As a set of chambers, this consists of $C$ and all chambers adjacent to $C$ via the edge $\{v_1,v_2\}$.\\

\noindent 3. If $n=2$ then $X_{0,2}=B(v_0,2)$ is the set of vertices adjacent to $v_0$. As a set of chambers, $X_{0,2}$ consists of all chambers in $\Delta$ containing $v_0$. We call this region the \emph{star} of $v_0$, and we usually denote it by Star$(v_0)$.

\begin{lemma}\label{lem: partition1}

Suppose $v\in P(n)$ and $v\in X_{i,n}$, and suppose $v$ is joined to a vertex $u$ of $P(n)$ with $u\neq v$.
\begin{itemize}
\item $v\notin X_{j,n}$ for all $j\neq i$.

\item $u\in X_{i,n}$.
\end{itemize}
\end{lemma}

\begin{proof}

Let $D_v,D_u$ be the summits of $v$ and $u$ respectively. By Theorem \ref{thm: peak}(2), there exists a chamber $E$ with $d(E,C)=n+1$ adjacent to $D_u$ and $D_v$, containing $v$ and $u$. So let $\mathcal{A}$ be an apartment containing $E$ and $C$, and $\mathcal{A}$ will contain $D$ and $D'$ by Theorem \ref{thm: minimal gallery}. Moreover, we know that $v\in\mathcal{A}_n\backslash\mathcal{A}_{n-1}$ by Proposition \ref{propn: peaks in apartments}.

Using Lemma \ref{lem: distance}, we know that $v\notin X_{j,n}$ for $j\neq i$, and realising $\mathcal{A}$ as the $\widetilde{A}_2$ Coxeter complex, it is clear that $u\in X_{i,n}$.\qedhere






\end{proof}

\noindent We now define the \emph{crown} of $X_{i,n}$ to be $$\text{Crown}(X_{i,n})=S(n)\cap X_{i,n}$$ i.e. Crown$(X_{i,n})$ is the set of summits in $\Delta_n$ that lie in $X_{i,n}$, and it follows from Lemma \ref{lem: partition1} that for $n>0$, we can partition $S(n)$ as $$S(n)=\text{Crown}(X_{0,n})\sqcup\text{Crown}(X_{1,n})\sqcup\text{Crown}(X_{2,n})$$ and thus $\Delta_n$ can be decomposed as
 \begin{equation}\label{eqn: crown decomposition}
\Delta_n=\text{Crown}(X_{0,n})\sqcup\text{Crown}(X_{1,n})\sqcup\text{Crown}(X_{2,n})\sqcup\Delta_{n-1}
\end{equation}

\noindent To explore the structure of the crown in more detail, we need to consider its intersection with the standard apartment $\mathcal{A}$, i.e. the set of all summits of $\mathcal{A}_n$ that lie in $X_{i,n}$. This intersection forms a single line of summits, each sharing a vertex at the base with its neighbour on either side, reminiscent of a flattened paper crown (hence the name). Unlike a paper crown, however, each peak is joined to the peaks of its neighbours on both sides, and the illustration below shows.

\begin{center}

\tikzset{every picture/.style={line width=0.75pt}} 

\begin{tikzpicture}[x=0.75pt,y=0.75pt,yscale=-1,xscale=1]

\draw  [color={rgb, 255:red, 0; green, 0; blue, 0 }  ,draw opacity=1 ] (86.11,64.32) -- (132.17,130.3) -- (40.05,130.3) -- cycle ;
\draw  [color={rgb, 255:red, 0; green, 0; blue, 0 }  ,draw opacity=1 ] (178.81,64.32) -- (224.87,130.29) -- (132.75,130.29) -- cycle ;
\draw  [color={rgb, 255:red, 0; green, 0; blue, 0 }  ,draw opacity=1 ] (270.94,64.32) -- (317,130.29) -- (224.87,130.29) -- cycle ;
\draw  [dash pattern={on 0.84pt off 2.51pt}]  (86.11,64.32) -- (270.94,64.32) ;
\draw  [dash pattern={on 4.5pt off 4.5pt}]  (317,130.29) -- (429,128.6) ;
\draw  [color={rgb, 255:red, 0; green, 0; blue, 0 }  ,draw opacity=1 ] (475.06,62.63) -- (521.13,128.6) -- (429,128.6) -- cycle ;
\draw  [color={rgb, 255:red, 0; green, 0; blue, 0 }  ,draw opacity=1 ] (567.19,62.63) -- (613.25,128.6) -- (521.13,128.6) -- cycle ;
\draw  [dash pattern={on 0.84pt off 2.51pt}]  (475.06,62.63) -- (567.19,62.63) ;
\draw  [dash pattern={on 4.5pt off 4.5pt}]  (270.94,64.32) -- (475.06,62.63) ;
\draw  [color={rgb, 255:red, 0; green, 0; blue, 0 }  ,draw opacity=1 ] (319.11,285.32) -- (365.17,351.3) -- (273.05,351.3) -- cycle ;
\draw  [dash pattern={on 0.84pt off 2.51pt}]  (86.69,196.27) -- (235,321.6) ;
\draw  [dash pattern={on 0.84pt off 2.51pt}]  (567.19,194.57) -- (405,320.6) ;
\draw    (235,321.6) -- (273.05,351.3) ;
\draw    (365.17,351.3) -- (405,320.6) ;
\draw  [color={rgb, 255:red, 0; green, 0; blue, 0 }  ,draw opacity=1 ] (132.75,130.29) -- (178.81,196.27) -- (86.69,196.27) -- cycle ;
\draw    (40.05,130.3) -- (86.69,196.27) ;
\draw  [color={rgb, 255:red, 0; green, 0; blue, 0 }  ,draw opacity=1 ] (521.13,128.6) -- (567.19,194.57) -- (475.06,194.57) -- cycle ;
\draw    (567.19,194.57) -- (613.25,128.6) ;
\draw  [color={rgb, 255:red, 0; green, 0; blue, 0 }  ,draw opacity=1 ] (224.87,130.29) -- (270.94,196.27) -- (178.81,196.27) -- cycle ;
\draw  [dash pattern={on 4.5pt off 4.5pt}]  (270.94,196.27) -- (475.06,194.57) ;

\draw (75,99) node [anchor=north west][inner sep=0.75pt]  [font=\small] [align=left] {$\displaystyle D_{1,0}^{( n)}$};
\draw (167,99) node [anchor=north west][inner sep=0.75pt]  [font=\small] [align=left] {$\displaystyle D_{2,0}^{( n)}$};
\draw (258,99) node [anchor=north west][inner sep=0.75pt]  [font=\small] [align=left] {$\displaystyle D_{3,0}^{( n)}$};
\draw (458,97) node [anchor=north west][inner sep=0.75pt]  [font=\small] [align=left] {$\displaystyle D_{m-1,0}^{( n)}$};
\draw (552,96) node [anchor=north west][inner sep=0.75pt]  [font=\small] [align=left] {$\displaystyle D_{m+1,0}^{( n)}$};
\draw (69,146) node [anchor=north west][inner sep=0.75pt]  [font=\small,color={rgb, 255:red, 245; green, 152; blue, 35 }  ,opacity=1 ] [align=left] {$\displaystyle D_{1,1}^{( n-1)}$};
\draw (551,145) node [anchor=north west][inner sep=0.75pt]  [font=\small,color={rgb, 255:red, 74; green, 144; blue, 226 }  ,opacity=1 ] [align=left] {$\displaystyle D_{1,2}^{( n-1)}$};
\draw (315,318) node [anchor=north west][inner sep=0.75pt]  [font=\small] [align=left] {$\displaystyle C$};
\draw (312,265) node [anchor=north west][inner sep=0.75pt]  [font=\small] [align=left] {$\displaystyle v_{0}$};
\draw (367.17,354.3) node [anchor=north west][inner sep=0.75pt]  [font=\small,color={rgb, 255:red, 245; green, 152; blue, 35 }  ,opacity=1 ] [align=left] {$\displaystyle v_{1}$};
\draw (259,354) node [anchor=north west][inner sep=0.75pt]  [font=\small,color={rgb, 255:red, 74; green, 144; blue, 226 }  ,opacity=1 ] [align=left] {$\displaystyle v_{2}$};
\draw (89,388) node [anchor=north west][inner sep=0.75pt]   [align=left] {\textbf{Figure 4:} The intersection of $\displaystyle X_{0,n}$ with a single apartment when $\displaystyle n$ is even,\\ \ \ \ \ \ \ \ \ \ \ \ \ \ \ the $\displaystyle m+1$ chambers at the top comprising the crown};
\draw (114,162) node [anchor=north west][inner sep=0.75pt]  [font=\small] [align=left] {$\displaystyle D_{1,0}^{( n-2)}$};
\draw (208,164) node [anchor=north west][inner sep=0.75pt]  [font=\small] [align=left] {$\displaystyle D_{2,0}^{( n-2)}$};
\draw (503,165) node [anchor=north west][inner sep=0.75pt]  [font=\small] [align=left] {$\displaystyle D_{m,0}^{( n-2)}$};

\end{tikzpicture}

\end{center}

\noindent The figure also shows that the crown of $X_{i,n}$ in the apartment lies atop of the crown of $X_{i,n-2}$, and an easy induction shows that there are exactly $m+1$ chambers, where $m:=\lfloor\frac{n}{2}\rfloor$. We label these chambers $D_{1,i}^{(n)},D_{2,i}^{(n)},\dots,D_{m+1,i}^{(n)}$.\\

\noindent\textbf{Note:} There is a choice for how we label these chambers. Our convention will be that the peak of $D_{j,i}^{(n)}$ is joined to the peaks of $D_{j-1,i}^{(n)}$ and $D_{j+1,i}^{(n)}$. 

Moreover, we know that the bases of $D_{1,i}^{(n)}$ and $D_{m+1,i}$ form edges of summits in $\mathcal{A}_{n-1}$, so fixing $i=0$, we will assume that $D_{1,0}^{(n)}$ is based at a summit in Crown$(X_{1,n-1})$ and $D_{m+1,0}^{(n)}$ is based at a summit in Crown$(X_{2,n-1})$.\\

\noindent For each $j=1,\dots,m+1$, define $w_{j,i}^{(n)}:=\delta(D_{j,i}^{(n)},C)\in\widetilde{W}$, and let $$S_{j,i}^{(n)}:=\mathcal{C}_{w_{j,i}^{(n)}}=\{D\in\Delta:\delta(D,C)=w_{j,i}^{(n)}\}$$ be its sphere in $\Delta$. Since every element of $\widetilde{W}$ corresponds uniquely to a chamber in $\mathcal{A}$, it follows that $S_{j,i}^{(n)}\cap\mathcal{A}=\{D_{j,i}^{(n)}\}$, and more generally, $S_{j,i}^{(n)}$ has intersection of size 1 with any apartment.

Let $P_{j,i}^{(n)}$ be the set of peaks of the summits in $S_{j,i}^{(n)}$, a set of vertices in bijection with $S_{j,i}^{(n)}$.\\

\noindent By symmetry, we can assume from now on that $j=0$, and we will define $S_j^{(n)}:=S_{0,j}^{(n)}$ and $P_j^{(n)}=P_{j,i}^{(n)}$. The following results complete our description of the crown of $\Delta_n$.

\begin{lemma}\label{lem: Weyl invariance}
Let $D$ be a summit of $\Delta_n$, with base $b$, and let $e$ be an edge of $D$ not equal to $b$. Then for any chambers $E_1,E_2$ adjacent to $D$ via $e$, $\delta(E_1,C)=\delta(E_2,C)$. 
\end{lemma}

\begin{proof}

Using Theorem \ref{thm: peak}, we know that $d(E_1,C)=d(E_2,C)=n+1$. So for $i=1,2$, fix an apartment $\mathcal{A}_i$ containing $E_i$ and $C$, and it follows that $\mathcal{A}_i$ will contain $D$, and all minimal galleries from $D$ to $C$.

Moreover, we know from the definition of a building that there exists an isometry $\iota:\mathcal{A}_1\to\mathcal{A}_2$ which is identical on $\mathcal{A}_1\cap\mathcal{A}_2$. So since $\iota(E_1)\in\mathcal{A}_2$ is adjacent to $D$ via $e$, and so is $E_2$, it follows that $\iota(E_1)=E_2$. Therefore $\delta(E_1,C)=\delta(\iota(E_1),\iota(C))=\delta(E_2,C)$.\end{proof}

\begin{theorem}\label{thm: d-partite}

For each $n\geq 1$, let $m:=\lfloor\frac{n}{2}\rfloor$, then \emph{Crown}$(X_{0,n})=S_1^{(n)}\sqcup\dots\sqcup S_{m+1}^{(n)}$. Moreover, fixing $j=1,\dots,m+1$:
\begin{enumerate}

    \item For all $v\in P_j^{(n)}$, all adjacent vertices to $v$ in $P(n)$ lie in either $P_{j-1}^{(n)}$ or $P_{j+1}^{(n)}$.

    \item If $v\in P_j^{(n)}$, no two distinct neighbours of $v$ in $P(n)$ have summits with the same base.
    
    \item If $1<j< m+1$, then for all $D\in S_j^{(n)}$, the base of $D$ joins a vertex in $P_{j-1}^{(n-2)}$ to a vertex in $P_{j}^{(n-2)}$.

    \item If $j=1$ (resp. $m+1$) then for all $D\in S_j^{(n)}$, the base of $D$ forms an edge of a summit in $S_{1,1}^{(n-1)}$ (resp. $S_{2,1}^{(n-1)}$).

    \item For each $v\in P_j^{(n)}$, $v$ is joined to $q$ vertices in $P_{j+1}^{(n)}$ (if $j<n$) and $q$ vertices in $P_{j-1}^{(n)}$ (if $j>1$). 

    \item $\left\vert S_{j}^{(n)}\right\vert = \left\vert P_{j}^{(n)} \right\vert = q^{n}$.
\end{enumerate}
\end{theorem}

\begin{proof}

For every summit $D\in$ Crown$(X_{0,n})$, $\delta(D,C)=w_{0,j}^{(n)}$ for some $j$, so $D\in S_{j}^{(n)}$. Moreover, if $D\in S_i^{(n)}$ for some $i\neq j$, then $\delta(D,C)=w_{0,j}^{(n)}=w_{0,i}^{(n)}$, which is impossible. Thus Crown$(X_{0,n})$ is the disjoint union of $S_{1}^{(n)},\dots,S_{m+1}^{(n)}$.

\begin{enumerate}

\item Fix $D\in S_{0,j}^{(n)}$ with peak $v$, and we know that $\delta(D,C)=w_{0,j}^{(n)}$, so after fixing an apartment $\mathcal{A}$ containing $D$ and $C$, we may assume that $D=D_{0,j}^{(n)}$. Thus the only vertices in $\mathcal{A}_n\backslash\mathcal{A}_{n-1}$ that are joined to $v$ are the peaks of $D_{0,j-1}^{(n)}$ and $D_{0,j+1}^{(n)}$ (cf. Figure 4).

\item If $v$ was joined to two peaks $u_1,u_2\in\Delta_n\backslash\Delta_{n-1}$ whose summits have the same base, then we may assume that $u_1,u_2\in P_{j+1}^{(n)}$. If $D_1,D_2$ are the summits of $u_1$ and $u_2$, then $D_1$ and $D_2$ are adjacent, and by Theorem \ref{thm: peak}, there exist chambers $E_1,E_2$ adjacent to $D$, containing $v$ and $u_1,u_2$ respectively, and they must be adjacent via the same edge $e$ of $D$. 

In particular, $E_1$ and $E_2$ are adjacent, so $E_1\sim D_1\sim D_2\sim E_2\sim E_1$ is a cycle of length 4 in $\Delta$, contradicting Lemma \ref{lem: no-cycles}.

\item If $1<j<m+1$ then the base of $D_{0,j}^{(n)}$ joins the peak of $D_{0,j-1}^{(n-2)}$ to the peak of $D_{0,j}^{(n-2)}$ (cf. Figure 4). But $\delta(D,C)=w_{0,j}^{(n)}$, so fixing any apartment $\mathcal{A}$ containing $D$ and $C$, we may assume without loss of generality that $D=D_{0,j}^{(n)}$, so $D\in S_j^{(n)}$.

\item If $j=1$ (resp. $n$) then the base of $D_{0,j}^{(n)}$ forms an edge of $D_{1,1}^{(n-1)}$ (resp. $D_{2,1}^{(n-1)}$), so by the same argument as in part 3, this base is the edge of a summit in $S_{1,1}^{(n-1)}$ (resp. $S_{2,1}^{(n-1)})$.

\item Note that for every vertex $u$ in $P_{j}^{(n)}$ with summit $D_u$, if $u$ is adjacent to $v\in P_{j+1}^{(n)}$, then there must exist a chamber $E_v$ with $d(E_v,C)=n+1$ adjacent to $D$ and $D_u$ and containing $u$ and $v$ by Theorem \ref{thm: peak}. Moreover, if $E_v$ is adjacent to $D$ via the edge $e$, then for any other chamber $E\neq D$ adjacent to $e$, $\delta(E,C)=\delta(E_u,C)$ by Lemma \ref{lem: Weyl invariance}.\\

But there are $q$ chambers $E\neq D$ adjacent to $e$ by Proposition \ref{propn: number of adjacent chambers}, and fixing an apartment containing $E$, we see that $E$ is adjacent to a chamber $D'$ with $\delta(D',C)=w_{0,j+1}^{(n)}$. So $D'\in S_{j+1}^{(n)}$ and the peak of $D'$ is joined to $v$ as required. Thus $v$ is joined to precisely $q$ vertices in $P_{j+1}^{(n)}$ if $j<n$, and the same argument shows that it is joined to $q$ vertices in $P_{j-1}^{(n)}$ if $j>1$.

\item Finally, to prove that $\left\vert S_{j}^{(n)}\right\vert = q^{n}$, we will use induction on $n$. If $n=0$, then $S_{j}^{(n)}=\{C\}$ has size $1=q^0$, and if $n=1$, $S_j^{(n)}$ is the set of all chambers adjacent to the edge $\{v_1,v_2\}$, not equal to $C$, and there are $q$ of these by Proposition \ref{propn: number of adjacent chambers}, so it has size $q=q^1$.\\

For $n\geq 2$, since $\vert S_{j-1}^{(n-2)}\vert=q^{n-2}$, and for each $v\in P_{j-1}^{(n-2)}$ there are $q$ adjacent vertices in $P_{j}^{(n-2)}$, there are $q^{n-1}$ edges joining vertices in $P_{j-1}^{(n-2)}$ to vertices in $P_j^{(n-2)}$, and each of these form the base of $q$ peaks of $\Delta_{n}$ by Proposition \ref{propn: number of adjacent chambers}, which implies that $\vert P_j^{(n)}\vert =q^n$.\qedhere

\end{enumerate}

\end{proof}

\noindent\textbf{Examples:} 1. If $n=1$, then Crown$(X_{0,n})$ is the set of all chambers adjacent to $C$ via the edge $\{v_1,v_2\}$, which is equal to $S_1^{(n)}$.\\

\noindent 2. If $n=2$,  Crown$(X_{0,n})=S_1^{(n)}\sqcup S_{2}^{(n)}$ consists of all chambers of distance 2 from $C$ containing $v_0$. Theorem \ref{thm: d-partite} says that vertices in $P_1^{(n)}\sqcup P_2^{(n)}$ form a bipartite graph.\\

\noindent Now, if $n\geq 2$, let $v\in P(n-2)$, and let $D$ be the summit at $v$, and let $X_D:=\text{Star(v)}\cong X_{0,2}$.

\begin{corollary}\label{cor: local}
If $D_1,D_2\in S_j^{(n-2)}$ for some $i,j$, and $D_1\neq D_2$, then $X_{D_1}\cap X_{D_2}\subseteq\Delta_{n-1}$
\end{corollary}

\begin{proof}

The intersection $X_D\cap\text{ Crown}(X_{0,n})$ comprises the summits in Crown$(X_{0,n})$ that contain $v$ at their base,

Suppose there exists a vertex $v\in X_{D_1}\cap X_{D_2}$ with $v\notin\Delta_n$. Then since $X_D\subseteq\Delta_{n+1}$, it follows from Theorem \ref{thm: peak}(1) that $v$ is a peak of $\Delta_{n+1}$.

But if $v_1,v_2$ are the peaks of $D_1$ and $D_2$ respectively, then $v$ is joined to $v_1$ and $v_2$, so applying Theorem \ref{thm: peak}(2) we see that $v_1$ and $v_2$ belong to the base of the summit of $v$. But by Theorem \ref{thm: d-partite}(3,4), the base of the summit at $v$ cannot join two distinct vertices in $P_{j}^{(n-1)}$, so this implies $v_1=v_2$, and hence $D_1=D_2$.\end{proof}

\noindent This result implies that we can decompose $S_j^{(n)}\sqcup S_{j+1}^{(n)}$ as a disjoint union of $X_D\cap\text{ Crown}(X_{0,n})\cong\text{ Crown}(X_{0,2})$, as $D$ ranges over the summits in $S_j^{(n-2)}$. Thus we can reduce study of the crown of $X_{0,n}$ to the crown of the small region $X_{0,2}$, which will be crucial in our argument.

\subsection{The action of $I$ on $\Delta_n$}

\noindent Now, recall that $G=SL_3(K)$ acts on $\Delta$ by automorphisms, and recall from section \ref{subsec: facet subgroups} how we define the subgroups $I_F\subseteq \text{Stab}_G(F)$ for each facet $F$ in $\Delta$, and let $I=I_C$ be the pro-$p$ Iwahori subgroup.

\begin{lemma}\label{lem: preservation}
For each $n\in\mathbb{N}$, the action of $I$ on $\Delta$ preserves 

\begin{itemize}
\item $\Delta_n$.

\item $X_{0,n},X_{1,n},X_{2,n}$.

\item \emph{Crown}$(X_{i,n})$ and \emph{Crown}$^e(X_{i,n})$ for $i=0,1,2$.

\item $S_{j,i}^{(n)}$ for each $j=1,\dots,\lfloor\frac{n}{2}\rfloor+1$.

\end{itemize}
\end{lemma}

\begin{proof}

First, if $v\in\Delta_n$ then there exists a chamber $D$ with $v\in D$ and $d(D,C)\leq n$. So given $g\in I$, $d(g\cdot D,C)=d(g\cdot D,g\cdot C)=d(D,C)=n$, so $g\cdot v\in\Delta_n$.

Moreover, if $v\notin\Delta_{n-1}$ then $g\cdot v\notin\Delta_{n-1}$, otherwise $v=g^{-1}\cdot(g\cdot v)\in\Delta_{n-1}$. So $I$ preserves all peaks of $\Delta_n$.\\

\noindent Also note that for all $g\in I$, $g\cdot v_i=v_i$ for $i=0,1,2$. So if $d(v,v_i)=m$ for some $i$, then $d(g\cdot v,v_i)=d(g\cdot v,g\cdot v_i)=d(v,v_i)=m$. So clearly if $v\in X_{i,n}$ then $g\cdot v\in X_{i,n}$.

Since $\text{ Crown}(X_{i,n})=X_{i,n}\cap S(n)$ by definition, it is clear that $I$ preserves $\text{ Crown}(X_{i,n})$. Moreover, for all $D\in\text{ Crown}(X_{i,n-2})$, $g\in I$, it is clear that $g\cdot X_D=X_{g\cdot D}$, and $g\cdot D\in\text{ Crown}(X_{i,n-2})$, so it follows that $I$ preserves Crown$^e(X_{i,n})$.\\

\noindent By definition, $S_{j,i}^{(n)}$ is the $w$-sphere of the chamber $D_{j,i}$ in the standard apartment, so $D\in S_{j,i}^{(n)}$ if and only if $\delta(D,C)=\delta(D_{j,i}^{(n)},C)$. But for any $g\in I$, since $g\cdot C=C$ it is clear that $\delta(g\cdot D,C)=\delta(g\cdot D,g\cdot C)=\delta(D,C)=\delta(D_{j,i}^{(n)},C)$, so $g\cdot D\in S_{j,i}^{(n)}$ as required.\end{proof}

\noindent In the following definition, $\Delta$ can be any Bruhat-Tits building.

\begin{definition}\label{defn: complete region}
We say that $\mathcal{X}$ is a \emph{complete region} of $\Delta$, if 
\begin{itemize}
\item there exists $n\in\mathbb{N}$ with $\Delta_n\subseteq\mathcal{X}\subset\Delta_{n+1}$, and

\item for all vertices $v$ in $\mathcal{X}$, there exists a chamber $D\in\mathcal{X}$ containing $v$.
\end{itemize}
\end{definition}

\noindent\textbf{Note:} We can realise a complete region as a set of chambers in $\Delta$. Indeed, we could define a complete region as a set of chambers $\mathcal{X}$ with $\Delta_n\subseteq\mathcal{X}\subseteq\Delta_{n+1}$.\\

\noindent\textbf{Examples:} 1. Of course, $\Delta_n$ is itself a complete region, since for every $v\in\Delta_n$, by definition, there exists a chamber $D$ with $v\in D$ and $d(D,C)\leq n$. So all vertices of $D$ lie in $\Delta_n$.\\

\noindent 2. If $G$ has rank 1, then $\Delta$ is the Bruhat-Tits tree of degree $q$, and $\Delta_n$ is the set of all vertices of distance no more than $m$ from either $v_0=[\mathcal{O}\oplus\mathcal{O}]$ or $v_1=[\pi\mathcal{O}\oplus\mathcal{O}]$ (see Figure 1). So any complete region would comprise precisely these vertices, and any collection of vertices in $\Delta_{n+1}\backslash\Delta_{n}$, each of which are joined by a unique edge to $\Delta_n$.\\

\noindent 3. If $G$ has type $\widetilde{A}_2$, then using Theorem \ref{thm: peak} we see that any set of vertices $\mathcal{X}$ containing $\Delta_n$ and a collection of vertices in $\Delta_{n+1}\backslash\Delta_n$ is a complete region, similar to the rank 1 case. This should hold in higher ranks, but we do not prove this here.

\begin{lemma}\label{lem: I-invariant complete region}
Let $\mathcal{X}$ be a complete region of $\Delta$ with $\Delta_{n-1}\subseteq\mathcal{X}\subseteq\Delta_{n}$. Then $\mathcal{X}$ is $I$-invariant if and only if $\mathcal{X}=\Delta_{n-1}\sqcup \underset{(i,j)\in \Gamma}\bigsqcup S_{j,i}^{(n)}$ for some unique subset $\Gamma\subseteq\{1,\dots,\lfloor\frac{n}{2}\rfloor+1\}\times \{0,1,2\}$.
\end{lemma}

\begin{proof}

One direction follows immediately from Lemma \ref{lem: preservation}. Conversely, if $D$ is a chamber in $\mathcal{X}$ with $D\notin\Delta_{n-1}$, then $D\in\Delta_{n}$, so $D\in S_{j,i}^{(n)}$ for some $i=0,1,2$, $j\leq m+1$. But since $S_{j,i}^{(n+1)}$ and $\mathcal{X}$ are $I$-invariant by Lemma \ref{lem: preservation}, it follows immediately that $S_{j,i}^{(n)}\subseteq\mathcal{X}$ as required.\end{proof}

\begin{proposition}\label{propn: distance fixing}
Let $D,E$ be two summits of $\Delta_n$, whose peaks are joined by an edge, and let $F$ be a chamber, adjacent to $E$ at the base, with $d(F,C)=n-1$. Then there exists $g\in I$ such that $g$ fixes $D$, but $g$ does not fix $F$.
\end{proposition}

\begin{proof}

By Theorem \ref{thm: peak}(2), we know that there exists a chamber $H$, with $d(H,C)=n+1$, adjacent to $D$ and $E$. Fix an apartment $\mathcal{A}$ containing $C$ and $H$, and it follows from Theorem \ref{thm: minimal gallery} that $\mathcal{A}$ contains $D,E$ and $F$.

Without loss of generality, we may assume that $\mathcal{A}$ is the standard apartment.\\

\noindent Realising vertices in $\Delta$ as equivalence classes of full rank lattices in $K^3$, the vertices of $D,E,F$ all have the form $[\langle \alpha e_1,\beta e_2,\gamma e_3\rangle]$, for $\alpha,\beta,\gamma\in K$, where $e_1,e_2,e_3$ is the standard basis for $K^3$. By definition of the hyperspecial chamber, $C$ has vertices $v_0=[\langle e_1,e_2,e_3\rangle]$, $v_1=[\langle e_1,e_2,\pi e_3\rangle]$ and $v_2=[\langle e_1,\pi e_2,\pi e_3\rangle]$.\\

\noindent Setting $u,v$ as the peaks of $D$ and $E$ respectively, if $M:=\lceil\frac{n}{2}\rceil$, then using Lemma \ref{lem: distance} and Lemma \ref{lem: partition1} we may assume without loss of generality that $u$ and $v$ both have distance $M$ from $v_0$, distance $M+1$ from $v_1$, and the same distance from $v_2$, which is either $M$ or $M+1$.\\

\noindent If $u,v$ have distance $M+1$ from $v_2$ (i.e. $n$ is even), then an easy induction shows that $u$ and $v$ have the form $u=[\langle\pi^n e_1,\pi^i e_2, e_3\rangle]$ for some $i\leq n$, and $v=[\langle \pi^n e_1,\pi^{i\pm 1}e_2,e_3\rangle]$. Moreover, the vectors at the base of $D$ have the form $w_1=[\langle \pi^{n-1}e_1,\pi^{i}e_2,e_3\rangle]$ and $w_2:=[\langle \pi^{n-1}e_1,\pi^{i-1}e_2,e_3\rangle]$. 

Also, the base of $E$ contains a vertex $w$ that it does not share with $D$. Since $E$ is adjacent to $F$ at the base, it follows that $w$ is also a vertex of $F$. So it remains to find an element $g\in I$ such that $g$ fixes $u,w_1$ and $w_2$, but $g$ does not fix $w$.\\ 

\noindent If $v=[\langle \pi^n e_1,\pi^{i-1}e_2,e_3\rangle]$ then $w=[\langle \pi^{n-1}e_1,\pi^{i-2}e_2,e_3\rangle]$, and we take $g:=\left(\begin{array}{ccc} 1 & \pi^{n-i} & 0\\ 0 & 1 & 0\\ 0 & 0 & 1\end{array}\right)$. On the other hand, if  $v=[\langle \pi^n e_1,\pi^{i+1}e_2,e_3\rangle]$ then $w=[\langle \pi^{n-1}e_1,\pi^{i+1}e_2,e_3\rangle]$, and we take $g:=\left(\begin{array}{ccc} 1 & 0 & 0\\ 0 & 1 & \pi^i\\ 0 & 0 & 1\end{array}\right)$.\\

\noindent Since $I$ is the group of matrices in $SL_3(\mathcal{O})$ that are unipotent upper triangular modulo $\pi$, we see that $g\in I$ in both cases. We can immediately calculate that $g$ fixes $u,w_1$ and $w_2$, and $g$ does not fix $w$ as required. A similar argument applies when $n$ is odd.\end{proof}

\subsection{The border of $\Delta_n$}\label{subsec: border}

In the proof of our main theorems, rather than using the action of $G$ on vertices or chambers in $\Delta$, it will usually be more fruitful to consider the action of $G$ on edges, and we will be particularly interested in the edges that can be said to lie on the boundary of the region. The following definition makes this precise:

\begin{definition}\label{defn: border}
We say that an edge $e$ in $\Delta$ \emph{lies on the border} of $\Delta_{n}$ if 
\begin{enumerate}
\item both vertices of $e$ lie in $\Delta_{n}$,

\item  at least one lies outside of $\Delta_{n-1}$, and 

\item $e$ is the base of a summit of $\Delta_r$ for some $r>n$.
\end{enumerate}
\end{definition}

\begin{lemma}\label{lem: border characterisation}
If $e$ lies on the border of $\Delta_{n}$, then there is a unique chamber $D$ in $\Delta_{n}$ containing $e$.
\end{lemma}

\begin{proof}

Setting $u,v\in\Delta_n$ as the two edges of $e$, we know using Lemma \ref{lem: chamber for facet} that there exist unique chambers $C(u),C(v),C(e)$ containing $u,v,e$ respectively, and of minimal distance to $C$ among all such chambers.

Without loss of generality, we may assume that $u\notin\Delta_{n-1}$, and it follows that $d(C(u),C)=n$. By Theorem \ref{thm: peak}(1), $u$ is a peak of $\Delta_n$, and clearly $C(u)$ is the summit at $u$.\\

\noindent If $v\in\Delta_{n-1}$, then by Definition \ref{defn: summit and peak} we know that $v\in C(u)$, and hence $e$ is an an edge of $C(u)$. So taking $D:=C(u)$, we know that $D\in\Delta_n$. But we also know that $e$ is the base of a summit $E$ of $\Delta_r$ for some $r>n$. Since $E$ is adjacent to $D$, it follows that $d(E,C)=n+1$, and hence all chambers adjacent to $D$ via $e$ have distance $n+1$ from $C$ by Corollary \ref{cor: adjacent to facet}. In particular, $D$ is the unique chamber adjacent to $e$ which lies in $\Delta_n$.\\

\noindent So we may assume that $v\notin\Delta_{n-1}$, and thus $v$ is a peak of $\Delta_n$ with summit $C(v)$. Using Theorem \ref{thm: peak}(2), there exists a unique chamber $D\in\Delta_n$ with $d(D,C)=n+1$, adjacent to $e,C(u),C(v)$. Clearly $d(C(e),C)\leq d(D,C)=n+1$, so either $C(e)=D$ or $d(C(e),C)\leq n$ by minimality. 

But since $u,v\in C(e)$, we know that $n=d(C(u),C)=d(C(v),C)\leq d(C(e),C)$, so if $d(C(e),C)\leq n$ then this forces equality, so $C(u)=C(v)=C(e)$ by minimality, and hence $u=v$, a contradiction.\\

\noindent Therefore $C(e)=D$, and all chambers adjacent to $D$ via $e$ are summits of $\Delta_{n+2}$, and hence lie outside $\Delta_n$ as required.\end{proof}

\noindent Moreover, if $D$ is a chamber in $\Delta$ with $d(D,C)=m$, we say that an edge $e$ of $D$ is an \emph{exterior edge} if for all chambers $E$ adjacent to $D$ via $e$, $d(E,C)=m+1$. Note that this does not necessarily imply that $e$ lies on the border of $\Delta_m$, since it is possible that no such chamber $E$ would be a peak of $\Delta_{m+1}$.

\begin{lemma}\label{lem: unique edge}

Let $\mathcal{X}$ be a complete region, let $D$ be a chamber in $\Delta$, and suppose that every exterior edge of $D$ lies in $\mathcal{X}$. Then $D\in\mathcal{X}$.

\end{lemma}

\begin{proof}

Let $m:=d(D,C)$ for convenience, and fix an exterior edge $e=\{v,w\}$ of $D$, then $e\in\mathcal{X}$ by assumption. Since $\mathcal{X}$ is complete, we know that $\Delta_n\subseteq\mathcal{X}\subseteq\Delta_{n+1}$ for some $n\in\mathbb{N}$, thus both vertices of $e$ lie in $\Delta_{n+1}$.\\

\noindent Let us first suppose that $u\in\Delta_n$, $w\in\Delta_{n+1}\backslash\Delta_{n}$. Then by Theorem \ref{thm: peak}, $w$ is a peak of $\Delta_{n+1}$, with summit $D'$ of distance $n+1$ from $C$, containing $u$, and all other chambers adjacent to $e$ have distance $n+2$. Moreover, since $\mathcal{X}$ is complete, $D'\in\mathcal{X}$. 

But since $e$ is exterior, we are assuming all chambers adjacent to $e$ have distance $m$ or $m+1$ from $C$, and $D$ is the unique such chamber with $d(D,C)=m$. Thus we conclude that $n=m+1$ and $D'=D\in\mathcal{X}$ as required. So we can assume from now on that either $v,w\in\Delta_n$, or $v,w\in\Delta_{n+1}\backslash\Delta_n$.\\

\noindent Suppose there exists another exterior edge $e'\neq e$ of $D$. Then similarly $e'\in\mathcal{X}$, and either both vertices of $e'$ lie in $\Delta_n$, or both lie in $\Delta_{n+1}\backslash\Delta_n$. But the vertices of $e$ and $e'$ comprise all vertices of $D$, and they must share a common vertex. 

Therefore, either all vertices of $D$ lie in $\Delta_{n+1}\backslash\Delta_n$, contradicting Theorem \ref{thm: d-partite}, or they all lie in $\Delta_n\subseteq\mathcal{X}$, so we conclude that $D\in\mathcal{X}$ as required.\\ 

\noindent So we can assume from now on that $e$ is the unique exterior edge of $D$. Since all summits of a region $\Delta_r$ have at least two exterior edges (the edges joined to the peak) by Theorem \ref{thm: peak}(2), it follows that $D$ is not a summit in the building.\\

\noindent Now, suppose that $v,w\in\Delta_n$, and assume for contradiction that $D\notin\mathcal{X}$. Then the third vertex of $D$ must lie outside of $\Delta_n$, making $D$ a summit of $\Delta_{r}$ for some $r\geq m+1$ by Theorem \ref{thm: peak}(1) -- contradiction.\\

\noindent On the other hand, if $u,w\in\Delta_{n+1}\backslash\Delta_n$, then they are both peaks of $\Delta_{n+1}$ by Theorem \ref{thm: peak}(1). Therefore, by Theorem \ref{thm: peak}(2), the bases of the summits $D_u$ and $D_w$ share a common vertex $v\in\Delta_n$, and $E=\{u,v,w\}$ is the unique chamber, containing $e$, with $d(E,C)=n+2$. And since $u,w\in\mathcal{X}$ and $v\in\Delta_n\subseteq\mathcal{X}$, it follows that $E\in\mathcal{X}$.

But $e$ is the base of a summit of $\Delta_{n+3}$ by Theorem \ref{thm: d-partite}, so all chambers adjacent to $e$ that are not equal to $E$ are peaks of $\Delta_{n+3}$, and none of these can be $D$. Hence $D=E\in\mathcal{X}$.\end{proof}

\section{Coefficient systems}\label{sec: coefficient systems}

We now return to the general setting. Throughout this section, let $G=\mathbb{G}(K)$, for $\mathbb{G}$ a split semisimple, simply connected algebraic group. Let $d\in\mathbb{N}$ be the rank of $G$, let $I$ be the canonical pro-$p$ Iwahori subgroup of $G$, and let $\mathbb{X}:=k[G/I]$ be the standard module. As in the previous section, $\Delta=\widetilde{\Delta}(G)$ will denote the Bruhat-Tits building of $G$, which has rank $d$, and we let $C$ be the hyperspecial chamber in $\Delta$.

\subsection{$\ast$-acyclic $\mathcal{H}$-modules}\label{subsec: acyclic modules}

\noindent Recall from section \ref{subsec: facet subgroups} how we define the subgroups $I_F\subseteq J_F$ for each facet $F$ in $\Delta$, and how we can realise $I_F$ as the set of all elements of the group $\mathcal{G}_F^\circ(\mathcal{O})$ that lie in the unipotent radical modulo $\pi$. With this description in mind, we define the following data as in \cite[Section 3.3.1]{gorenstein} and \cite[section 1.3]{torsion}:

\begin{definition}
For each face $F$ of $C$, define $\mathbb{X}_F:=k[\mathcal{G}_F^{\circ}(\mathcal{O})/I]=\text{\emph{ind}}_I^{\mathcal{G}_F^{\circ}(\mathcal{O})}(1)$, and $\mathcal{H}_F:=\text{End}_{k[\mathcal{G}_F^{\circ}(\mathcal{O})]}(\mathbb{X}_F)^{\text{\emph{op}}}$
\end{definition}

\noindent This is completely analogous to the definition of the standard module $\mathbb{X}$ and the pro-$p$ Iwahori-Hecke algebra $\mathcal{H}$. Indeed, $\mathcal{H}_F$ is a finite dimensional subalgebra of $\mathcal{H}$, and $\mathcal{H}$ is free as a left and right $\mathcal{H}_F$-module \cite[Proposition 1.3]{torsion}. 

\begin{lemma}\label{lem: face induction}
If $F$ is a face of $C$, then $\mathbb{X}_F\otimes_{\mathcal{H}_{F}}\mathcal{H}\cong\mathbb{X}^{I_F}$ via $x\otimes h\mapsto h(x)$.
\end{lemma}

\begin{proof}

This is \cite[Proposition 1.3]{torsion}.\end{proof}




\noindent Now, let $M$ be a $\mathcal{H}$-module, and recall from \cite[section 1.3.1]{torsion} that $M$ is \emph{$\ast$-acyclic} if $Ext_{\mathcal{H}}^i(M,\mathcal{H})=0$ for all $i\geq 1$. Acyclicity is a very desirable homological property, and the following lemma ensures that by focusing on modules induced from $\mathcal{H}_F$, it is one we can often deduce.

\begin{lemma}\label{lem: acyclic induction}
If $F$ is a face of $C$ and $N$ is a finitely generated $\mathcal{H}_F$-module, then $N\otimes_{\mathcal{H}_F}\mathcal{H}$ is a $\ast$-acyclic $\mathcal{H}$-module.
\end{lemma}

\begin{proof}
This is \cite[Corollary 1.5]{torsion}.\end{proof}

\noindent Using Lemma \ref{lem: face induction}, it follows that $\mathbb{X}^{I_F}$ is $\ast$-acyclic. Moreover, since $\mathcal{H}$ is free over $\mathcal{H}_F$, the functor $-\otimes_{\mathcal{H}_F}\mathcal{H}$ is exact, so for any submodule $N$ of $\mathbb{X}_F$, if we let $M$ be the $\mathcal{H}$-submodule of $\mathbb{X}^{I_F}$ generated by $N$, then $$\mathbb{X}^{I_F}/M\cong (\mathbb{X}_F/N)\otimes_{\mathcal{H}_F}\mathcal{H}$$ is $\ast$-acyclic.

\begin{lemma}\label{lem: adjacent acyclicity}
Let $v$ be a vertex in $\Delta$, and $S$ is any set of edges in $\Delta$ adjacent to $v$. If we set $$N:=\underset{e\in S}{\sum}{\mathbb{X}^{I_e}}$$ then $\mathbb{X}^{I_v}/N$ is $\ast$-acyclic.
\end{lemma}

\begin{proof}
Firstly, $v=g\cdot v_0$ for some $g\in G$, where $v_0$ is the hyperspecial vertex. If we let $S_0:=g^{-1}S$ and $N_0:=\underset{e\in S_0}{\sum}{\mathbb{X}^{I_e}}$, then there is an isomorphism of $\mathcal{H}$-modules $\mathbb{X}^{I_v}/N\cong\mathbb{X}^{I_{v_0}}/N_0$ via $y+N\mapsto g\cdot y+N_0$, so we may assume that $v=v_0$, and hence $\mathbb{X}^{I_v}\cong\mathbb{X}_v\otimes_{\mathcal{H}_v}\mathcal{H}$ by Lemma \ref{lem: face induction}.\\

\noindent For each $e\in S$, $v$ is a face of $e$, so $\mathbb{X}^{I_e}=(\mathbb{X}^{I_v})^{I_e}\cong \mathbb{X}_v^{I_e}\otimes_{\mathcal{H}_v}\mathcal{H}$. So if we let $V$ be the $\mathcal{H}_v$-submodule of $\mathbb{X}_v$ generated by $\{\mathbb{X}_v^{I_e}:e\in S\}$, then $N=\underset{e\in S}{\sum}{\mathbb{X}^{I_e}}$ is spanned by $V$, i.e. $V\otimes_{\mathcal{H}_v}\mathcal{H}=N$.\\

\noindent But since $\mathcal{H}$ is free over $\mathcal{H}_v$, the functor $-\otimes_{\mathcal{H}_v}\mathcal{H}$ is exact. So applying it to the exact sequence of $\mathcal{H}_v$-modules $$0\to V\to\mathbb{X}_v\to\mathbb{X}_v/V\to 0$$ we obtain an exact sequence $$0\to N\to \mathbb{X}^{I_v}\to(\mathbb{X}_v/V)\otimes_{\mathcal{H}_v}\mathcal{H}\to 0$$ In other words $\mathbb{X}^{I_v}/N\cong(\mathbb{X}_v/V)\otimes_{\mathcal{H}_v}\mathcal{H}$ as $\mathcal{H}$-modules, and since $\mathbb{X}_v/V$ is finitely generated as an $\mathcal{H}_v$-module, it follows from Lemma \ref{lem: acyclic induction} that $\mathbb{X}^{I_v}/N$ is $\ast$-acyclic.\end{proof}

\noindent Now, for any $\mathcal{H}$-module $M$, we define the \emph{dual} of $M$ to be the $\mathcal{H}$-module $M^*:=\text{Hom}_{\mathcal{H}}(M,\mathcal{H})$. The following result adapts the proof of \cite[Corollary 2.6]{torsion}.

\begin{proposition}\label{propn: direct limit}
If $(M_n)_{n\in \mathbb{N}}$ is a direct system of $\ast$-acyclic $\mathcal{H}$-modules, such that $M_{n+1}/M_n$ is $\ast$-acyclic for each $n\in\mathbb{N}$, then $M:=\underset{n\in I}{\varinjlim}\text{ }M_n$ is $\ast$-acyclic.
\end{proposition}

\begin{proof}

Firstly, $\ast$-acyclicity of $M_{n+1}/M_n$ implies that $Ext_{\mathcal{H}}^1(M_{n+1}/M_n,\mathcal{H})=0$, and it follows that the sequence $0\to(M_{n+1}/M_n)^*\to M_{n+1}^*\to M_n^*\to 0$ is exact, and hence $M_{n+1}^*\to M_n^*$ is surjective.\\

\noindent Now, consider the spectral sequence defined by $$E_2^{i,j}:=\underset{n}{\varprojlim}^{(i)}Ext_{\mathcal{H}}^j(M_n,\mathcal{H})\implies Ext_{\mathcal{H}}^{i+j}(M,\mathcal{H})$$ Since each $M_n$ is $\ast$-acyclic, only the first column of the $E_2$ page can be non-zero, from which we deduce an isomorphism $$\underset{n}{\varprojlim}^{(i)}\text{ Hom}_{\mathcal{H}}(M_n,\mathcal{H})\cong Ext_{\mathcal{H}}^{i}(M,\mathcal{H})$$ for each $i\geq 0$. For $i>1$, $\underset{n}{\varprojlim}^{(i)}\text{ Hom}_{\mathcal{H}}(M_n,\mathcal{H})=0$ by \cite[Definition 3.5.1]{Weibel}, and since the transition maps $M_{n+1}^*\to M_n^*$ are surjective, it follows from \cite[Lemma 3.5.3]{Weibel} that $\underset{n}{\varprojlim}^{(1)}\text{ Hom}_{\mathcal{H}}(M_n,\mathcal{H})=\underset{n}{\varprojlim}^{(1)}M_n^*=0$. Therefore, $Ext_{\mathcal{H}}^{i}(M,\mathcal{H})=0$ for all $i>0$ as required.\end{proof}

\subsection{Coefficient systems of $\mathcal{H}$-modules}

For each $i=0,\dots,d$, let $\mathcal{F}_i$ be the set of facets of dimension $i$ in $\Delta$. For each $i\leq d$ and each $F\in\mathcal{F}_i$, we can define an \emph{orientation} on $F$ (see \cite[Chapter \rom{2}.1]{coefficient} for the precise definition). In fact, for $i\geq 1$, there are two possible orientations on $F$, and we will denote these by $(F,c)$ and $(F,-c)=:\sigma(F,c)$.  \\

\noindent To give a rough illustration, if $i=1$ and $F=e$ is an edge, then the two orientations can be regarded as the two ways to make $e$ a directed edge. We only need to specify which vertex of $e$ is the origin, and which is the target. 

If $i=2$ and $F$ is a 2-simplex, then the two orientations of $F$ correspond to the two possible ways of orienting the three edges of $F$ to give then the same direction:

\begin{center}

\tikzset{every picture/.style={line width=0.75pt}} 

\begin{tikzpicture}[x=0.75pt,y=0.75pt,yscale=-1,xscale=1]

\draw   (188,125.6) -- (285,265.6) -- (91,265.6) -- cycle ;
\draw    (91,265.6) -- (146.87,184.25) ;
\draw [shift={(148,182.6)}, rotate = 124.48] [color={rgb, 255:red, 0; green, 0; blue, 0 }  ][line width=0.75]    (10.93,-3.29) .. controls (6.95,-1.4) and (3.31,-0.3) .. (0,0) .. controls (3.31,0.3) and (6.95,1.4) .. (10.93,3.29)   ;
\draw    (188,125.6) -- (236.86,195.96) ;
\draw [shift={(238,197.6)}, rotate = 235.22] [color={rgb, 255:red, 0; green, 0; blue, 0 }  ][line width=0.75]    (10.93,-3.29) .. controls (6.95,-1.4) and (3.31,-0.3) .. (0,0) .. controls (3.31,0.3) and (6.95,1.4) .. (10.93,3.29)   ;
\draw    (285,265.6) -- (180,265.6) ;
\draw [shift={(178,265.6)}, rotate = 360] [color={rgb, 255:red, 0; green, 0; blue, 0 }  ][line width=0.75]    (10.93,-3.29) .. controls (6.95,-1.4) and (3.31,-0.3) .. (0,0) .. controls (3.31,0.3) and (6.95,1.4) .. (10.93,3.29)   ;
\draw   (477,129.6) -- (574,269.6) -- (380,269.6) -- cycle ;
\draw    (477,129.6) -- (428.13,200.95) ;
\draw [shift={(427,202.6)}, rotate = 304.41] [color={rgb, 255:red, 0; green, 0; blue, 0 }  ][line width=0.75]    (10.93,-3.29) .. controls (6.95,-1.4) and (3.31,-0.3) .. (0,0) .. controls (3.31,0.3) and (6.95,1.4) .. (10.93,3.29)   ;
\draw    (574,269.6) -- (521.15,194.24) ;
\draw [shift={(520,192.6)}, rotate = 54.96] [color={rgb, 255:red, 0; green, 0; blue, 0 }  ][line width=0.75]    (10.93,-3.29) .. controls (6.95,-1.4) and (3.31,-0.3) .. (0,0) .. controls (3.31,0.3) and (6.95,1.4) .. (10.93,3.29)   ;
\draw    (380,269.6) -- (482,269.6) ;
\draw [shift={(484,269.6)}, rotate = 180] [color={rgb, 255:red, 0; green, 0; blue, 0 }  ][line width=0.75]    (10.93,-3.29) .. controls (6.95,-1.4) and (3.31,-0.3) .. (0,0) .. controls (3.31,0.3) and (6.95,1.4) .. (10.93,3.29)   ;

\draw (170,202) node [anchor=north west][inner sep=0.75pt]   [align=left] {$\displaystyle ( F,c)$};
\draw (459,206) node [anchor=north west][inner sep=0.75pt]   [align=left] {$\displaystyle ( F,c)$};
\draw (195,296) node [anchor=north west][inner sep=0.75pt]   [align=left] {\textbf{Figure 5:} The two orientations on a 2-simplex};

\end{tikzpicture}

\end{center}

\noindent More generally, the two orientations of $F\in\mathcal{F}_i$ correspond to the two ways that all faces of $F$ have compatible orientation. This is stated more precisely in \cite{coefficient}, but since we are largely concerned with the rank 2 building in this paper, we will not explore this in more depth now. \\

\noindent\textbf{Notation:} Suppose $F\in\mathcal{F}_i$, $c$ is an orientation of $F$, $F'\in\mathcal{F}_{i-1}$ is a face of $F$, and $F''\in\mathcal{F}_{i+1}$ contains $F$ as a face.

\begin{enumerate}
\item Denote by $c\downarrow_{F'}$ the orientation of $F'$ induced by $c$.

\item There is a unique orientation on $F''$ which restricts to $c$ on $F$. We denote this orientation by $c\uparrow^{F''}$.
\end{enumerate}

\noindent Moreover, for each $i=0,\dots,d$, let $$\mathcal{F}_i^{o}:=\{(F,c):F\in\mathcal{F}_i,c\text{ an orientation of }F\}$$ be the set of all oriented $i$-facets in $\Delta$.\\

\noindent\textbf{Note:} If $\overset{\to}{F}=(F,c)$ is an oriented facet, we sometimes write the subgroup $I_F$ as $I_{\overset{\to}{F}}$, though the orientation does not change the subgroup.\\

\noindent Now, let $\mathbf{V}$ be a $(G,\mathcal{H})$-bimodule, and continuing to follow \cite[Chapter \rom{2}.1]{coefficient}, we define the \emph{coefficient system} of $\mathbf{V}$ to be the collection of $\mathcal{H}$-submodules $$\underline{\underline{\mathbf{V}}}:=\{\mathbf{V}^{I_F}:F\in\mathcal{F}_i\text{ for some }i\leq d\}$$ 

\begin{definition}\label{defn: i-chain}
For each $i=0,\dots,d$, we say that a function $\alpha:\mathcal{F}_i^{\text{o}}\to\mathbf{V}$ is an \emph{oriented cellular $i$-chain on $\underline{\underline{\mathbf{V}}}$} if

\begin{itemize}
\item \emph{supp}$(\alpha)$ is a finite subset of $\mathcal{F}_i^{\text{o}}$.

\item $\alpha(F)\in\mathbf{V}^{I_F}$ for all $F\in\mathcal{F}_i^{\text{o}}$.

\item $\alpha(\sigma(F))=-\alpha(F)$ for all $F\in\mathcal{F}_i^o$.
\end{itemize}

\noindent Let $C_i(\Delta,\mathbf{V})$ denote the set of all oriented cellular $i$-chain on $\underline{\underline{\mathbf{V}}}$

\end{definition}

\noindent Note that each $C_i(\Delta,\mathbf{V})$ is a $(G,\mathcal{H})$-bimodule, where the $\mathcal{H}$-action is given by $$(x\cdot \alpha)(F,c)=x\alpha(F,c)$$ while the $G$-action is given by $$(g\cdot \alpha)(F,c)=g\alpha(g^{-1}F,g^{-1}c)$$ where $g^{-1}c$ is the orientation given on the edges by: if $\overset{\rightarrow}{e}$ is an oriented edge, then $o(g^{-1}\cdot\overset{\rightarrow}{e})=g^{-1}o(\overset{\rightarrow}{e})$ and $t(g^{-1}\cdot\overset{\rightarrow}{e})=g^{-1}\cdot t(\overset{\rightarrow}{e})$, where $o,t$ denote the \emph{origin} and \emph{target} of $\overset{\rightarrow}{e}$. This $G$-module structure will be crucial in the proof of Theorem \ref{letterthm: exact for star}.\\

\noindent For each $i=0,\dots,d-1$ there exists a map $\varepsilon_i:C_{i+1}(\Delta,\mathbf{V})\to C_i(\Delta,\mathbf{V})$, where $$\varepsilon_i(\alpha)(F,c):=\underset{\underset{F\text{a face of }F'}{F'\in\mathcal{F}_i}}{\sum}{\alpha(F',c\uparrow^{F'})}$$

\noindent There is also a map $\delta:C_0(\Delta,\mathbf{V})\to\mathbf{V},\alpha\mapsto\underset{v\in\mathcal{F}_0}{\sum}{\alpha(v)}$, and it is easily checked that $\delta\circ\varepsilon_0=0$ and $\varepsilon_i\circ\varepsilon_{i+1}=0$ for all $0\leq i\leq d$.

\begin{definition}\label{defn: coefficient system}
The sequence $$0\longrightarrow C_d(\Delta,\mathbf{V})\underset{\varepsilon_{d-1}}{\longrightarrow} C_{d-1}(\Delta,\mathbf{V})\underset{\varepsilon_{d-2}}{\longrightarrow}\dots\underset{\varepsilon_{0}}{\longrightarrow} C_0(\Delta,\mathbf{V})\underset{\delta}{\longrightarrow} \mathbf{V}\longrightarrow 0$$ is called the associated \emph{oriented chain complex} of the coefficient system $\underline{\underline{\mathbf{V}}}$.
\end{definition}

\noindent \textbf{Note:} If $\mathbf{V}=\mathbb{X}$, then the associated complex is exact (\cite[Remark 3.1(1)]{gorenstein}).  In general, this need not be true, but this will be the case that we focus on.

\subsection{Local coefficient systems}

Fix a set of vertices $\mathcal{X}$ in $\Delta$, and recall that for any facet $F$ of $\Delta$, we say that $F$ \emph{lies in $\mathcal{X}$} (or $F\in\mathcal{X}$) if all vertices of $F$ lie in $\mathcal{X}$. We will assume that $\mathcal{X}$ can be realised as a set of $j$-facets for some $j\leq d$, i.e. for every vertex $v\in\mathcal{X}$, there exists $F\in\mathcal{F}_j$ lying in $\mathcal{X}$ with $v\in F$.\\  

\noindent\textbf{Note:} It is possible for facets of dimension greater than $j$ to lie in $\mathcal{X}$, i.e. if all $j$-faces of a larger facet lie in $\mathcal{X}$. In most cases, we will take $j=d$ anyway, so this discrepancy will not pose a problem.\\

\noindent For each $i=0,\dots,d$, set $$\mathcal{F}_i(\mathcal{X}):=\{F\in\mathcal{F}_i:F\text{ lies in }\mathcal{X}\}$$ and define the \emph{local coefficient system} with respect to $\mathcal{X}$ to be the subcollection of $\underline{\underline{\mathbf{V}}}$ defined by $$\underline{\underline{\mathbf{V}}}(\mathcal{X}):=\{V^{I_F}:F\text{ lies in }\mathcal{X}\}=\{V^{I_F}:F\in\mathcal{F}_i(\mathcal{X})\text{ for some }0\leq i\leq j\}$$ We can define the set of oriented $i$-chains on the local coefficient system $\underline{\underline{\mathbf{V}}}(\mathcal{X})$ by  $$C_i(\mathcal{X},\mathbf{V}):=\{\alpha\in C_i(\Delta,\mathbf{V}):\alpha(F,c)=0\text{ if }F\notin\mathcal{F}_i(\mathcal{X})\}$$

\noindent This is a $\mathcal{H}$-submodule of $C_i(\Delta,\mathbf{V})$, and if $H$ is a subgroup of $G$ that permutes the $i$-facets in $\mathcal{X}$, then the $G$-action on $C_i(\Delta,\mathbf{V})$ restricts to an $H$-action on $C_i(\mathcal{X},\mathbf{V})$.\\ 

\noindent It is clear from the definition of the maps $\varepsilon_i$ that $\varepsilon_i\left(C_{i+1}(\mathcal{X},\mathbf{V})\right)\subseteq C_i(\mathcal{X},\mathbf{V})$. Thus we can define the \emph{oriented chain complex} of the local coefficient system $\underline{\underline{\mathbf{V}}}(\mathcal{X})$ to be the sequence 
\begin{equation}\label{eqn: local coefficient system}
0\longrightarrow C_d(\mathcal{X},\mathbf{V})\underset{\varepsilon_{d-1}}{\longrightarrow} C_{d-1}(\mathcal{X},\mathbf{V})\underset{\varepsilon_{d-2}}{\longrightarrow}\dots\underset{\varepsilon_0}{\longrightarrow} C_0(\mathcal{X},\mathbf{V})\underset{\delta}{\longrightarrow} S(\mathcal{X})\longrightarrow 0
\end{equation}

\noindent where $S(\mathcal{X}):=\delta(C_0(\mathcal{X},\mathbf{V}))$. For convenience, we will refer to this sequence as the \emph{local oriented chain complex} with respect to $\mathcal{X}$.\\

\noindent We define $C_i^{(n)}(\Delta,\mathbf{V}):=C_i(\Delta_n,\mathbf{V})$, and $S_n:=S(\Delta_n)$. Then the resulting local oriented chain complex

\begin{equation}\label{eqn: m-local coefficient system}
0\longrightarrow C_d^{(n)}(\Delta,\mathbf{V})\underset{\varepsilon_{d-1}}{\longrightarrow} C_{d-1}^{(n)}(\Delta,\mathbf{V})\underset{\varepsilon_{d-2}}{\longrightarrow}\dots\underset{\varepsilon_0}{\longrightarrow} C_0^{(n)}(\Delta,\mathbf{V})\underset{\delta}{\longrightarrow} S_n\longrightarrow 0
\end{equation}

\noindent is called the \emph{oriented chain complex to degree $m$}.\\




\noindent Recall the statement of Conjecture \ref{conj: exactness 1}, that the local oriented chain complex (\ref{eqn: local coefficient system}) is exact whenever \textbf{(A)} $\mathcal{X}$ is an $I$-invariant complete region of $\Delta$, or \textbf{(B)} $\mathcal{X}$ consists of a single facet of $C$.\\

\noindent The argument in \cite[Lemma 2.2]{torsion} proves that Conjecture \ref{conj: exactness 1} holds when $G$ has rank 1, but in general it remains unsolved. In this section, we will explore its consequences.

\subsection{Acyclicity in degree 0}

Ultimately, we want to prove that the canonical morphism $\mathbb{X}^*\to\mathcal{H}$ is a surjection, as stated in Conjecture \ref{sur} (or the hypothesis \textbf{(sur)} in \cite{torsion}). Equivalently, we want to show that $\mathcal{H}$ is a direct summand of $\mathbb{X}$ as an $\mathcal{H}$-module. 

Following the argument in \cite{torsion}, we will approach a stronger statement, that $Ext_{\mathcal{H}}^i(\mathbb{X}/\mathcal{H},\mathcal{H})=0$ for all $i>0$, i.e. the $\mathcal{H}$-module $\mathbb{X}/\mathcal{H}$ is \emph{$\ast$-acyclic} in the sense of section \ref{subsec: acyclic modules}.\\

\noindent First note that for any $x\in\mathbb{X}$, $x=\delta(\alpha)$ for some $\alpha\in C_0(\Delta,\mathbb{X})$ by exactness of the global oriented chain complex. Since supp$(\alpha)\subseteq V(\Delta)$ is finite, there must exist $n\in\mathbb{N}$ such that supp$(\alpha)\subseteq\Delta_n$, and hence $x\in S_n$, i.e. $\mathbb{X}=\underset{n\in\mathbb{N}}{\bigcup}{S_n}$. 

Therefore, writing $\mathbb{X}/\mathcal{H}=\underset{n}{\varinjlim}\text{ }S_n/\mathcal{H}$, we see using Proposition \ref{propn: direct limit} that to prove $\mathbb{X}/\mathcal{H}$ is $\ast$-acyclic, it remains to show that $S_n/\mathcal{H}$ and $S_{n+1}/S_n$ are $\ast$-acyclic for all $n\in\mathbb{N}$.\\

\noindent Let us first suppose that $n=0$. Then $S_0=\delta(C_0(\Delta_0,\mathbb{X}))$, and since chains in $C_0(\Delta_0,\mathbb{X})$ have support only on the hyperspecial chamber $C=\{v_0,\dots,v_{d-1}\}$, it follows that $$S_0=\mathbb{X}^{K_1^{(0)}}+\dots+\mathbb{X}^{K_1^{(d-1)}}=\mathbb{X}^{I_{v_0}}+\dots+\mathbb{X}^{I_{v_{d-1}}}$$

\begin{lemma}\label{lem: base case}
If we assume Conjecture \ref{conj: exactness 1}, then $S_0/\mathcal{H}$ is a $\ast$-acyclic $\mathcal{H}$-module.
\end{lemma}

\begin{proof}

For any $i=0,\dots,d-1$, we know that $\mathbb{X}^{I_{v_i}}/\mathcal{H}\cong (\mathbb{X}_{v_i}/\mathcal{H}_{v_i})\otimes_{\mathcal{H}_{v_i}}\mathcal{H}$ is $\ast$-acyclic by Lemma \ref{lem: acyclic induction}. So fixing $n\in\mathbb{N}$ with $0<n<d$, assume for induction that for all subsets $J'\subseteq\{0,\dots,d-1\}$ of size $n-1$, the $\mathcal{H}$-module $\left(\underset{j'\in J'}{\sum}{\mathbb{X}^{I_{v_{j'}}}}\right)/\mathcal{H}$ is $\ast$-acyclic.\\

\noindent Fix a subset $J\subseteq\{0,\dots,d-1\}$ of size $n$, choose any $i\in J$, and let $J':=J\backslash\{i\}$.  Setting $M:=\mathbb{X}^{I_{v_i}}\cap\underset{j'\in J'}{\sum}{\mathbb{X}^{I_{v_{j'}}}}$, consider the short exact sequence of $\mathcal{H}$-modules $$0\to \left(\underset{j'\in J'}{\sum}{\mathbb{X}^{I_{v_{j'}}}}\right)/\mathcal{H}\to \left(\underset{j\in J}{\sum}{\mathbb{X}^{I_{v_{j}}}}\right)/\mathcal{H}\to\mathbb{X}^{I_{v_i}}/M\to 0$$

\noindent Since we know by induction that $\left(\underset{j'\in J'}{\sum}{\mathbb{X}^{I_{v_{j'}}}}\right)/\mathcal{H}$ is $\ast$-acyclic, it remains to prove that $\mathbb{X}^{I_{v_i}}/M$ also is, and it will follow from a long exact sequence argument that $\left(\underset{j\in J}{\sum}{\mathbb{X}^{I_{v_{j}}}}\right)/\mathcal{H}$ is $\ast$-acyclic as required.

In fact, we will prove that $M$ is equal to $N:=\underset{e\in S}{\sum}{\mathbb{X}^{I_e}}$, where $S$ is the set of all oriented edges of $C$ with target $v_i$, and whose origin lien in $J'$. It will follow from Lemma \ref{lem: adjacent acyclicity} that $\mathbb{X}^{I_{v_i}}/M$ is $\ast$-acyclic as required.\\

\noindent For any edge $e\in S$, since $e=\{v_i,v_k\}$ for some $k\in J'$, we know that $\mathbb{X}^{I_{e}}\subseteq\mathbb{X}^{I_{v_i}}\cap\mathbb{X}^{I_{v_k}}\subseteq M$, so clearly $N\subseteq M$. On the other hand, if $x\in M=\mathbb{X}^{I_{v_i}}\cap\underset{j'\in J'}{\sum}{\mathbb{X}^{I_{v_{j'}}}}$, then we can write $x=\underset{j'\in J'}{\sum}x_{j'}$ for some $x_{j'}\in\mathbb{X}^{I_{v_{j'}}}$.\\

\noindent Define $\alpha\in C_0(\Delta,\mathbb{X})$ by $$\alpha(v):=\begin{cases} -x_{j'} & v=v_{j'}\text{ for some }j'\in J\\ x & v=v_i \\ 0 & \text{otherwise}\end{cases}$$ and we see immediately that $\delta(\alpha)=\underset{v\in V(\Delta)}{\sum}{\alpha(v)}=x-\underset{j'\in J'}{\sum}{x_{v_{j'}}}=0$, so by exactness of the oriented chain complex, we know that $\alpha=\varepsilon_0(\beta)$ for some $\beta\in C_1(\Delta,\mathbb{X})$.\\

\noindent Moreover, the set $F:=\{v_j:j\in J\}$ of vertices forms an $n$-face of $C$, and clearly $\alpha$ is non-zero only on the vertices of $F$. So if $\mathcal{X}:=F$, then $\alpha\in C_0(\mathcal{X},\mathbb{X})$. But clearly $\mathcal{X}$ satisfies \textbf{(B)}, so applying Conjecture \ref{conj: exactness 1} it follows that we can choose $\beta$ to lie in $C_1(\mathcal{X},\mathbb{X})$, i.e. we may assume that $\beta$ is non-zero only on edges of $F$. 

In particular, the only oriented edges with target $v_i$ which are not killed by $\beta$ lie in $S$, so $x=\alpha(v_i)=\varepsilon_0(\beta)(v_i)=\underset{e\in S}{\sum}{\beta(e)}\in N$ as required.\end{proof}

\noindent\textbf{Note:} This lemma is the only part of the proof of Theorem \ref{letterthm: implies sur} that uses hypothesis \textbf{(B)} from Conjecture \ref{conj: exactness 1}.\\

\noindent We will now proceed by induction on $n$, to prove that $S_n/\mathcal{H}$ is $\ast$-acyclic for all $n\in\mathbb{N}$. Suppose that $S_{n}/\mathcal{H}$ is $\ast$-acyclic for some $n\geq 0$. Utilising hypothesis \textbf{(A)} of Conjecture \ref{conj: exactness 1}, we will show that $S_{n+1}/S_n$ is also $\ast$-acyclic, and since we have a short exact sequence $$0\to S_n/\mathcal{H}\to S_{n+1}/\mathcal{H}\to S_{n+1}/S_n\to 0$$ it will follow that $S_{n+1}/\mathcal{H}$ is also $\ast$-acyclic, and applying induction and Proposition \ref{propn: direct limit} it will follow that $\mathbb{X}/\mathcal{H}=\underset{n}{\varprojlim}\text{ }S_n/\mathcal{H}$ is $\ast$-acyclic as we require.

\subsection{Proof of Theorem \ref{letterthm: implies sur}}

We will now use Conjecture \ref{conj: exactness 1} to argue inductively that $S_{n+1}/S_n$ is a $\ast$-acyclic module. Consider first the oriented chain complexes to degree $n+1$ and $n$ with coefficients in $\mathbb{X}$, as given in (\ref{eqn: m-local coefficient system}). Since we are assuming these sequences are exact by Conjecture \ref{conj: exactness 1}, we may quotient them to get an exact sequence $$0\longrightarrow \frac{C_d^{(n+1)}(\Delta,\mathbb{X})}{C_d^{(n)}(\Delta,\mathbb{X})}\underset{\varepsilon_{d-1}}{\longrightarrow} \frac{C_{d-1}^{(n+1)}(\Delta,\mathbb{X})}{C_{d-1}^{(n)}(\Delta,\mathbb{X})}\underset{\varepsilon_{d-2}}{\longrightarrow}\dots\underset{\varepsilon_0}{\longrightarrow}\frac{C_{0}^{(n+1)}(\Delta,\mathbb{X})}{C_{0}^{(n)}(\Delta,\mathbb{X})}\underset{\delta}{\longrightarrow} \frac{S_{n+1}}{S_n} \longrightarrow 0$$

\noindent But $$C_{0}^{(n+1)}(\Delta,\mathbb{X})/C_{0}^{(n)}(\Delta,\mathbb{X})=\left\{\alpha+C_{0}^{(n)}(\Delta,\mathbb{X}):\alpha(v)=0\text{ if }v\notin\Delta_{n+1}\backslash\Delta_n\right\}\cong C_0(\Delta_{n+1}\backslash\Delta_n,\mathbb{X})$$ via $\alpha+C_{0}^{(n)}(\Delta,\mathbb{X})\mapsto \alpha|_{\Delta_{n+1}/\Delta_n}$.\\


\noindent Fix an $I$-invariant set of vertices $\mathcal{B}\subseteq\Delta_{m+1}\backslash\Delta_m$, and clearly $C_0(\mathcal{B},\mathbb{X})$ is an $\mathcal{H}$-submodule of $C_0(\Delta_{m+1}\backslash\Delta_m,\mathbb{X})$. Moreover, we can associate a larger region $\mathcal{X}_{\mathcal{B}}$ of $\Delta$ to $\mathcal{B}$, defined as the union of the vertices of
\begin{itemize}
\item $\Delta_n$ and 

\item all chambers $D\in\Delta$, containing a vertex in $\mathcal{B}$, with $d(D,C)=n+1$.
\end{itemize}

\noindent Clearly $\Delta_m\subseteq\mathcal{X}_\mathcal{B}\subseteq\Delta_{n+1}$, so $C_0^{(n)}(\Delta,\mathbb{X})\subseteq C_0(\mathcal{X}_{\mathcal{B}},\mathbb{X})\subseteq C_0^{(n+1)}(\Delta,\mathbb{X})$.

\begin{lemma}\label{lem: X_B difference}
$\mathcal{X}_\mathcal{B}$ is an $I$-invariant complete region of $\Delta$ with $\mathcal{X}_\mathcal{B}\backslash\Delta_n=\mathcal{B}$. Moreover $$C_0(\mathcal{X}_{\mathcal{B}},\mathbb{X})/C_0^{(n)}(\Delta,\mathbb{X})\cong C_0(\mathcal{B},\mathbb{X})$$ as $\mathcal{H}$-modules.
\end{lemma}

\begin{proof}

Clearly $\Delta_{n}$ is $I$-invariant, and for any chamber $D$ with $d(D,C)=n+1$ containing $v\in\mathcal{B}$, it follows from Lemma \ref{lem: chamber for facet} that $D=C(v)$ is the unique chamber of $\Delta$ containing $v$ of distance $n+1$ from $C$, and for any $g\in I$, $g\cdot D=g\cdot C(v)=C(g\cdot v)$. So since $g\cdot v\in\mathcal{B}$ it follows that $g\cdot D\in\mathcal{X}_\mathcal{B}$, and hence $\mathcal{X}_\mathcal{B}$ is $I$-invariant.\\

\noindent Furthermore, $\mathcal{B}\subseteq\mathcal{X}_\mathcal{B}\backslash\Delta_n$, and if $u$ is a vertex of $\mathcal{X}_\mathcal{B}$ that lies outside of $\Delta_n$, then by definition it must belong to a chamber $D$ containing some $v\in\mathcal{B}$ with $d(C,D)=n+1$.\\ 

\noindent Choose a minimal gallery $D=D_{n+1}\sim D_n\sim\dots\sim D_1\sim D_0=C$ and take $D'=D_n$. Then $D'$ is adjacent to $D$ and $d(D',C)=n$, so since $\mathcal{B}\cap\Delta_n=\varnothing$, we know that $v\notin D'$. 

So let $F$ be the codimension 1 facet of $\Delta$ adjacent to $D$ and $D'$. Then $F$ must contain every vector of $D$ besides $v$. So if $u\neq v$ then $u\in F\subseteq D'$, and hence $u\in\Delta_n$ -- contradiction.\\

\noindent Therefore, $u=v\in\mathcal{B}$, so $\mathcal{B}=\mathcal{X}_\mathcal{B}\backslash\Delta_n$ as required. Thus $\mathcal{X}_\mathcal{B}=\mathcal{B}\sqcup\Delta_n$, and it follows that the $\mathcal{H}$-module map $C_0(\mathcal{X}_{\mathcal{B}},\mathbb{X})\to C_0(\mathcal{B},\mathbb{X}),\alpha\mapsto\alpha|_{\mathcal{B}}$ is surjective with kernel $\{\alpha\in C_0(\mathcal{X}_{\mathcal{B}},\mathbb{X}):\alpha(v)=0\text{ for all }v\notin\Delta_n\}=C_0^{(n)}(\Delta,\mathbb{X})$. \end{proof}

\begin{lemma}\label{lem: exactness application}
If we assume Conjecture \ref{conj: exactness 1}, then $\varepsilon_0(C_1^{(n+1)}(\Delta,\mathbb{X}))\cap C_0(\mathcal{X}_\mathcal{B},\mathbb{X})=\varepsilon_0(C_1(\mathcal{X}_\mathcal{B},\mathbb{X}))$.
\end{lemma}

\begin{proof}

Clearly $\varepsilon_0(C_1(\mathcal{X}_\mathcal{B},\mathbb{X}))\subseteq\varepsilon_0(C_1^{(n+1)}(\Delta,\mathbb{X}))\cap C_0(\mathcal{X}_\mathcal{B},\mathbb{X})$, and given $\alpha\in \varepsilon_0(C_1^{(n+1)}(\Delta,\mathbb{X}))\cap C_0(\mathcal{X}_\mathcal{B},\mathbb{X})$, it is clear that $\delta(\alpha)=0$, and hence $\alpha\in\ker(\delta:C_0(\mathcal{X}_\mathcal{B},\mathbb{X})\to S(\mathcal{X}_\mathcal{B}))$.

But $\mathcal{X}_\mathcal{B}$ is $I$-invariant by Lemma \ref{lem: X_B difference}, so applying Conjecture \ref{conj: exactness 1}, this kernel is equal to $\varepsilon_0(C_1(\mathcal{X}_\mathcal{B},\mathbb{X}))$, so $\alpha\in\varepsilon_0(C_1(\mathcal{X}_\mathcal{B},\mathbb{X}))$ and equality must hold.\end{proof}

\noindent In light of this lemma, let $\pi$ be the quotient map $\pi:C_0^{(n+1)}(\Delta,\mathbb{X})\to C_0(\Delta_{n+1}/\Delta_n,\mathbb{X})$, and define $$E(\mathcal{B}):=\pi\left(\varepsilon_0(C_1(\mathcal{X}_\mathcal{B},\mathbb{X}))\right)=\pi\left(\varepsilon_0(C_1^{(n+1)}(\Delta,\mathbb{X}))\cap C_0(\mathcal{X}_\mathcal{B},\mathbb{X})\right)$$ Using Lemma \ref{lem: X_B difference}, it is clear that $E(\mathcal{B})$ is a submodule of $C_0(\mathcal{B},\mathbb{X})$, so we define $$A(\mathcal{B}):=C_0(\mathcal{B},\mathbb{X})/E(\mathcal{B})$$ and we see that $$A(\Delta_{n+1}\backslash\Delta_n)=C_0(\Delta_{n+1}\backslash\Delta_n,\mathbb{X})/\varepsilon_0(C_1^{(n+1}(\Delta,\mathbb{X}))\cong S_{n+1}/S_n$$ so we will prove that $A(\mathcal{B})$ is $\ast$-acyclic for all non-empty $I$-invariant subsets $\mathcal{B}\subseteq\Delta_{n+1}\backslash\Delta_n$.\\

\noindent\textbf{Note:} In the result below, we will assume that if $v\in\Delta_{n+1}\backslash\Delta_n$ and $g\in I$, then $v$ and $g\cdot v$ are not joined by an edge, i.e. no two vertices in an $I$-orbit in $\Delta_{n+1}\backslash\Delta_n$ are joined by an edge. This is true in rank 1, or if $G=SL_3(K)$ by Theorem \ref{thm: d-partite}, and it should hold true in general, but since we will not prove it here.

\begin{proposition}\label{propn: acyclic exact}
Let $\mathcal{B}\subseteq\Delta_{n+1}\backslash\Delta_{n}$ be a non-empty, $I$-invariant subset. If we assume Conjecture \ref{conj: exactness 1} then for any $I$-orbit $o\subseteq\mathcal{B}$, there exists a $\ast$-acyclic $\mathcal{H}$-module $L_o$ and a short exact sequence $$0\to A(\mathcal{B}\backslash o)\to A(\mathcal{B})\to L_o\to 0$$
\end{proposition}

\begin{proof}

For each $v\in o$, let $T_v$ be the set of all oriented edges $e$ of $\Delta$ with target $v$ and whose origin lies in $\mathcal{X}_{\mathcal{B}}$. Since no two vertices in $o$ are joined by an edge, it follows that if $u,v\in o$ and $u\neq v$, then $T_u$ and $T_v$ are disjoint as a set of edges, even ignoring orientation.\\


\noindent Let $N_v:=\underset{e\in T_v}{\sum}{\mathbb{X}^{I_e}}$, and $L_v:=\mathbb{X}^{I_v}/N_v$. Then $L_v$ is $\ast$-acyclic by Lemma \ref{lem: adjacent acyclicity}, so let $L_o:=\underset{v\in o}{\bigoplus}L_v$, and $L_o$ is also $\ast$-acyclic.\\

\noindent There is a surjection $\phi$ from $C_0(\mathcal{X}_\mathcal{B},\mathbb{X})$ to $L_o$ sending $\alpha$ to $\underset{v\in o}{\sum}(\alpha(v)+N_v)$. The kernel of $\phi$ contains $C_0(\Delta_n,\mathbb{X})$, so $\phi$ induces a map $\overline{\phi}:C_0(\mathcal{B},\mathbb{X})\to L_o$.

Moreover, $\ker(\phi)$ also contains $\varepsilon_0(C_1(\mathcal{X}_\mathcal{B},\mathbb{X}))$, because if $\beta\in C_1(\mathcal{X}_\mathcal{B},\mathbb{X})$ then $\beta(e)\in\mathbb{X}^{I_e}$ for all $e\in T$, so $\varepsilon_0(\beta)(v)\in N_v$ for all $v\in o$. Therefore, $E(\mathcal{B})\subseteq\ker(\overline{\phi})$, inducing a surjection $A(\mathcal{B})\to L_o$.\\

\noindent Using Lemma \ref{lem: X_B difference}, we know that $\mathcal{X}_\mathcal{B}=\mathcal{B}\sqcup\Delta_{n}$, and hence the vertices of $o$ are the only vertices of $\mathcal{X}_\mathcal{B}$ that lie outside of $\mathcal{X}_{\mathcal{B}\backslash o}$. It follows that there is a natural embedding $\iota$ from $C_0(\mathcal{X}_{\mathcal{B}\backslash o},\mathbb{X})$ to $C_0(\mathcal{X}_\mathcal{B},\mathbb{X})$, where we extend a chain $\alpha\in C_0(\mathcal{X}_{\mathcal{B}\backslash o},\mathbb{X})$ to $\mathcal{X}_{\mathcal{B}}$, sending each $v\in o$ to 0. So the image of $\iota$ lies in the kernel of $\phi$.\\

\noindent Moreover, if $\phi(\alpha)=0$ then $\alpha(v)\in N_v$ for each $v\in o$, so $\alpha(v)=\underset{e\in T_v}{\sum}b_e$ for some $b_e\in \mathbb{X}^{I_{e}}$. So we define $\beta\in C_1(\mathcal{X}_{\mathcal{B}},\mathbb{X})$ by

$$\beta(e):=\begin{cases}
b_e & e\in T\\ -b_e&\sigma(e)\in T \\0 &\text{otherwise}\end{cases}$$

\noindent\textbf{Note:} This makes sense since $T_v,T_u$ are disjoint when regarded as sets of unoriented edges for $u\neq v$, so $e\mapsto b_e$ is well defined.\\

\noindent Then $\varepsilon_0(\beta)(v)=\alpha(v)$ for all $v\in o$, so $\alpha-\varepsilon_0(\beta)\in C_0(\mathcal{X}_\mathcal{B},\mathbb{X})$ lies in the image of $C_0(\mathcal{X}_{\mathcal{B}\backslash o},\mathbb{X})$ under $\iota$. Moreover, since $\varepsilon_0(\beta)\in \varepsilon_0(C_1(\mathcal{X}_{\mathcal{B}},\mathbb{X}))$ maps to $E(\mathcal{B})$ under $\pi$, this implies that $\pi(\alpha)+E(\mathcal{B})=\pi(\alpha-\varepsilon_0(\beta))+E(\mathcal{B})$ lies in the image of the composition $$C_0(\mathcal{X}_{\mathcal{B}\backslash o},\mathbb{X})\overset{\iota}{\to} C_0(\mathcal{X}_\mathcal{B},\mathbb{X})\overset{\pi}{\to} C_0(\mathcal{B},\mathbb{X})\twoheadrightarrow A(\mathcal{B})$$ giving us an exact sequence $$C_0(\mathcal{X}_{\mathcal{B}\backslash o},\mathbb{X})\to A(\mathcal{B})\to L_o\to 0$$

\noindent Finally, $C_1(\mathcal{X}_{\mathcal{B}\backslash o},\mathbb{X})$ is contained in $C_1(\mathcal{X}_{\mathcal{B}},\mathbb{X})$, and it is clear that $\varepsilon_0(C_1(\mathcal{X}_{\mathcal{B}\backslash o},\mathbb{X}))$ maps to 0 under the first map in this sequence.

Moreover, if $\alpha\in C_0(\mathcal{X}_{\mathcal{B}\backslash o},\mathbb{X})$ maps to zero, then $\alpha+E(\mathcal{B})=0$ so $\alpha\in C_0(\mathcal{X}_{\mathcal{B}\backslash o},\mathbb{X})\cap \varepsilon_0(C_1(\mathcal{X}_\mathcal{B},\mathbb{X}))$, which is contained in $\varepsilon_0(C_1(\mathcal{X}_{\mathcal{B}\backslash o},\mathbb{X}))$ by Lemma \ref{lem: exactness application}. This gives an exact sequence $0\to A(\mathcal{B}\backslash o)\to A(\mathcal{B})\to L_o\to 0$ as required.\end{proof}

\begin{corollary}\label{cor: inductive step}
If we assume Conjecture \ref{conj: exactness 1}, then for any $I$-invariant subset $\mathcal{B}\subseteq\Delta_{n+1}\backslash\Delta_{n}$, $A(\mathcal{B})$ is $\ast$-acyclic. In particular, $A(\Delta_{n+1}\backslash\Delta_{n})\cong S_{n+1}/S_{n}$ is $\ast$-acyclic.
\end{corollary}

\begin{proof}

Let $M$ be the number of $I$-orbits in $\mathcal{B}$. If $M=0$ then $\mathcal{B}=\varnothing$ and $A(\mathcal{B})=0$. So we will assume that $M>0$ and apply induction on $M$.\\

\noindent By Proposition \ref{propn: acyclic exact}, for any $I$-orbit $o\subseteq\mathcal{B}$ there is a short exact sequence $$0\to A(\mathcal{B}\backslash o)\to A(\mathcal{B})\to L_o\to 0$$ for some $\ast$-acyclic module $L_o$. But $A(\mathcal{B}\backslash o)$ is $\ast$-acyclic by induction, so it follows from a long exact sequence argument that $A(\mathcal{B})$ must be $\ast$-acyclic as required.\end{proof}

\noindent We can now complete the proof of our main general result.\\

\noindent\emph{Proof of Theorem \ref{letterthm: implies sur}.} We know using Lemma \ref{lem: base case} that $S_0/\mathcal{H}$ is $\ast$-acyclic, and by Corollary \ref{cor: inductive step}, we know that $S_{n+1}/S_n$ is $\ast$-acyclic for all $m\geq 0$. If we suppose, for induction that $S_{n}/\mathcal{H}$ is $\ast$-acyclic for some $n\geq 0$, then considering the Ext-sequence associated with the short exact sequence of $\mathcal{H}$-modules $$0\to S_n/\mathcal{H}\to S_{n+1}/\mathcal{H}\to S_{n+1}/S_n\to 0$$ we see that $S_{n+1}/\mathcal{H}$ is also $\ast$-acyclic.\\

\noindent Therefore, $S_n/\mathcal{H}$ is $\ast$-acyclic for all $n\in\mathbb{N}$. Setting $M_n:=S_n/\mathcal{H}$, since $\underset{n}{\bigcup}S_n=\mathbb{X}$, we see that $\underset{n}{\varinjlim}{\text{ }M_n}=\mathbb{X}/\mathcal{H}$. So since $M_n$ and $M_{n+1}/M_n\cong S_{n+1}/S_n$ are $\ast$-acyclic, it follows from Proposition \ref{propn: direct limit} that $\mathbb{X}/\mathcal{H}$ is also $\ast$-acyclic. In other words, $Ext^i(\mathbb{X}/\mathcal{H},\mathcal{H})=0$ for all $i>0$.

In particular, $Ext^1(\mathbb{X}/\mathcal{H},\mathcal{H})=0$ and thus the extension $0\to\mathcal{H}\to\mathbb{X}\to\mathbb{X}/\mathcal{H}\to 0$ is trivial, i.e. $\mathcal{H}$ is a direct summand of $\mathbb{X}$, and thus the canonical morphism $\mathbb{X}^*\to\mathcal{H}$ is surjective, which is precisely Conjecture \ref{sur}.\qed

\subsection{Shift invariance}\label{subsec: shift invariance}

In light of Theorem \ref{letterthm: implies sur}, the remainder of the paper will be dedicated to exploring how to prove Conjecture \ref{conj: exactness 1}. We will focus on proving exactness of the local oriented chain complex \ref{eqn: local coefficient system} for a complete, $I$-invariant region $\mathcal{X}$ of $\Delta$.\\

\noindent Moreover, the biggest challenge should lie in proving exactness of the sequence $C_1(\mathcal{X},\mathbb{X})\to C_0(\mathcal{X},\mathbb{X})\to\mathbb{X}$ (which is actually the only thing we needed in the proof of Theorem \ref{letterthm: implies sur}). So for the remainder of this section, we will make the following assumption on $\mathcal{X}$:

\begin{equation}
\tag{\textbf{E}}
    \text{\parbox{.78\textwidth}{$0\to C_d(\mathcal{X},\mathbb{X})\to \dots\to C_1(\mathcal{X},\mathbb{X})\to C_0(\Delta,\mathbb{X})$ is exact.}}
    \label{E}
\end{equation}









\begin{definition}
We say that a chain $\beta'\in C_1(\Delta,\mathbb{X})$ is a \emph{shift} of $\beta$ if there exists $\gamma\in C_2(\Delta,\mathbb{X})$ such that $\beta'=\beta+\varepsilon_1(\gamma)$.

Moreover, if $\mathcal{Y}$ is a set of chambers in $\Delta$, then we say $\beta'$ is a \emph{$\mathcal{Y}$-shift} of $\beta$ if the chain $\gamma$ is zero on all chambers outside $\mathcal{Y}$.
\end{definition}

\noindent\textbf{Note:} It follows from exactness of the oriented chain complex that $\beta'$ is a shift of $\beta$ if and only if $\varepsilon_0(\beta')=\varepsilon_0(\beta)$. Moreover, it follows from (\ref{E}) that if $\beta,\beta'\in C_1(\mathcal{X},\mathbb{X})$, then $\varepsilon_0(\beta)=\varepsilon_0(\beta')$ if and only if $\beta'$ is an $\mathcal{X}$-shift of $\beta$.\\

\noindent We will now explore some properties of the action of $G$ on chains. Fix an $I$-invariant complete region $\mathcal{X}$ and note that the $I$-action on $C_i(\Delta,\mathbb{X})$ preserves $C_i(\mathcal{X},\mathbb{X})$ for each $i$. 


\begin{definition}\label{defn: shift invariance}
If $\beta\in C_1(\Delta,\mathbb{X})$ and $\varepsilon_0(\beta)\in C_0(\mathcal{X},\mathbb{X})$, then for any subgroup $S$ of $I$, we say that $\beta$ is \emph{$(S,\mathcal{X})$-shift invariant} if for all $g\in S$, $$(g-1)\cdot\beta\in C_1(\Delta_0,\mathbb{X})+\varepsilon_1(C_2(\Delta,\mathbb{X}))$$
\end{definition}

\noindent Alternatively stated, a chain $\beta\in C_1(\Delta,\mathbb{X})$ is $(S,\mathcal{X})$-shift invariant if for all $g\in S$, $(g-1)\cdot \beta$ has a shift which lies in $C_1(\mathcal{X},\mathbb{X})$, and note that this is trivially true whenever $\beta\in C_1(\mathcal{X},\mathbb{X})$.\\ 

\noindent For convenience, let $M_\mathcal{X}:=C_1(\mathcal{X},\mathbb{X})+\varepsilon_1(C_2(\Delta,\mathbb{X}))$, and note that $M_\mathcal{X}$ is $I$-invariant.

\begin{definition}\label{defn: collapsible}
If $\mathcal{L}$ is an $I$-invariant region of $\Delta$ and $\mathcal{X}\subseteq\mathcal{L}$, we say that $\mathcal{X}$ is \emph{$\mathcal{L}$-collapsible} if $\beta\in M_\mathcal{X}$ for all $\beta\in C_1(\mathcal{L},\mathbb{X})$ such that $\beta$ is $(I,\mathcal{X})$-shift invariant and $\varepsilon_0(\beta)\in C_0(\mathcal{X},\mathbb{X})$.

We say that $\mathcal{X}$ is \emph{collapsible} if it it $\mathcal{L}$-collapsible for all $I$-invariant complete regions $\mathcal{X}\subseteq\mathcal{L}$.
\end{definition}

\begin{lemma}\label{lem: one way}
If the local oriented chain complex $$0\to C_d(\mathcal{X},\mathbb{X})\to\dots\to C_0(\mathcal{X},\mathbb{X})\to\mathbb{X}$$ is exact, then $\mathcal{X}$ is a collapsible region of $\Delta$.
\end{lemma}

\begin{proof}

Suppose $\mathcal{X}\subseteq\mathcal{L}$, $\beta\in C_1(\mathcal{L},\mathbb{X})$ is $(I,\mathcal{X})$-shift invariant, and $\varepsilon_0(\beta)\in C_0(\mathcal{X},\mathbb{X})$.

It follows from exactness of the local oriented chain complex that $\varepsilon_0(\beta)=\varepsilon_0(\beta')$ for some $\beta'\in C_1(\mathcal{X},\mathbb{X})$, and hence $\beta-\beta'\in\varepsilon_1(C_2(\Delta,\mathbb{X}))$, and $\beta\in M_\mathcal{X}$. Therefore, $\mathcal{X}$ is $\mathcal{L}$-collapsible, and since this holds for all $\mathcal{L}$, it follows that $\mathcal{X}$ is collapsible.\end{proof}

\noindent The converse to Lemma \ref{lem: one way} is not immediately obvious, because the definition of collapsible regions only refers to those chains $\beta$ that are shift invariant. However, the results below show that it does indeed hold.

\begin{proposition}\label{propn: loose end}
Suppose $k$ has characteristic $p$, $\mathcal{X}\subseteq\mathcal{L}$ are complete, $I$-invariant regions, and $\mathcal{X}$ is $\mathcal{L}$-collapsible. Then for all $\beta\in C_1(\mathcal{L},\mathbb{X})$ such that $\varepsilon_0(\beta)\in C_0(\mathcal{X},\mathbb{X})$, $\beta\in M_\mathcal{X}$.
\end{proposition}

\begin{proof}

\noindent Firstly, let $V$ be the set of all vertices in $\Delta$ that are adjacent to a vertex in $\mathcal{X}$, and let $$K:=\{g\in I:g\cdot v=v\text{ for all }v\in V\}$$ For any $\beta\in C_1(\mathcal{L},\mathbb{X})$, we know that $I_v\supseteq K$ for all $v\in\mathcal{X}$ by Proposition \ref{propn: adjacent chambers facet subgroup}, so since $g\cdot\varepsilon_0(\beta)(v)=\varepsilon_0(\beta)(v)$ for all $g\in I_v$, it follows that $(g\cdot\varepsilon_0(\beta))(v)=g\varepsilon_0(\beta)(g^{-1}\cdot v)=g\varepsilon_0(\beta)=\varepsilon_0(\beta)$ for all $v\in\mathcal{X}$.

In other words, $\varepsilon_0((g-1)\cdot\beta)(v)=0$ for all $g\in K$, $v\in\mathcal{X}$. So if we assume that $\varepsilon_0(\beta)\in C_0(\mathcal{X},\mathbb{X})$, then since $\mathcal{X}$ is $I$-invariant, and $K\subseteq I$, we know that $\varepsilon_0((g-1)\beta)=(g-1)\varepsilon_0(\beta)$ is zero outside of $\mathcal{X}$, and hence $\varepsilon_0((g-1)\beta)=0$. By exactness of the oriented chain complex, this means that $(g-1)\cdot\beta\in \varepsilon_1(C_2(\Delta,\mathbb{X}))\subseteq M_\mathcal{X}$. In particular, $\beta$ is $(K,\mathcal{X})$-shift invariant.\\

\noindent Now, $V$ is $I$-invariant, so $K$ is an open normal subgroup of $I$. In particular, $I/K$ is a finite $p$-group, and hence there exists a chain of subgroups $I=S_0\supseteq S_1\supseteq\dots\supseteq S_r=K$ such that each $S_i$ is normal in $I$ and $S_i/S_{i+1}$ has order $p$. Fix $i\geq 0$ and let us suppose that if $\beta\in C_1(\mathcal{L},\mathbb{X})$ is $(S_i,\mathcal{X})$-shift invariant and $\varepsilon_0(\beta)\in C_0(\mathcal{X},\mathbb{X})$, then $\beta\in M_\mathcal{X}$, which we know to hold if $i=0$ by the assumption that $\mathcal{X}$ is $\mathcal{L}$-collapsible.\\

\noindent Since $\beta$ is always $(S_i,\mathcal{X})$-shift invariant if $i=r$, we will assume that $i<r$, and prove that the same property holds if $\beta$ is $(S_{i+1},\mathcal{X})$-shift invariant. The result will follow from induction on $i$.\\

\noindent Since $S_i/S_{i+1}$ has order $p$, it is generated by a single element $g\in S_i$ such that $g^p\in S_{i+1}$. Using $(S_{i+1},\mathcal{X})$-shift invariance of $\beta$, we know that $(g^p-1)\cdot\beta\in M_\mathcal{X}$, i.e. $g^p\cdot\beta\equiv\beta$ (mod $M_\mathcal{X}$).

So, let $\beta_0:=\beta+g\cdot\beta+g^2\cdot\beta+\dots+g^{p-1}\cdot\beta\in C_1(\mathcal{L},\mathbb{X})$, then $g\cdot\beta_0\equiv\beta_0$ (mod $M_\mathcal{X}$). Also, note that since $k$ has characteristic $p$, we may write $\beta_0=(g-1)^{p-1}\cdot\beta$, so we will also define $\beta_j:=(g-1)^{p-1-j}\cdot\beta$ for all $j=0,\dots,p$. Note that $\varepsilon_0(\beta_j)\in C_0(\mathcal{X},\mathbb{X})$ for all $j$.\\

\noindent Also note that for any $h\in S_{i+1}$, $k<p$, $g^{-k}hg^k\in S_{i+1}$, so $g^{-k}hg^k\beta\equiv\beta$ (mod $M_\mathcal{X}$). Therefore
\begin{align*}
h\cdot\beta_j&=h(g-1)^{p-1-j}\cdot\beta=\underset{1\leq k\leq p-1-j}{\sum}\binom{p-1-j}{k}(-1)^{p-1-j-k}hg^k\cdot\beta\\&=\underset{1\leq k\leq p-1-j}{\sum}\binom{p-1-j}{k}(-1)^{p-1-j-k}g^kg^{-k}hg^k\cdot\beta\\&\equiv\underset{1\leq k\leq p-1-j}{\sum}\binom{p-1-j}{k}(-1)^{p-1-j-k}g^k\cdot\beta\equiv (g-1)^{p-1-j}\cdot\beta\equiv\beta_j\text{ (mod }M_\mathcal{X})
\end{align*}

\noindent Therefore, $\beta_j$ is $(S_{i+1},\mathcal{X})$-shift invariant for all $j$. But $g\cdot\beta_0\equiv\beta_0$ (mod $M_\mathcal{X}$), $M_\mathcal{X}$ is $I$-invariant and $S_i$ is generated by $S_{i+1}$ and $g$, so it follows that $\beta_0$ is $(S_i,\mathcal{X})$-shift invariant.\\

\noindent Let us suppose, for a second induction, that $\beta_j$ is $(S_i,\mathcal{X})$-shift invariant for some $j\geq 0$. Then by assumption, $\beta_j\in M_\mathcal{X}$. But $\beta_j=(g-1)\cdot\beta_{j+1}$, so $(g-1)\cdot\beta_{j+1}\in M_\mathcal{X}$, and it follows that $\beta_{j+1}$ is $(S_i,\mathcal{X})$-shift invariant. Applying the second induction we see that $\beta_{p-1}=\beta$ is $(S_i,\mathcal{X})$-shift invariant as required.\end{proof}

\begin{corollary}\label{cor: extended regions}
Suppose $k$ has characteristic $p$, $\mathcal{X}$ is a complete, $I$-invariant, collapsible region of $\Delta$, and (\ref{E}) is satisfied. Then the local oriented chain complex $$0\to C_d(\mathcal{X},\mathbb{X})\to\dots\to C_1(\mathcal{X},\mathbb{X})\to C_0(\mathcal{X},\mathbb{X})\to\mathbb{X}$$ is exact.
\end{corollary}

\begin{proof}

Using (\ref{E}), we know that $0\to C_d(\mathcal{X},\mathbb{X})\to\dots\to C_1(\mathcal{X},\mathbb{X})\to C_0(\Delta,\mathbb{X})$ is exact, so it remains to prove that $C_1(\mathcal{X},\mathbb{X})\to C_0(\mathcal{X},\mathbb{X})\to\mathbb{X}$ is exact, i.e. to show that for all $\beta\in C_1(\Delta,\mathbb{X})$ such that $\varepsilon_0(\beta)\in C_0(\mathcal{X},\mathbb{X})$, we can choose $\beta'\in C_1(\mathcal{X},\mathbb{X})$ such that $\varepsilon_0(\beta')=\varepsilon_0(\beta)$, i.e. $\beta\in C_1(\mathcal{X},\mathbb{X})+\varepsilon_1(C_2(\Delta,\mathbb{X}))=M_\mathcal{X}$.

But since $\beta$ has finite support, we know that $\beta\in C_1(\Delta_n,\mathbb{X})$ for some $n\in\mathbb{N}$ with $\mathcal{X}\subseteq\Delta_n$. But since $\mathcal{X}$ is collapsible, we know that it is $\Delta_n$-collapsible, so applying Proposition \ref{propn: loose end}, we see that $\beta\in M_\mathcal{X}$ as required.\end{proof}

\noindent So to prove Conjecture \ref{conj: exactness 1}, it remains to show that all complete, $I$ invariant regions of $\Delta$ are collapsible. In the following section, we will explore how to prove this when $G=SL_3(K)$, completing the proof of Theorem \ref{letterthm: exact for star}

\section{Orbits in coefficient systems}\label{sec: orbits}

From now on, we will assume that $G=SL_3(K)$, and hence $\Delta=\widetilde{\Delta}(G)$ is the $\widetilde{A}_2$-Bruhat-Tits building. Again, let $C=\{v_0,v_1.v_2\}$ be the hyperspecial chamber in $\Delta$, where $v_0=[\mathcal{O}^3]$ is the hyperspecial vertex. We will also assume from now on that $k$ has characteristic $p$.

\subsection{Approaching Conjecture \ref{conj: exactness 1}}

Recall from the statement of Conjecture \ref{conj: exactness 1} in rank 2 that we need to prove that the local oriented chain complex 

\begin{equation}\label{eqn: rank 2 coefficient system}
0\to C_2(\mathcal{X},\mathbb{X})\underset{\varepsilon_1}{\to} C_1(\mathcal{X},\mathbb{X})\underset{\varepsilon_0}{\to} C_0(\mathcal{X},\mathbb{X})\underset{\delta}{\to}\mathbb{X}
\end{equation} 

\noindent is exact whenever \textbf{(A)} $\mathcal{X}$ is an $I$-invariant complete region of $\Delta$, or \textbf{(B)} $\mathcal{X}$ consists of the vertices of a single face of $C$. Note that if $\mathcal{X}=C=\Delta_0$, then $\mathcal{X}$ satisfies \textbf{(A)} and \textbf{(B)}.

\begin{lemma}\label{lem: (A) implies (B)}
If $\mathcal{X}$ consists of a single vertex or edge of $C$, then the local oriented chain complex for $\mathcal{X}$ is exact.
\end{lemma}

\begin{proof}

If $\mathcal{X}=\{v\}$ for $v$ a vertex of $C$, then $C_2(\mathcal{X},\mathbb{X})=C_1(\mathcal{X},\mathbb{X})=0$ and $C_0(\mathcal{X},\mathbb{X})$ consists of all functions from $\mathcal{X}=\{v\}$ to $\mathbb{X}^{I_v}$, so clearly $\delta$ has kernel $0=\varepsilon_0(C_1(\mathcal{X},\mathbb{X}))$ on $C_0(\mathcal{X},\mathbb{X})$.\\

\noindent If $\mathcal{X}=\{e\}$, for $e$ an oriented edge of $C$, and we will assume without loss of generality that $o(e)=v_1, t(e)=v_2$, then $C_2(\mathcal{X},\mathbb{X})=0$, and $C_1(\mathcal{X},\mathbb{X})$ consists of all functions from $\{e\}$ to $\mathbb{X}^{I_{e}}$. So clearly $\varepsilon_0$ is injective when restricted to $C_1(\mathcal{X},\mathbb{X})$. Thus we only need to prove that $\varepsilon_0(C_1(\mathcal{X},\mathbb{X}))$ coincides with the kernel of $\delta$ on $C_0(\mathcal{X},\mathbb{X})$.

But if $\alpha\in C_0(\mathcal{X},\mathbb{X})$ and $\delta(\alpha)=0$ then $\alpha(v_1)+\alpha(v_2)=0$. So $\alpha(v_1)=-\alpha(v_2)\in\mathbb{X}^{I_{v_1}}\cap\mathbb{X}^{I_{v_2}}=\mathbb{X}^{\langle I_{v_1},I_{v_2}\rangle}$.\\ 

\noindent We will see in Proposition \ref{propn: generating_subgroup2} that $\langle I_{v_1},I_{v_2}\rangle=I_e$, so $\alpha(v_1)\in\mathbb{X}^{I_e}$. So define $\beta\in C_1(\mathcal{X},\mathbb{X})$ by $\beta(e):=\alpha(v_1)$ and $\beta(\sigma(e)):=-\alpha(v_1)=\alpha(v_2)$, and we see that $\alpha=\varepsilon_0(\beta)$ as required.\end{proof}





\noindent This lemma completes the proof when $\mathcal{X}$ satisfies \textbf{(B)}, so we can assume from now on that $\mathcal{X}$ is a complete, $I$-invariant region of $\Delta$, and hence $\Delta_n\subseteq\mathcal{X}\subseteq\Delta_{n+1}$ for some $n\in\mathbb{N}$.

\begin{proposition}\label{propn: exactness1}
For any complete region $\mathcal{X}$ of $\Delta$, the chain complex $$0\to C_2(\mathcal{X},\mathbb{X})\to C_1(\mathcal{X},\mathbb{X})\to C_0(\Delta,\mathbb{X})$$ is exact.
\end{proposition}

\begin{proof}
Since $\varepsilon_1$ is injective, we only need to prove that $C_2(\mathcal{X},\mathbb{X})\to C_1(\mathcal{X},\mathbb{X})\to C_0(\Delta,\mathbb{X})$ is exact.\\

\noindent Since im$(\varepsilon_1)=\ker(\varepsilon_0)$ globally, we know that there exists a unique $\gamma\in C_2(\Delta,\mathbb{X})$ such that $\beta=\varepsilon_1(\gamma)$. So it remains only to prove that $\gamma\in C_2(\mathcal{X},\mathbb{X})$, i.e. for any oriented chamber $(D,c)$ of $\Delta$, if $\gamma(D,c)\neq 0$ then $D\in\mathcal{X}$.\\

\noindent Suppose for contradiction that there exists a chamber $D\notin\mathcal{X}$ such that $\gamma(D)\neq 0$. Suppose further that $n:=d(D,C)$ is maximal among all chambers $D\in\Delta\backslash\mathcal{X}$ such that $\gamma(D)\neq 0$.\\ 

\noindent By Corollary \ref{cor: adjacent to facet} there exists an edge $e$ of $D$ such that for every chamber $E$ of $\Delta$ containing $e$ as an edge, $d(E,C)=n+1$ if $E\neq D$, and hence $\gamma(E)=0$. In other words, $e$ is an exterior edge of $D$ in the sense of section \ref{subsec: border}.

But for any such edge $e$, and any orientation $c$ of $e$, $$\beta(e,c)=\varepsilon_1(\gamma)(e,c)=\underset{\underset{e\text{ an edge of }E}{E\in\mathcal{F}_2}}{\sum}{\gamma(E,c\uparrow^E)}=\gamma(D,c\uparrow^D)$$ so since $\gamma(D,c\uparrow^D)\neq 0$, it follows that $e\in\mathcal{X}$. Applying Lemma \ref{lem: unique edge} gives us that $D\in\mathcal{X}$ -- contradiction.\end{proof}

\noindent\textbf{Note:} With some small tweaks we expect that this proof can be generalised to show that $C_d(\mathcal{X},\Delta)\to C_{d-1}(\mathcal{X},\Delta)\to C_{d-2}(\mathcal{X},\Delta)$ is exact in full generality. Proving Theorem \ref{thm: peak} in types $\widetilde{B}_2$ and $\widetilde{G}_2$ would be enough to carry the proof over to all groups of rank 2.\\

\noindent This result means that Assumption (\ref{E}) from section \ref{subsec: shift invariance} is satisfied for $G=SL_3(K)$. So in light of Corollary \ref{cor: extended regions}, it remains only to prove that $\mathcal{X}$ is a collapsible region of $\Delta$, and Conjecture \ref{conj: exactness 1} for $SL_3(K)$ will follow.

\subsection{The key lemma}\label{subsec: key lemma}

\noindent From now on, fix a chain $\beta\in C_1(\Delta,\mathbb{X})$, and an $I$-invariant complete region $\mathcal{X}$ of $\Delta$. Choose $n\in\mathbb{N}$ such that $\Delta_n\subseteq\mathcal{X}\subseteq\Delta_{n+1}$, and we will make the following assumptions on $\beta$:

 \begin{itemize}

\item $\epsilon_0(\beta)$ is zero outside $\mathcal{X}$ (i.e. $\varepsilon_0(\beta)\in C_0(\mathcal{X},\mathbb{X})$).

\item $\beta$ is zero outside $\Delta_{n+1}$ (i.e. $\beta\in C_1(\Delta_{n+1},\mathbb{X}))$,

\item $\beta$ is $(I,\mathcal{X})$-shift invariant.
\end{itemize}

\noindent With these assumptions, we want to prove that $\beta\in M_\mathcal{X}=C_1(\mathcal{X},\mathbb{X})+\varepsilon_1(C_2(\Delta,\mathbb{X}))$, and it will follow that $\mathcal{X}$ is a $\Delta_{n+1}$-collapsible region of $\Delta$.\\

\noindent Recall from Definition \ref{defn: border} how we define the \emph{border} of $\Delta_{n+1}$, and recall from Lemma \ref{lem: border characterisation} that for each edge $e$ on the border, there is a unique chamber $D(e)\in\Delta_{n+1}$ containing $e$.\\

\noindent The following technical result will be crucial in our arguments to follow.

\begin{lemma}\label{lem: crucial}
Let $Y=\{e_1,\dots,e_r\}$ be a set of oriented edges on the border of $\Delta_{n+1}$, where $e_i\notin\mathcal{X}$ for all $i$, and let $A$ be a subgroup of $G$ such that 

\begin{enumerate}
\item If $g\cdot e_i=e_i$ for some $g\in A, 1\leq i\leq r$, then $g\in I_{e_i}$.

\item $A$ acts transitively on $Y$.

\item $\text{\emph{Stab}}_A(e_i)=\text{\emph{Stab}}_A(e_j)$ for $1\leq i,j\leq r$, so $N:=\text{\emph{Stab}}_A(e_1)$ is a normal subgroup of $A$.

\item $A/N=(A\cap I)/N$.

\end{enumerate}

\noindent Then there exists $\gamma\in C_2(\Delta_{n+1},\mathbb{X})$, where $\gamma(D)=0$ if $D\neq D(e_i)$ for some $1\leq i\leq r$, such that if $\beta':=\beta+\varepsilon_1(\gamma)$ then $\beta'(h\cdot e_i)=h\beta'(e_i)$ for all $h\in A$, $i=1,\dots,r$. 

In particular, $\{\beta'(e_1),\dots,\beta'(e_r)\}\subseteq\mathbb{X}$ forms a single $A$-orbit under the left action of $G$.
\end{lemma}

\begin{proof}

For convenience, set $D_i:=D(e_i)$ for each $i$. Since $N$ fixes every edge in $Y$, we know that $A/N=A\cap I/N$ acts transitively on $Y$. So for each $i=1,\dots,r$, choose $h_i\in A\cap I$ such that $e_i=h_i\cdot e_1$, and we can of course take $h_1=1$.

Then for any $h\in A$, $h\cdot e_1=e_j$ for some $j$. So $h\cdot e_1=h_j\cdot e_1$, and thus $h^{-1}h_j\in N$, i.e. $hN=h_jN$. Therefore, $A/N=\{h_1N,\dots,h_rN\}$.\\

\noindent Let $\beta_1=\beta$, and let $s_k^{(1)}:=\beta(e_k)$ for each $k$, and clearly $h_1s_1^{(1)}=s_1^{(1)}$. So suppose for induction that $\beta_1,\dots,\beta_{i-1}\in C_1(\Delta_{n+1},\mathbb{X})$ are defined for some $i\leq r$, and for each $j<i$:
\begin{itemize}
\item $\beta_j=\beta+\varepsilon_1(\gamma_j)$ for some $\gamma_j\in C_2(\Delta,\mathbb{X})$,

\item $\gamma_j$ is zero on all chambers outside $\{D_1,\dots,D_{j}\}$, and 

\item if $s_k^{(j)}:=\beta_j(e_k)$ then $s_k^{(j)}=h_ks_1^{(j)}$ for all $k\leq j<i$.
\end{itemize}

\noindent Now, using shift invariance we are assuming that for all $g\in S$, $g\cdot\beta-\beta\in C_1(\mathcal{X},\mathbb{X})+\varepsilon_1(C_2(\Delta,\mathbb{X}))$. Since it is clear that $g\cdot\varepsilon_1(\gamma_j)-\varepsilon_1(\gamma_j)=\varepsilon_1(g\cdot\gamma_j-\gamma_j)\in\varepsilon_1(C_2(\Delta,\mathbb{X}))$, we similarly have that $g\cdot\beta_j-\beta_j\in C_1(\mathcal{X},\mathbb{X})+\varepsilon_1(C_2(\Delta,\mathbb{X}))$.\\

\noindent In other words, taking $j=i-1$, for each $h\in I\cap A$ we can write $h\cdot \beta_{i-1}=\beta_{i-1}+z_h+\varepsilon_1(\gamma_h)$, where $z_h\in C_1(\mathcal{X},\mathbb{X})$, and $\gamma_h\in C_2(\Delta,\mathbb{X})$.\\ 

\noindent Moreover, since $\beta_{i-1}\in C_1(\Delta_{n+1},\mathbb{X})$, it follows that $h\cdot\beta_{i-1}\in C_1(\Delta_{n+1},\mathbb{X})$. So since $\beta_{i-1},h\cdot\beta_{i-1},z_h\in C_1(\Delta_{n+1},\mathbb{X})$ and $\varepsilon_0(h\cdot\beta_{i-1}-\beta_{i-1}-z_h)=0$, it follows from Proposition \ref{propn: exactness1} that $\gamma_h\in C_2(\Delta_{n+1},\mathbb{X})$. \\





\noindent Therefore, since $\{e_1,\dots,e_r\}$ lie on the border of $\Delta_{n+1}$, $\gamma_h$ is zero on all chambers adjacent to $e_j$, not equal to $D_j$, and thus $\varepsilon_1(\gamma_h)(e_j)=\gamma_h(D_j,c_j)$ for all $j$ (where $c_j$ is the orientation of $D_j$ that agrees with that of $e_j$). \\

\noindent Moreover, we are assuming that $e_j\notin\mathcal{X}$ for all $j$, so $z_h(e_j)=0$. Therefore, we see that $$\beta_{i-1}(e_{i-1})+\gamma_h(D_{i-1},c_{i-1})=(h\cdot \beta_{i-1})(e_{i-1})=h\beta_{i-1}(h^{-1}e_{i-1})$$

\noindent So let $h:=h_{i-1}h_i^{-1}\in A\cap S$, so that $h^{-1}e_{i-1}=e_i$, and 
\begin{align*}
\beta_{i-1}(e_i)&=\beta_{i-1}(h^{-1}e_{i-1})=h^{-1}h\beta_{i-1}(h^{-1}e_{i-1})\\&=h^{-1}\beta_{i-1}(e_{i-1})+h^{-1}\gamma_h(D_{i-1},c_{i-1})\\&=h^{-1}\beta_{i-1}(e_{i-1})+(h^{-1}\cdot \gamma_h)(h^{-1}D_{i-1},h^{-1}c_{i-1})
\end{align*}

\noindent But $h^{-1}D_{i-1}=D_i$ and $h^{-1}c_{i-1}=c_i$, so define $\gamma_i'\in C_2(\Delta,\mathbb{X})$ by $$\gamma_i'(D,c):=\begin{cases}
-(h^{-1}\cdot \gamma_h)(D_i,c_i) & (D,c)=(D_i,c_i)\\
(h^{-1}\cdot \gamma_h)(D_i,c_i) & (D,c)=(D_i,-c_i)\\
0 & \text{otherwise}
\end{cases}$$

\noindent Let $\beta_i:=\beta_{i-1}+\varepsilon_1(\gamma_i')$, and $\gamma_i:=\gamma_{i-1}+\gamma_i'$. Then $\beta_i=\beta_1+\varepsilon_1(\gamma_i)$ and $\gamma_i$ is zero on all chambers outside $\{D_1,\dots,D_i\}$.\\

\noindent Set $s_k^{(i)}:=\beta_i(e_k)=\beta_{i-1}(e_k)+\varepsilon_1(\gamma_i')(e_k)$ for each $k\leq i$. For $k<i$, this is equal to $$\beta_{i-1}(e_k)=
s_k^{(i-1)}=h_ks_1^{(i-1)}=h_ks_1^{(i)}$$ and $$s_i^{(i)}=\beta_{i-1}(e_i)+\gamma_i'(D_i,c_i)=\beta_{i-1}(e_i)-(h^{-1}\cdot \gamma_h)(D_i,c_i)=h^{-1}\beta_{i-1}(e_{i-1})=h^{-1}s_{i-1}^{(i-1)}$$

\noindent Thus $s_{i}^{(i)}=h_ih_{i-1}^{-1}s_{i-1}^{(i-1)}=h_ih_{i-1}^{-1}h_{i-1}s_1^{(i-1)}=h_is_1^{(i-1)}=h_is_1^{(i)}$.\\

\noindent So by induction, we may choose $\beta':=\beta_r$, $\gamma:=\gamma_r$, and thus $\beta'(e_j)=h_j\beta'(e_1)$ for all $j=1,\dots,r$.\\

\noindent Finally, if $h\in N$ then $h\cdot e_i=e_i$ for all $i$, so $h\in I_{e_i}$ for all $i$ by assumption. So since $\beta'(e_i)\in\mathbb{X}^{I_{e_i}}$, it follows that $h\cdot \beta'(e_i)=\beta'(e_i)$. So since $A/N=\{h_1N,\dots,h_nN\}$, the result follows.\end{proof}


\subsection{Shifting chains on summits}\label{subsec: I-invariance}

From now on, we will assume that the residue field of $K$ has order $p$. Again, we suppose $\beta\in C_1(\Delta,\mathbb{X})$ satisfies all the assumptions of section \ref{subsec: key lemma}. To outline how we will apply Lemma \ref{lem: crucial}, we must define some further  data: Fix a vertex $v\in\Delta_{n+1}\backslash\mathcal{X}$, and we know by Theorem \ref{thm: peak} that $v$ is a peak of $\Delta_{n+1}$.\\ 

\noindent Let $D_v$ be the summit at $v$, and let $e_v,f_v$ be two oriented edges of $D_v$ with target $v$.  Let $E_1,\dots,E_q$ (resp. $F_1,\dots,F_q$) be all chambers in $\Delta$ that meet $D$ at $f_v$ (resp. $e_v$), and let $e_i$ (resp. $f_i$) be the oriented edge of $E_i$ (resp. $F_i$) with target $v$, but which is not equal to $f_v$ (resp. $e_v$). The diagram below illustrates this cumbersome statement when $q=2$.

\begin{center}

\tikzset{every picture/.style={line width=0.75pt}} 

\begin{tikzpicture}[x=0.75pt,y=0.75pt,yscale=-1,xscale=1]

\draw   (323.88,194.76) -- (455.34,377.51) -- (192.43,377.51) -- cycle ;
\draw  [color={rgb, 255:red, 0; green, 0; blue, 0 }  ,draw opacity=1 ] (70.91,205.39) -- (323.88,194.76) -- (192.13,377.31) -- cycle ;
\draw  [color={rgb, 255:red, 0; green, 0; blue, 0 }  ,draw opacity=1 ] (576.99,195.29) -- (455.34,377.51) -- (323.95,195.21) -- cycle ;
\draw    (134.41,118.12) -- (152.89,201.23) ;
\draw    (134.41,118.12) -- (323.88,194.76) ;
\draw    (324.22,195) -- (530.47,130.73) ;
\draw    (530.47,130.73) -- (511.62,196.97) ;
\draw  [dash pattern={on 0.84pt off 2.51pt}]  (152.89,201.23) -- (191.16,370.89) ;
\draw  [dash pattern={on 0.84pt off 2.51pt}]  (511.62,196.97) -- (455.34,377.51) ;

\draw (317.56,206) node [anchor=north west][inner sep=0.75pt]   [align=left] {$\displaystyle v$};
\draw (312.59,299.32) node [anchor=north west][inner sep=0.75pt]   [align=left] {$\displaystyle D_{v}$};
\draw (520.88,212.32) node [anchor=north west][inner sep=0.75pt]  [color={rgb, 255:red, 0; green, 0; blue, 0 }  ,opacity=1 ] [align=left] {$\displaystyle E_{1}$};
\draw (244.58,265.89) node [anchor=north west][inner sep=0.75pt]   [align=left] {$\displaystyle e_{v}$};
\draw (387.79,265.04) node [anchor=north west][inner sep=0.75pt]   [align=left] {$\displaystyle f_{v}$};
\draw (483.49,156.91) node [anchor=north west][inner sep=0.75pt]  [color={rgb, 255:red, 0; green, 0; blue, 0 }  ,opacity=1 ] [align=left] {$\displaystyle E_{2}$};
\draw (98.11,213.18) node [anchor=north west][inner sep=0.75pt]  [color={rgb, 255:red, 0; green, 0; blue, 0 }  ,opacity=1 ] [align=left] {$\displaystyle F_{1}$};
\draw (180,150.94) node [anchor=north west][inner sep=0.75pt]  [color={rgb, 255:red, 0; green, 0; blue, 0 }  ,opacity=1 ] [align=left] {$\displaystyle F_{2}$};
\draw (450.21,196.83) node [anchor=north west][inner sep=0.75pt]   [align=left] {$\displaystyle e_{1}$};
\draw (418.17,138.86) node [anchor=north west][inner sep=0.75pt]   [align=left] {$\displaystyle e_{2}$};
\draw (182.16,201.95) node [anchor=north west][inner sep=0.75pt]   [align=left] {$\displaystyle f_{1}$};
\draw (233.79,137.15) node [anchor=north west][inner sep=0.75pt]   [align=left] {$\displaystyle f_{2}$};
\draw (140.4,391.22) node [anchor=north west][inner sep=0.75pt]   [align=left] {\textbf{Figure 6:} The chambers adjacent to the summit at $\displaystyle v$};

\end{tikzpicture}

\end{center}

\noindent Fix any two chambers $E,F$ of $\Delta$ such that $d(E,D_v)=d(F,D_v)=2$, $f_j$ is an edge of $F$ for some $j$, and $e_i$ is an edge of $E$ for some $i$. Then $E$ and $F$ are summits of $\Delta_{n+2}$ by Theorem \ref{thm: d-partite}. Suppose further that the peaks of $E$ and $F$ are joined by an edge. This is illustrated below.

\begin{center}

\tikzset{every picture/.style={line width=0.75pt}} 

\begin{tikzpicture}[x=0.75pt,y=0.75pt,yscale=-1,xscale=1]

\draw   (318.12,194.58) -- (411.35,327.07) -- (224.89,327.07) -- cycle ;
\draw  [color={rgb, 255:red, 245; green, 152; blue, 35 }  ,draw opacity=1 ] (113.11,210.37) -- (317.97,194.07) -- (224.37,326.31) -- cycle ;
\draw  [color={rgb, 255:red, 245; green, 152; blue, 35 }  ,draw opacity=1 ] (519.84,213.92) -- (411.35,327.07) -- (318.14,193.95) -- cycle ;
\draw  [color={rgb, 255:red, 245; green, 152; blue, 35 }  ,draw opacity=1 ][dash pattern={on 4.5pt off 4.5pt}] (481.43,159.69) -- (411.34,327.68) -- (318.13,194.56) -- cycle ;
\draw  [color={rgb, 255:red, 74; green, 144; blue, 226 }  ,draw opacity=1 ] (266.77,47.85) -- (317.42,193.74) -- (153.94,152.8) -- cycle ;
\draw  [color={rgb, 255:red, 74; green, 144; blue, 226 }  ,draw opacity=1 ] (433.79,71.31) -- (521,213.2) -- (318.81,193.27) -- cycle ;
\draw [color={rgb, 255:red, 245; green, 152; blue, 35 }  ,draw opacity=1 ]   (153.94,152.8) -- (176.25,206.81) ;
\draw [color={rgb, 255:red, 245; green, 152; blue, 35 }  ,draw opacity=1 ] [dash pattern={on 4.5pt off 4.5pt}]  (176.25,206.81) -- (224.53,326.22) ;
\draw [color={rgb, 255:red, 208; green, 2; blue, 2 }  ,draw opacity=1 ]   (266.77,47.85) -- (434.37,71.37) ;

\draw (310.22,263.19) node [anchor=north west][inner sep=0.75pt]   [align=left] {$\displaystyle D_{v}$};
\draw (313.84,202.82) node [anchor=north west][inner sep=0.75pt]   [align=left] {$\displaystyle v$};
\draw (154.13,249.86) node [anchor=north west][inner sep=0.75pt]  [color={rgb, 255:red, 245; green, 152; blue, 35 }  ,opacity=1 ] [align=left] {$ $};
\draw (146.08,218.5) node [anchor=north west][inner sep=0.75pt]  [color={rgb, 255:red, 245; green, 152; blue, 35 }  ,opacity=1 ] [align=left] {$\displaystyle F_{1}$};
\draw (174.45,165.19) node [anchor=north west][inner sep=0.75pt]  [color={rgb, 255:red, 245; green, 152; blue, 35 }  ,opacity=1 ] [align=left] {$\displaystyle F_{2}$};
\draw (196.34,204.39) node [anchor=north west][inner sep=0.75pt]  [color={rgb, 255:red, 245; green, 152; blue, 35 }  ,opacity=1 ] [align=left] {$\displaystyle f_{1}$};
\draw (216.61,171.46) node [anchor=north west][inner sep=0.75pt]  [color={rgb, 255:red, 245; green, 152; blue, 35 }  ,opacity=1 ] [align=left] {$\displaystyle f_{2}$};
\draw (404.97,202.04) node [anchor=north west][inner sep=0.75pt]  [color={rgb, 255:red, 245; green, 152; blue, 35 }  ,opacity=1 ] [align=left] {$\displaystyle e_{1}$};
\draw (406.59,173.03) node [anchor=north west][inner sep=0.75pt]  [color={rgb, 255:red, 245; green, 152; blue, 35 }  ,opacity=1 ] [align=left] {$\displaystyle e_{2}$};
\draw (472.69,216.15) node [anchor=north west][inner sep=0.75pt]  [color={rgb, 255:red, 245; green, 152; blue, 35 }  ,opacity=1 ] [align=left] {$\displaystyle E_{1}$};
\draw (447.56,168.33) node [anchor=north west][inner sep=0.75pt]  [color={rgb, 255:red, 245; green, 152; blue, 35 }  ,opacity=1 ] [align=left] {$\displaystyle E_{2}$};
\draw (234.39,114.23) node [anchor=north west][inner sep=0.75pt]  [color={rgb, 255:red, 74; green, 144; blue, 226 }  ,opacity=1 ] [align=left] {$\displaystyle F$};
\draw (424.81,127.56) node [anchor=north west][inner sep=0.75pt]  [color={rgb, 255:red, 74; green, 144; blue, 226 }  ,opacity=1 ] [align=left] {$\displaystyle E$};
\draw (277.36,249.86) node [anchor=north west][inner sep=0.75pt]  [color={rgb, 255:red, 0; green, 0; blue, 0 }  ,opacity=1 ] [align=left] {$\displaystyle e_{v}$};
\draw (347.89,250.65) node [anchor=north west][inner sep=0.75pt]  [color={rgb, 255:red, 0; green, 0; blue, 0 }  ,opacity=1 ] [align=left] {$\displaystyle f_{v}$};
\draw (199.16,351.64) node [anchor=north west][inner sep=0.75pt]   [align=left] {\textbf{Figure 7:} The chambers $\displaystyle E$ and $\displaystyle F$};

\end{tikzpicture}

\end{center}

\noindent For convenience, unless the choice of vertex $v\in\Delta_{n+1}\backslash\mathcal{X}$ is ambiguous, we will often just refer to $D_v,e_v$ and $f_v$ as $D,e$ and $f$.


\begin{definition}\label{defn: H_e and H_f}
We define the subgroups $H_e\leq I_e$ and $H_f\leq I_f$ as: $$H_e:=\{g\in I_e:g\cdot F=F\}$$ and $$H_f:=\{g\in I_f:g\cdot E=E\}$$
\end{definition}

\noindent\textbf{Note:} 1. This definition depends on the choice of the chambers $E$ and $F$, but varying the choice simply yields conjugates of $H_e$ and $H_f$.\\

\noindent 2. Using Proposition \ref{propn: distance fixing} we see that the intersection of $H_e$ (resp. $H_f$) with $I$ does not fix $E_i$ (resp. $F_i$) for any $i$.\\

\noindent In section \ref{subsec: X_2 properties}, we will realise these subgroups $H_e,H_f$ explicitly in the case where $v=v_0$ is the hyperspecial vertex. By symmetry in the building, the results we will prove there carry accross to any choice of $v$. In particular the following properties of $H_e$ and $H_f$ follow immediately from Lemma \ref{lem: order p} and Proposition \ref{propn: generating_subgroup1}:

\begin{proposition}\label{H-properties} $ $
\begin{itemize}
\item \emph{$H_e\cap H_f=I_v$.}

\item \emph{$\text{Stab}_{H_e}(E_i)=\text{Stab}_{H_f}(F_j)=I_v$ for all $i,j=1,\dots p$.}

\item \emph{$H_e/I_v$ and $H_f/I_v$ have order $p$.}

\item \emph{$I_D=\langle H_e,H_f\rangle$.}
\end{itemize}
\end{proposition}

\noindent Using these subgroups $H_e$ and $H_f$, we can now prove the following results, which demonstrate the usefulness of Lemma \ref{lem: crucial}.

\begin{proposition}\label{propn: split}
Let $v\in \Delta_{n+1}\backslash \mathcal{X}$, let $D=D_v$ be the summit of $\Delta_{n+1}$ at $v$, and let the chambers $E_1,\dots,E_p,F_1,\dots,F_p,E,F$, the oriented edges $e_v,f_v,e_1,\dots,e_p,f_1,\dots,f_p$, and the subgroups $H_e,H_f$ be defined as above.\\

\noindent Then there exists $\gamma\in C_2(\Delta,\mathbb{X})$ which is non-zero only on $E_1,\dots,E_p,F_1,\dots,F_p$ such that if $\beta':=\beta+\varepsilon_1(\gamma)$ then $H_e$ acts transitively on $\{\beta'(e_1),\dots,\beta'(e_p)\}$, and $H_f$ acts transitively on $\{\beta'(f_1),\dots,\beta'(f_p)\}$.
\end{proposition}

\begin{proof}




Using Theorem \ref{thm: d-partite}, we see that all chambers adjacent to $e_i,f_j$ not equal to $E_i,F_j$ are summits of $\Delta_{n+2}$, and thus $e_1,\dots,e_p,f_1,\dots,f_p$ lie on the border of $\Delta_{n+1}$ by Definition \ref{defn: border}. Moreover, since $v\notin\mathcal{X}$, the edges $e_1,\dots,e_p,f_1,\dots,f_p$ are not contained in $\mathcal{X}$.

So to find $\gamma\in C_2(\Delta,\mathbb{X})$ satisfying the required condition, we only need to show that the action of $H_e$ and $H_f$ on $\{e_1,\dots,e_p\}$ and $\{f_1,\dots,f_p\}$ satisfy the hypotheses of Lemma \ref{lem: crucial}. By symmetry, it suffices to prove this for $H_e$.\\

\noindent But we also know by Proposition \ref{H-properties} that $H_e/I_v$ has order $p$, and Stab$_{H_e}(e_i)=$ Stab$_{H_f}(f_k)=I_v$ for all $i,k=1,\dots,p$. In particular, Stab$_{H_e}(e_i)=$ Stab$_{H_e}(e_j)$ for all $i,j$, so hypothesis 3 of Lemma \ref{lem: crucial} is satisfied. But since $H_e/I_v$ permutes $e_1,\dots,e_p$, and this action is non-trivial, every non-trivial element of $H_e/I_v$ must act by a $p$-cycle, so it follows that the action is transitive, giving us hypothesis 2.

Moreover, if $g\in H_e$ and $g\cdot e_i=e_i$ for all $i$, then $g\in I_v$. So $g$ must stabilise all chambers adjacent to $e_i$, for $i=1,\dots,p$, and thus $g\in I_{e_i}$ by Proposition \ref{propn: adjacent chambers facet subgroup}, so hypothesis 1 is also satisfied.\\

\noindent Therefore, to apply Lemma \ref{lem: crucial}, it remains only to prove hypothesis 4, i.e. that $H_e/I_v=(I\cap H_e)/K_v$. Again, since $H_e/I_v$ has order $p$, we only need to show that $i\cap H_e\not\subseteq I_v$, and this only requires us to show that $I\cap H_e$ does not stabilise $e_i$ for some $i$.\\ 

\noindent Using Proposition \ref{propn: adjacent chambers facet subgroup}, we can realise $I\cap H_e$ as $$I\cap H_e:=\{g\in I:g\cdot F=F,g\cdot D=D \text{ and }g\cdot F_i=F_i\text{ for all }i\}$$

\noindent But $I$ stabilises $C$, $d(D,C)=r$, $d(C,F_k)=n+1$ and $d(C,F)=n+2$, so it follows from Corollary \ref{cor: gallery fixing} that any element of $I$ that stabilises $F$ will stabilise $F_k$ and $D$, and hence every chamber adjacent to $e$, i.e. $I\cap \text{Stab}_G(F)\subseteq I\cap H_e$ by Proposition \ref{propn: adjacent chambers facet subgroup}. 

But by Proposition \ref{propn: distance fixing}, there exists an element of $I$ that fixes $F$ but does not fix any $E_i$, i.e. there exists an element of $I\cap H_e$ that lies outside $I_v$ as required. Therefore, $H_e\cap I/I_v=H_e/I_v$ as required.\\

\noindent So, applying Lemma \ref{lem: crucial}, there exist $\gamma_e,\gamma_f\in C_2(\Delta_{n+1},\mathbb{X})$ such that $\gamma_e(D')=0$ (resp. $\gamma_f(D')=0$) whenever $D'$ is a chamber not equal to $E_i$ (resp. $F_i$) for any $i$, and if $\beta'=\beta+\varepsilon_1(\gamma_e+\gamma_f)$ then $H_e$ acts transitively on $\{\beta'(e_1),\dots,\beta'(e_q)\}$, and $H_f$ acts transitively on $\{\beta'(f_1)\cdots,\beta'(f_q)\}$ as required.\end{proof}

\begin{corollary}\label{cor: split2}
Let $v\in\Delta_{n+1}\backslash\mathcal{X}$, let $D$ be the summit of $\Delta_{n+1}$ at $v$, and let the set of chambers $\mathcal{Y}:=\{D,E_1,\dots,E_p,F_1,\dots,F_p\}$, and the oriented edges $e,f,e_1,\dots,e_p,f_1,\dots,f_p$ be defined as in Proposition \ref{propn: split}.\\ 

\noindent Then there exists a $\mathcal{Y}$-shift $\beta'$ of $\beta$ such that $$\beta'(e)+\underset{1\leq i\leq p}{\sum}\beta'(e_i)=\beta'(f)+\underset{1\leq j\leq p}{\sum}\beta'(f_j)=0$$
\end{corollary}

\begin{proof}

Applying Proposition \ref{propn: split}, we know that there exists $\gamma''\in C_2(\Delta,\mathbb{X})$ that is non-zero only on $E_{1},\dots,E_{p},F_{1},\dots,F_{p}$ such that if $\beta''=\beta+\varepsilon_1(\gamma'')$ then $H_{e}$ acts transitively on $\{\beta''(e_{1}),\dots,\beta''(e_{p})\}$ and $H_{f}$ acts transitively on $\{\beta''(f_{1}),\dots,\beta''(f_{p})\}$.\\

\noindent Therefore, the sums $\underset{1\leq i\leq p}{\sum}{\beta''(e_{i})}$ and $\underset{1\leq j\leq p}{\sum}{\beta''(f_{j})}$ are respectively $H_{e}$ and $H_{f}$-invariant.\\

\noindent But $v\notin\mathcal{X}$, so we know that $\varepsilon_0(\beta'')(v)=\varepsilon_0(\beta)(v)=0$, which implies that $$\beta''(e)+\underset{1\leq i\leq p}{\sum}{\beta''(e_{i})}=-\beta''(f)-\underset{1\leq j\leq p}{\sum}{\beta''(f_{j})}$$ But $\beta''(e)\in\mathbb{X}^{I_{e}}$ and $\beta''(f)\in\mathbb{X}^{I_{f}}$, so since $H_{e}\subseteq I_{e}$ and $H_{f}\subseteq I_{f}$, this implies that the left hand side of this equality is $H_{e}$-invariant, while the right hand side is $H_{f}$-invariant. So set $\ell:=\beta''(e)+\underset{1\leq i\leq p}{\sum}{\beta''(e_{i})}$, and we see that $\ell$ is invariant under $\langle H_{e},H_{f}\rangle$, which is equal to $I_{D}$ by Proposition \ref{H-properties}.\\

\noindent So, letting $o$ be the orientation of $e$, define $\gamma'\in C_2(\Delta,\mathbb{X})$ by $$\gamma'(E,c):=\begin{cases}
-\ell & (E,c)=(D,o\uparrow^{D})\\
\ell & (E,c)=(D,-\sigma(o)\uparrow^D)\\
0 & \text{otherwise}
\end{cases}$$  Clearly $\gamma'$ is non-zero only on $D$, so let $\gamma:=\gamma'+\gamma''$, and $\gamma$ is non-zero only on $\{D,E_1,\dots,E_p,F_1,\dots,F_p\}$. Define $$\beta':=\beta''+\varepsilon_1(\gamma')=\beta+\varepsilon_1(\gamma)$$ Then $\beta'(e_{i})=\beta''(e_{i})$, $\beta'(f_{i})=\beta''(f_{i})$ for all $i\leq p$, while $\beta'(e)=\beta''(e)-\ell$, $\beta'(f)=\beta''(f)+\ell$. In particular: $$\beta'(e)+\underset{1\leq i\leq p}{\sum}{\beta'(e_{i})}=\beta''(e)+\underset{1\leq i\leq p}{\sum}{\beta''(e_{i})}-\ell=\ell-\ell=0$$ and similarly $\beta'(f)+\underset{1\leq j\leq p}{\sum}{\beta'(f_{j})}=0$ as required.\end{proof}

\noindent Interpreting this statement geometrically, it means we can divide $\beta$ on the region in Figure 6 into a sum of two chains, each non-zero on precisely one side of $D_v$, and the image of both under $\varepsilon_0$ will annihilate $v$.

\subsection{Unbroken regions}

We will assume throughout this section that $\mathcal{X}$ is a complete, $I$-invariant region of $\Delta$, with $\Delta_n\subseteq\mathcal{X}\subseteq\Delta_{n+1}$ and $\beta\in C_1(\Delta_{n+1},\mathbb{X})$ satisfies all the assumptions at the start of section \ref{subsec: key lemma}. Again, we ultimately want to prove that $\beta\in C_1(\mathcal{X},\mathbb{X})+\varepsilon_1(C_2(\Delta,\mathbb{X}))$, i.e. $\beta$ has a shift in $C_1(\mathcal{X},\mathbb{X})$, and we will now see the first cases of when this holds.\\ 





\noindent Recall from section \ref{subsec: Delta decomposition} that we can decompose $\Delta_{n+1}$ as $$\Delta_{n+1}=\Delta_{n}\sqcup\text{ Crown}(X_{0,n+1})\sqcup\text{ Crown}(X_{1,n+1})\sqcup\text{ Crown}(X_{2,n+1})$$ where Crown$(X_{i,n+1})=S_{1,i}^{(n+1)}\sqcup S_{2,i}^{(n+1)}\sqcup\cdots\sqcup S_{m+1,i}^{(n+1)}$, $m:=\lfloor\frac{n+1}{2}\rfloor$. Again, $S_{j,i}^{(n+1)}$ is a set of summits of $\Delta_{n+1}$, and we let $P_{j,i}^{(n+1)}$ be the corresponding set of peaks.\\

\noindent Also, using Lemma \ref{lem: I-invariant complete region}, we know that $\mathcal{X}=\Delta_n\sqcup\underset{(j,i)\in\Gamma(\mathcal{X})}{\bigsqcup}S_{j,i}^{(n+1)}$ for some subset $\Gamma=\Gamma(\mathcal{X})\subseteq\{1,\dots,m+1\}\times\{0,1,2\}$. Note that $\Gamma$ is empty if and only if $\mathcal{X}=\Delta_n$. 

Moreover, we can decompose $\Gamma$ as $$\Gamma(\mathcal{X})=(\Gamma_0(\mathcal{X})\times\{0\})\sqcup(\Gamma_1(\mathcal{X})\times\{1\})\times(\Gamma_2(\mathcal{X})\times\{2\})$$ where $\Gamma_i(\mathcal{X})\subseteq\{1,\dots,m+1\}$.

\begin{definition}\label{defn: unbroken}
We say $\mathcal{X}$ is \emph{unbroken} if for each $i=0,1,2$, $\Gamma_i(\mathcal{X})$ is either empty or is an interval $\{m_i,m_i+1,\dots,M_i\}$ for some $m_i\leq M_i$.
\end{definition}

\noindent \textbf{Note:} If $n=0$, then all complete regions in $\Delta_{n+1}$ are unbroken, since $m=\lfloor\frac{n+1}{2}\rfloor=0$.\\

\noindent For now, we will assume further that $\mathcal{X}$ is unbroken, and write $\Gamma_i(\mathcal{X})=\{m_i,\dots,M_i\}$ for some $1\leq m_i\leq M_i\leq m+1$. Note that $M_i=0$ for all $i$ if and only if $\mathcal{X}=\Delta_n$.\\

\noindent Now, let $R_i=R_i(\beta)$ (resp. $r_i=r_i(\beta)$) be maximal (resp. minimal) such that $\beta$ is non-zero on an edge adjacent to some vertex in $P_{R_i,i}^{(n+1)}$ (resp. $P_{r_i,i}^{(n+1)}$). On the other hand, if no such vertex exists and $\beta$ is zero on Crown$(X_{i,n+1})$, then we say that $R_i=r_i=0$. We will also assume that for each $i\in\{0,1,2\}$, either

\begin{itemize}

\item $M_i=m_i=0$ or

\item $r_i\leq m_i\leq M_i\leq R_i$.

\end{itemize}

\begin{lemma}\label{lem: unbroken advantage}
If $M_i=R_i$ and $m_i=r_i$ for all $i$, then $\beta\in C_1(\mathcal{X},\Delta)$.
\end{lemma}

\begin{proof}

Suppose $\beta$ is non-zero on an edge $e\in\Delta_{n+1}\backslash$, and suppose $e$ contains a vertex $v\notin\mathcal{X}$. Then since $\Delta_n\subseteq\mathcal{X}$, $v$ must be a peak in Crown$(X_{i,n+1})$ for some $i\in\{0,1,2\}$, so $v\in P_{j,i}^{(n+1)}$ for some $j\in\{1,\dots,m+1\}$.

But by the definition of $r_i$ and $R_i$, this implies that $j\geq r_i=m_i$ and $j\leq R_i=M_i$, so $j\in\{m_i,\dots,M_i\}=\Gamma_i$, and thus $S_{j,i}^{(n+1)}\subseteq\mathcal{X}$, which implies that $v\in\mathcal{X}$ -- contradiction.\end{proof}

\noindent So we now want to find a shift $\beta'$ of $\beta$ with $r_i(\beta')=m_i$ and $R_i(\beta')=m_i$. We will first consider the simplest case, when $M_i=0$, which implies that $\mathcal{X}$ contains no peaks in Crown$(X_{i,n+1})$.

\begin{lemma}\label{lem: M=1}

If $M_i=0$ and $r_i=R_i$, then there exists a \emph{Crown}$(X_{i,n+1})$-shift $\beta'$ of $\beta$ with $R_i(\beta')=r_i(\beta')=0$.

\end{lemma}

\begin{proof}

Let $j:=r_i=R_i$. If $j=0$ then we can take $\beta'=\beta$, so assume that $j\geq 1$, and thus there exists $v\in P_{j,i}^{(n+1)}$ such that $\beta$ is non-zero on some edge adjacent to $v$.\\

\noindent But the only oriented edges in $\Delta_{n+1}$ meeting $v$ that are not adjacent to any vertex in $P_{i,j-1}^{(n+1)}$ or $P_{i,j+1}^{(n+1)}$ are the edges $e_v,f_v$ of the summit $D_v$, and we assume without loss of generality that $e_v,f_v$ have target $v$.\\ 

\noindent Since $e_v,f_v$ are the only edges adjacent to $v$ at which $\beta$ can be non-zero, we deduce that $\varepsilon_0(\beta)(v)=\beta(e_v)+\beta(f_v)$. But since $M_i=m_i=0$ we know that $\mathcal{X}$ contains no summit of $\Delta_{n+1}$ in $X_{i,n+1}$, so $v\notin\mathcal{X}$ and thus $\varepsilon_0(\beta)(v)=0$. Therefore, $$\beta(f_v)=-\beta(e_v)\in\mathbb{X}^{I_{e_v}}\cap\mathbb{X}^{I_{f_v}}=\mathbb{X}^{\langle I_{e_v},I_{f_v}\rangle}=\mathbb{X}^{I_{D_v}}$$

\noindent Writing $o$ as the orientation of $e_v$, define $\gamma_v\in C_2(\Delta,\mathbb{X})$ by $$\gamma_v(E,c)=\begin{cases}-\beta(e_v) & (E,o)=(D_v,o\uparrow^{D_v})\\ \beta(e_v) & (E,o)=\sigma(D_v,o\uparrow^{D_v})\\ 0 & \text{otherwise}\end{cases}$$ Let $\gamma$ be the sum of all $\gamma_v$, and $v$ ranges over the peaks in $P_{j,i}^{(n+1)}$, and let $\beta':=\beta+\varepsilon_1(\gamma)$.

Then for each such peak $v\in P_{1,i}^{(n+1)}$, $\beta'(e_v)=\beta(e_v)+\varepsilon_1(\gamma)(e_v)=\beta(e_v)-\beta(e_v)=0$, and similarly $\beta'(f_v)=0$, and it follows that $\beta'$ is zero on Crown$(X_{i,n+1})$, i.e. $n_i(\beta')=0$.\end{proof}

\noindent\textbf{Note:} If $n=0$ then either $M_i=0$ or $M_i=R_i$, so it follows from this lemma that $\beta$ has a shift in $C_1(\mathcal{X},\mathbb{X})$ as we require. So we may assume from now on that $n\geq 1$.

\subsection{Proof of Theorem \ref{letterthm: exact for star}}

Recall how we define the \emph{star} of the hyperspecial vertex $v_0$, i.e. Star$(v_0)$ is the set of all chambers containing $v_0$ (i.e. Star$(v_0)=X_{0,2}$ in our notation). Also recall from the introduction we call an $I$-invariant region $\mathcal{Y}\subseteq\Delta$ \emph{star-bounded} if it is contained in Star$(v_0)$.

We will now state the following important conjecture, which is implicit in the statement of Theorem \ref{letterthm: exact for star}, and which we will assume for the remainder of the section:

\begin{letterconj}\label{conj: did not get it}
If $\mathcal{Y}$ is a star-bounded region in $\Delta$, then $\mathcal{Y}$ is \emph{Star}$(v_0)$-collapsible.
\end{letterconj}

\noindent This conjecture remains the only obstacle to completing a full proof of Conjecture \ref{sur} for $SL_3(K)$, as we will now demonstrate.\\





\noindent Once again, fix $\mathcal{X}$ a complete, $I$-invariant unbroken region of $\Delta$, with $\Delta_n\subseteq\mathcal{X}\subseteq\Delta_{n+1}$ for some $n\geq 1$, and we fix a chain $\beta\in C_1(\Delta_{n+1},\mathbb{X})$ satisfying the assumptions stated at the start of section \ref{subsec: key lemma}.\\

\noindent Given a summit $D$ of $\Delta_{n-1}$, with peak $v$, recall from section \ref{subsec: Delta decomposition} how we define the region $X_D:=\text{Star}(v)$, consisting of all chambers containing $v$, which is isometric with Star$(v_0)$.

Therefore, assuming Conjecture \ref{conj: did not get it}, we can assume that for any $I$-invariant region $\mathcal{L}\subseteq X_D$, $\mathcal{L}$ is $X_D$-collapsible.

\begin{proposition}\label{propn: isolating each summit}

Fix $i\in\{0,1,2\}$, and suppose that $r_i(\beta)<R_i(\beta)-1$ and $M_i<R_i(\beta)-1$ (resp. $m_i> r_i(\beta)+1$). Then there exists an $X_{i,n+1}$-shift $\beta_0$ of $\beta$ such that $R_i(\beta_0)<R_i(\beta)$ (resp. $r_i(\beta_0)>r_i(\beta)$).






\end{proposition}

\begin{proof}

It suffices to prove the statement for $R_i$, and the statement for $r_i$ follows by symmetry. So let $j:=R_i(\beta)$.\\

\noindent Firstly, fix any summit $E\in S_{j-1,i}^{(n+1)}$, and let $v$ be the peak of $E$. Since $j-1>M_i$, it follown that $(j-1,i)\notin \Gamma$, and hence $v\notin\mathcal{X}$. So as in the proof of Proposition \ref{propn: split}, let $\mathcal{Y}_E:=\{E,E_1,\dots,E_p,F_1,\dots,F_p\}$, where $E_1,\dots,E_p,F_1,\dots,F_p\in$ Crown$(X_{0,n+1})$ are adjacent to $E$.

Note that $\mathcal{Y}_{E}$ and $\mathcal{Y}_{E'}$ share no common chamber if $E\neq E'$, since two distinct peaks in $P_{j-2,i}^{(n+1)}$ cannot be joined by an edge by Theorem \ref{thm: d-partite}.\\

\noindent Also, let $e_v,f_v$ be the oriented edges of $E$ with target $v$, where the origin of $f_v$ lies in $P_{j-1,i}^{(n-1)}$. Let $e_{i,v},f_{i,v}$ be the oriented edges of $E_i$ and $F_i$ respectively with target $v$, these edges giving the same arrangement as in Figure 6.\\

\noindent Applying Corollary \ref{cor: split2}, we can find a $\mathcal{Y}_E$-shift $\beta_E$ of $\beta$ such that $$\beta_E(e_v)+\underset{1\leq i\leq p}{\sum}{\beta_E(e_{i,v})}=\beta_E(f_v)+\underset{1\leq i\leq p}{\sum}{\beta_E(f_{i,v})}=0$$ Write $\beta_E:=\beta+\varepsilon_1(\gamma_E)$ for some $\gamma_E\in C_2(\Delta,\mathbb{X})$, zero on all edges outside $\mathcal{Y}_E$.\\

\noindent Define $\gamma$ to be the sum of all $\gamma_E$, as $E$ ranges over all summits in $S_{j-1,i}^{(n+1)}$. Then since the regions $\mathcal{Y}_E$ are mutually disjoint as sets of chambers, it follows $\gamma$ restricts to $\gamma_E$ on each $\mathcal{Y}_E$. So defining $\beta':=\beta+\varepsilon_1(\gamma)$, it follows that for each peak $v\in P_{j-1,i}^{(n+1)}$, we still have the identity $$\beta'(e_v)+\underset{1\leq i\leq p}{\sum}{\beta'(e_{i,v})}=\beta'(f_v)+\underset{1\leq i\leq p}{\sum}{\beta'(f_{i,v})}=0$$

\noindent Now, for each summit $D\in S_{j-1,i}^{(n-1)}$, define $\beta_D\in C_1(\Delta,\mathbb{X})$ by $$\beta_D(e):=\begin{cases} \beta'(e) & e\in X_D\text{ and }e\neq f_v,\sigma(f_v)\text{ for any }v\in P_{j-1,i}^{(n+1)}\cap X_D \\ 0 & \text{otherwise} \end{cases}$$ Then clearly $\beta_D$ is zero on edges outside $X_D$. Moreover, since $f_v$ is never joined to a peak in $P_{j,i}^{(n+1)}$, it is clear that $\beta_D$ agrees with $\beta'$ on all edges in $X_D$ adjacent to peaks in $P_{j,i}^{(n+1)}$. In particular, $\varepsilon_0(\beta_D)$ is zero on these peaks.\\

\noindent Furthermore, for any $v\in P_{j-1,i}^{(n+1)}\cap X_D$, we know that $e_v,e_{k,v}\in X_D$ for all $k=1,\dots, p$, and these are the only edges adjacent to $v$ on which $\beta_D$ can be non-zero, thus $$\varepsilon_0(\beta_D)(v)=\beta_D(e_v)+\underset{1\leq i\leq p}{\sum}{\beta_D(e_{i,v})}=\beta'(e_v)+\underset{1\leq i\leq p}{\sum}{\beta'(e_{i,v})}=0$$ Thus $\beta_D$ is a chain defined on $X_D$ such that $\varepsilon_0(\beta_D)$ is zero outside of $X_D\cap\Delta_n$. But $X_D\cap\Delta_n$ is an $X_D$-collapsible region of $\Delta$ by Conjecture \ref{conj: did not get it}, so applying Proposition \ref{propn: loose end} we see that $\beta_D\in C_1(X_D\cap\Delta_n,\mathbb{X})+\varepsilon_1(C_2(\Delta,\mathbb{X}))\subseteq C_1(\Delta_n,\mathbb{X})+\varepsilon_1(C_2(\Delta,\mathbb{X}))$.\\

\noindent Therefore, we can find $\gamma_D\in C_2(\Delta,\mathbb{X})$ such that $\beta_D+\varepsilon_1(\gamma_D)\in C_1(\Delta_n\cap X_D,\mathbb{X})$. Moreover, since $X_D\subseteq \Delta_2(D):=\{E\in\Delta:d(D,E)\leq 2\}$, it follows from Proposition \ref{propn: exactness1} that we can assume $\gamma_D$ is zero outside of $\Delta_2(D)\subseteq X_{i,n+1}$.\\

\noindent Now, let $\gamma_0$ be the sum of all $\gamma_D$, as $D$ ranges over all summits in $S_{j-1,i}$, and let $\beta_0:=\beta'+\varepsilon_1(\gamma_0)$. Note that $\gamma_0$ is zero outside of the union $\mathcal{L}$ of all $X_D$, and hence $\beta_0$ agrees with $\beta'$ on all edges that are not contained in $\mathcal{L}$. In particular, $\beta_0$ is zero outside $\Delta_{n+1}$, and also on all edges adjacent to peaks in $P_{k,i}^{(n+1)}$ for $k>j$, i.e. $R_i(\beta_0)\leq R_i(\beta)$.

It remains to show that $\beta_0$ is zero on all edges in $\Delta_{n+1}$ adjacent to vertices in $P_{j,i}^{(n+1)}$, and it will follow that $R_i(\beta_0)< R_i(\beta)$.\\

\noindent For each $v\in P_{j,i}^{(n+1)}$, let $e$ be an edge adjacent to $v$, and since $\beta_0$ is zero on all edges containing vertices in $P_{j+1,i}^{(n+1)}$ or outside $\Delta_{n+1}$, we may assume that the second vertex of $e$ lies in $\Delta_n$ or $P_{j-1,i}^{(n+1)}$. Let $E$ be any chamber in $\mathcal{L}$ adjacent to $e$.

The summit at $v$ contains a unique $u\in P_{j-1,i}^{(n-1)}$, so let $D$ be the summit of $u$. The second vertex of $e$ must be joined to $u$, so $e\in X_D$, and $E$ must be contained in $X_D$.\\ 

\noindent But $E\notin\Delta_n$, so for any $D'\in S_{j-1,i}^{(n-1)}$ with $D\neq D'$, $E\notin X_{D'}$ by Corollary \ref{cor: local}. Hence $\gamma_{D'}(E)=0$ for all $D\neq D'$, and thus $\gamma_0(E)=\gamma_D(E)$.\\

\noindent Since this is true for all chambers in $\mathcal{L}$ adjacent to $e$, and since $\beta'(e)=\beta_D(e)$ by the definition of $\beta_D$, we must have that $\varepsilon_1(\gamma_0)(e)=\varepsilon_1(\gamma_D)(e)$, so $$\beta_0(e)=\beta'(e)+\varepsilon_1(\gamma_0)(e)=\beta_D(e)+\varepsilon_1(\gamma_D)(e)=(\beta_D+\varepsilon_1(\gamma_D))(e)$$ and since $\beta_D+\varepsilon_1(\gamma_D)\in C_1(\mathcal{X},\mathbb{X})$ and $e\notin\mathcal{X}$, this must be 0 as required.\end{proof}

\begin{proposition}\label{propn: j=1,2}
If $M_i=0$, then there exists an $X_{i,n+1}$-shift $\beta_0$ of $\beta$ such that $r_i(\beta_0)=R_i(\beta_0)=0$.
\end{proposition}

\begin{proof}

If $r_i(\beta)=R_i(\beta)$, then this follows from Lemma \ref{lem: M=1}, so we may assume that $r_i(\beta)<R_i(\beta)$. If $R_i(\beta)>r_i(\beta)+1$, then since $0=M_i\neq R_i(\beta)-1$ it follows from Proposition \ref{propn: isolating each summit} that we can find an $X_{i,n+1}$-shift $\beta'$ of $\beta$ with $R_i(\beta')<R_i(\beta)$. So replacing $\beta$ with $\beta'$ and repeating this process, we may assume that $R_i(\beta)-1\leq r_i(\beta)<R_i(\beta)$, i.e $r_i(\beta)=R_i(\beta)-1$.\\

\noindent Let $j:=r_i(\beta)$, then the only edges in Crown$(X_{i,n+1})$ at which $\beta$ is non-zero are adjacent to peaks in $P_{j,i}^{(n+1)}$ and $P_{i,j+1}^{(n+1)}$. For each summit $D\in S_{j,i}^{(n-1)}$, define $\beta_D\in C_1(X_D,\mathbb{X})$ by $$\beta_D(e):=\begin{cases}\beta(e) & e\in X_D\\ 0& \text{otherwise}\end{cases}$$ But for any vertex $v\in (P_{j,i}^{(n+1)}\sqcup P_{i,j+1}^{(n+1)})\cap\mathcal{L}_D$, all adjacent vertices to $v$ in $P_{j,i}^{(n+1)}\sqcup P_{j+1,i}^{(n+1)}\sqcup\Delta_n$ lie in $X_D$, and hence $\beta$ is zero on all edges adjacent to $v$ that do not lie in $X_D$, so it follows that $\varepsilon_0(\beta_D)(v)=\varepsilon_0(\beta)(v)=0$.\\

\noindent Therefore, $\varepsilon_0(\beta_D)\in C_1(\mathcal{L}_D\cap\Delta_n,\mathbb{X})$, and $X_D\cap\Delta_n$ is an $X_D$-collapsible region of $\Delta$ by Conjecture \ref{conj: did not get it}, so $\beta_D\in C_1(X_D\cap\Delta_n,\mathbb{X})+\varepsilon_1(C_2(\Delta,\mathbb{X}))$ by Proposition \ref{propn: loose end}.\\

\noindent Therefore, choosing $\gamma_D\in C_2(\Delta,\mathbb{X})$ such that $\beta_D+\varepsilon_1(\gamma_D)\in C_1(X_D\cap\Delta_n,\mathbb{X})\subseteq C_1(\Delta_n,\mathbb{X})$, an identical argument to the proof of Proposition \ref{propn: isolating each summit} shows that we can find an $X_{i,n+1}$-shift $\beta_0$ of $\beta$ that agrees with $\beta_D+\varepsilon_1(\gamma_D)$ on Crown$(X_{i,n+1})\cap\mathcal{L}_D$ for all $D\in S_{j,i}^{(n-1)}$. 

In other words, $\beta_0$ is zero on all edges in Crown$(X_{i,n+1})$, so it follows that $r_i(\beta_0)=R_i(\beta_0)$.\end{proof}

\begin{theorem}\label{thm: future lettered theorem}
If we assume Conjecture \ref{conj: did not get it}, then $\Delta_n$ is a collapsible region of $\Delta$ for all $n\in\mathbb{N}$.
\end{theorem}

\begin{proof}

We need to prove that $\Delta_n$ is $\mathcal{L}$-collapsible for all $I$-invariant complete regions $\mathcal{L}$ containing $\Delta_n$. Choose any such region $\mathcal{L}$, and suppose $\mathcal{L}\subseteq\Delta_{k+1}$ for some $k\geq n$.

Fix $\beta\in C_1(\mathcal{L},\mathbb{X})$ and suppose that $\beta$ is $(I,\Delta_n)$-shift invariant and $\varepsilon_0(\beta)\in C_0(\Delta_n,\mathbb{X})$. By Definition \ref{defn: collapsible}, we only need to prove that $\beta\in C_1(\Delta_n,\mathbb{X})+\varepsilon_1(C_2(\Delta,\mathbb{X}))$, i.e. that $\beta$ has a shift in $C_1(\Delta_n,\mathbb{X})$.\\

\noindent If $k=n$ then $\beta\in C_1(\Delta_{n+1},\mathbb{X})$. Since our region $\mathcal{X}$ is $\Delta_n$, we have that $M_i=0$ for all $i\in\{0,1,2\}$. So if $n=0$ then it follows from Lemma \ref{lem: M=1} that we can find a shift $\beta'$ of $\beta$ with $R_i(\beta')=r_i(\beta')=0$.

Therefore, we may assume $n\geq 1$, and thus we can apply Proposition \ref{propn: j=1,2} to find an $X_{i,n+1}$-shift $\beta_0$ of $\beta$ such that $r_i(\beta_0)=R_i(\beta)=0$, i.e $\beta_0$ is zero on Crown$(X_{i,n+1})$. 

But since $\beta_0$ is an $X_{i,n+1}$-shift and $X_{i,n+1}\cap\text{Crown}(X_{i',n+1})=\varnothing$ for $i'\neq i$, it follows that $\beta_0$ agrees with $\beta$ on Crown$(X_{i,n+1})$. So repeating this argument for all $i\in\{0,1,2\}$, we deduce that $\beta_0 \in C_1(\Delta_n,\mathbb{X})$, and hence $\beta\in C_1(\Delta_n,\mathbb{X})+\varepsilon_1(C_2(\Delta,\mathbb{X}))$.\\

\noindent If $k>n$ then we will apply induction on $k-n$. The same argument shows that $\beta$ has a shift $\beta_0\in C_1(\Delta_k,\mathbb{X})$, so replacing $k$ with $k-1$ we can apply induction to find a shift $\beta_0'$ of $\beta_0$ with $\beta_0'\in C_1(\Delta_n,\mathbb{X})$ as required.\end{proof}

\noindent In light of this result, we will assume from now on that $\mathcal{X}\neq\Delta_n$, and thus there exists $i\in\{0,1,2\}$ such that $M_i\neq 0$. Therefore, we can assume that $r_i(\beta)\leq m_i\leq M_i\leq R_i(\beta)$.

Using Proposition \ref{propn: isolating each summit}, we can assume further that $R_i(\beta)\leq M_i+1$ and $r_i(\beta)\geq m_i-1$.

\begin{proposition}\label{propn: penultimate step}
There exists an $X_{i,n+1}$-shift $\beta_0$ of $\beta$ such that $R_i(\beta_0)\leq M_i$ and $r_i(\beta_0)\geq m_i$.
\end{proposition}

\begin{proof}

By symmetry, we only need to find a shift $\beta_0$ such that $R_i(\beta_0)\leq M_i$, and since $R_i(\beta)\leq M_i+1$, we can assume that $R_i(\beta)=M_i+1$. So let $j:=R_i(\beta)$.\\

\noindent The argument is largely identical to the proof of Proposition \ref{propn: isolating each summit}, so will refer to this proof where necessary. For each summit $D\in S_{j-1,i}^{(n-1)}$, define $\beta_D\in C_1(X_D,\mathbb{X})$ by $$\beta_D(e):=\begin{cases}\beta(e) & e\in X_D\\ 0\ &\text{otherwise}\end{cases}$$ Then for any peak $v\in P_{j,i}^{(n+1)}$, all edges adjacent to $v$ in $P_{j,i}^{(n+1)}\sqcup P_{j-1,i}^{(n+1)}\sqcup\Delta_n$ lie in $X_D$, and since $R_i(\beta)=j$, these are the only edges at which $\beta$ is non-zero, so $\varepsilon_0(\beta_D)(v)=\varepsilon_0(\beta)(v)=0$.

In particular, $\varepsilon_0(\beta_D)$ is non-zero only on vertices in $(\Delta_n\sqcup S_{j-1,i}^{(n+1)})\cap X_D$, so it follows from Conjecture \ref{conj: did not get it} that it is $X_D$-collapsible. Applying Proposition \ref{propn: loose end},  we can find a chain $\gamma_D\in C_2(\Delta,\mathbb{X})$ such that $$\beta_D+\varepsilon_1(\gamma_D)\in C_1((\Delta_n\sqcup S_{j-1,i}^{(n+1)})\cap X_D,\mathbb{X})$$ and since $X_D$ is isometric with a complete, $I$-invariant region of $\Delta$, it follows from Proposition \ref{propn: exactness1} that $\gamma_D\in C_2(X_D,\mathbb{X})\subseteq C_2(X_{i,n+1},\mathbb{X})$. Moreover, since $S_{j-1,i}^{(n+1)}\subseteq\mathcal{X}$, this implies that $\beta_D+\varepsilon_1(\gamma_D)\in C_1(\mathcal{X},\mathbb{X})$.\\

\noindent The argument in the proof of Proposition \ref{propn: isolating each summit} now shows that we can find an $X_{i,n+1}$-shift $\beta_0$ of $\beta$ that agrees with $\beta_D+\varepsilon_1(\gamma_D)$ on Crown$(X_{i,n+1})\cap X_D$ for all $D\in S_{j,i}^{(n-1)}$. In particular, $\beta_0$ is zero on edges adjacent to vertices in $P_{j,i}^{(n+1)}$, so $R_i(\beta_0)<j$. But since $j=M_i+1$, this implies that $R_i(\beta_0)\leq M_i$ as required.\end{proof}

\begin{corollary}\label{cor: final for unbroken}
Suppose $\Delta_n\subsetneq\mathcal{X}\subsetneq\Delta_{n+1}$ is a complete, $I$-invariant, unbroken region of $\Delta$. If we assume Conjecture \ref{conj: did not get it}, then $\mathcal{X}$ is $\Delta_{n+1}$-collapsible.
\end{corollary}

\begin{proof}

We only need to show that if $\beta\in C_1(\Delta_{n+1},\mathbb{X})$ is $(I,\mathcal{X})$-shift invariant and $\varepsilon_0(\beta)\in C_0(\mathcal{X},\mathbb{X})$ then there exists a shift $\beta'\in C_1(\mathcal{X},\mathbb{X})$ of $\beta$.\\

\noindent For each $i\in\{0,1,2\}$, if $M_i=m_i=0$ then the proof of Theorem \ref{thm: future lettered theorem} shows that $\beta$ has a $X_{i,n+1}$-shift which is zero on Crown$(X_{i,n+1})$, so we will assume that $M_i\neq 0$, and hence $r_i(\beta)\leq m_i\leq M_i\leq R_i(\beta)$.\\

\noindent If $r_i=R_i$ then this forces equality, and thus $\beta$ is zero on Crown$(X_{i,n+1})$ outside $\mathcal{X}$. Similarly, if $r_i=R_i-1$ then $R_i\leq M_i+1$ and $r_i\geq m_i-1$, so by Proposition \ref{propn: penultimate step} there exists an $X_{i,n+1}$ shift $\beta_0$ of $\beta$ such that $R_i=M_i$ and $r_i=m_i$. So we may assume from that $r_i<R_i-1$.\\

\noindent Therefore, if $R_i(\beta)>M_i+1$ or $r_i(\beta)<m_i-1$, then we can apply Proposition \ref{propn: isolating each summit} to find an $X_{i,n+1}$-shift $\beta_0$ of $\beta$ with $R_i(\beta_0)<R_i(\beta)$ or $r_i(\beta_0)>r_i(\beta)$, so we may assume that $R_i\leq M_i+1$ and $r_i\geq m_i-1$. So applying Proposition \ref{propn: penultimate step} again, we can find an $X_{i,n+1}$-shift $\beta_1$ of $\beta$ with $r_i(\beta_1)=m_i$ and $R_i(\beta_1)=M_i$.\\

\noindent Therefore, after applying a sequence of $X_{i,n+1}$-shifts, we may assume that $r_i(\beta)=m_i$ and $R_i(\beta)=M_i$, which means that the only edges in Crown$(X_{i,n+1})$ on which $\beta$ can be non-zero are those that are not adjacent to vertices in $P_{k,i}^{(n+1)}$ for $k$ outside $\{m_i,\dots,M_i\}$, i.e. edges in $\mathcal{X}\cap$ Crown$(X_{i,n+1})$. Since this is true for each $i\in\{0,1,2\}$, it follows that $\beta\in C_1(\mathcal{X},\mathbb{X})$.\end{proof}

\noindent The final step in the argument is to drop the assumption that $\mathcal{X}$ is unbroken. In this case, we will need the following result.

\begin{proposition}\label{propn: middle vertex}
If $i\in\{0,1,2\}$, $j\in\{2,\dots,m\}$ and $(j,i)\notin \Gamma(\mathcal{X})$ then there exist chains $\beta^+,\beta^-\in C_1(\Delta_{n+1},\mathbb{X})$ such that

\begin{itemize}

\item $\beta^++\beta^-$ is an $X_{i,n+1}$-shift of $\beta$.

\item $\varepsilon_0(\beta^+)(v)=\varepsilon_0(\beta^-)(v)=0$ for all peaks $v$ of $\Delta_{n+1}$ with $v\notin\mathcal{X}$.

\item $\beta^+$ (resp. $\beta^-$) is zero on all edges adjacent to vertices in $P_{k,i}^{(n+1)}$ for $k<j$ (resp. $k>j$).

\item $\beta^-$ is zero outside Crown$(X_{i,n+1})$.

\end{itemize}

\end{proposition}

\begin{proof}

Firstly, fix any summit $E\in S_{j,i}^{(n+1)}$, let $v$ be the peak of $E$, and since $(j,i)\notin \Gamma$, it follows that $v\notin\mathcal{X}$, so $\varepsilon_0(\beta)(v)=0$.\\ 

\noindent Let $e_v,f_v$ be the oriented edges of $E$ with target $v$, where the origin of $f_v$ (resp. $e_v$) lies in $P_{j+1,i}^{(n+1)}$ (resp. $P_{j-1,i}^{(n+1)}$) if $j<m$ (resp. $j>2$) and is a peak of $X_{i+1,n}$ (resp. $X_{i-1,n}$) if $j=m$ (resp. $j=2$). Let $e_{i,v},f_{i,v}$ be the oriented edges of $E_i$ and $F_i$ respectively with target $v$, these edges giving the arrangement in Figure 6.

An identical argument to the proof of Proposition \ref{propn: isolating each summit} shows that we can find a Crown$(X_{i,n+1})$ shift $\beta'$ of $\beta$ satisfying $$\beta'(e_v)+\underset{1\leq i\leq p}{\sum}{\beta'(e_{i,v})}=\beta'(f_v)+\underset{1\leq i\leq p}{\sum}{\beta'(f_{i,v})}=0$$

\noindent For simplicity, let $\Psi\subseteq$ Crown$(X_{i,n+1})$ be the set of all oriented edges in $\Delta_{n+1}$ adjacent to a peak in $P_{k,i}^{(n+1)}$ for some $k\geq j$, and define $\beta^-\in C_1(\Delta_{n+1},\mathbb{X})$ by $$\beta^-(e):=\begin{cases}\beta'(e) & e\in\Psi\text{ and }e\neq f_v,\sigma(f_v)\text{ for some }v\in P_{j,i}^{(n+1)}\\0 & \text{otherwise}\end{cases}$$ and let $\beta^+:=\beta'-\beta$. Then clearly $\beta^-$ is zero outside Crown$(X_{i,n+1})$ on all edges adjacent to vertices in $P_{k,i}^{(n+1)}$ for $k<j$, and $\beta^+$ is zero on all edges adjacent to vertices in $P_{k,i}^{(n+1)}$ for $k>j$, and since $\beta^+,\beta^-$ either agree with $\beta'$ or are zero, it follows that $\beta^+$ agrees with $\beta'$ (and hence $\beta$) on all edges in Crown$(X_{i',n+1})$ for $i'\neq i$.\\

\noindent Finally, if $v\notin\mathcal{X}$ is a peak of $\Delta_{n+1}$, then if $v\notin P_{k,i}^{(n+1)}$ for some $k\geq j$, then $\beta^-(e)=0$ for all edges adjacent to $v$, so $\varepsilon_0(\beta)(v)=0$, and $\varepsilon_0(\beta^+)(v)=\varepsilon_0(\beta')(v)=0$. Similarly, if $v\in P_{k,i}^{(n+1)}$ for $k>j$ then $\varepsilon_0(\beta^+)(v)=0$ and $\varepsilon_0(\beta^-)(v)=\varepsilon_0(\beta')(v)=0$.

If $v\in P_{j,i}^{n+1}$ then the only edges adjacent to $v$ at which $\beta^-$ is non-zero are $e_v,e_{1,v},\dots,e_{p,v}$, so $$\varepsilon_0(\beta^-)(v)=\beta^-(e_v)+\underset{1\leq i\leq p}{\sum}{\beta^-(e_{i,v})}=\beta'(e_v)+\underset{1\leq i\leq p}{\sum}{\beta'(e_{i,v})}=0$$ and thus $\varepsilon_0(\beta^+)(v)=\varepsilon_0(\beta')(v)-\varepsilon_0(\beta^-)(v)=0$.\end{proof}

\noindent Now, for any integers $r\leq R$, we write $[r,R]$ to mean the interval $\{j\in\mathbb{N}:r\leq j\leq R\}$. Recall that $\Gamma(\mathcal{X})=(\Gamma_0(\mathcal{X})\times\{0\})\sqcup (\Gamma_1(\mathcal{X})\times\{1\})\sqcup (\Gamma_2(\mathcal{X})\times\{2\})$, and we can decompose $\Gamma_i(\mathcal{X})$ into a disjoint union $$\Gamma_i(\mathcal{X})=[a_{i,1},A_{i,1}]\sqcup\dots\sqcup [a_{i,t_i},A_{i,t_i}]$$ where $a_{i,j}\leq A_{i,j}$, $A_{i,j}<a_{i,j+1}$, and note that $\mathcal{X}$ is unbroken if and only if $t_i=t_i(\mathcal{X})=1$ for all $i\in\{0,1,2\}$.

\begin{theorem}\label{thm: final step}
Let $\mathcal{X}$ be a complete, $I$-invariant region of $\Delta$. If we assume Conjecture \ref{conj: did not get it}, then $\mathcal{X}$ is a collapsible region of $\Delta$.
\end{theorem}

\begin{proof}


\noindent Let $t=t(\mathcal{X}):=t_0(\mathcal{X})+t_1(\mathcal{X})+t_2(\mathcal{X})$, and if $t=0$ then $\mathcal{X}=\Delta_n$, which is collapsible by Theorem \ref{thm: future lettered theorem}. So we proceed by induction on $t$.

Moreover, if $t_i=1$ for all $i$, then $\mathcal{X}$ is unbroken, and the result follows from Corollary \ref{cor: final for unbroken}. So we may assume that $t_i>1$ for some $i$, so $A_{i,1}<a_{i,2}$. We will assume without loss of generality that $i=0$.\\

\noindent Since $\mathcal{X}\subseteq\Delta_{n+1}$ and $\Delta_{n+1}$ is collapsible by Theorem \ref{thm: future lettered theorem} again, we only need to show that if $\beta\in C_1(\Delta_{n+1},\mathbb{X})$ is $(I,\mathcal{X})$-shift invariant and $\varepsilon_0(\beta)\in C_0(\mathcal{X},\mathbb{X})$, then $\beta\in C_1(\mathcal{X},\mathbb{X})+\varepsilon_1(C_2(\Delta,\mathbb{X}))$.\\

\noindent Set $j:=A_{0,1}+1\notin\Gamma_i(\mathcal{X})$, and we see that $(j,i)\notin\Gamma$. Thus we can find chains $\beta^+,\beta^-\in C_1(\Delta_{n+1},\mathbb{X})$ satisfying the conditions stated Proposition \ref{propn: middle vertex}. In particular, if we define $$\mathcal{X}^+:=\mathcal{X}\backslash\underset{k<j}{\bigcup}S_{k,0}^{(n+1)}$$ and $$\mathcal{X}^-:=\mathcal{X}\backslash\underset{k>j}{\bigcup}S_{k,0}^{(n+1)}$$ then $\mathcal{X}=\mathcal{X}^+\cup\mathcal{X}^-$, $\varepsilon_0(\beta^+)\in C_0(\mathcal{X}^+,\mathbb{X})$ and $\varepsilon_0(\beta^-)\in C_0(\mathcal{X}^-,\mathbb{X})$.\\

\noindent But $\Gamma_{i'}(\mathcal{X}^+)=\Gamma_{i'}(\mathcal{X})$ and $\Gamma_{i'}(\mathcal{X}^-)=\varnothing$ for $i'\neq 0$, while $\Gamma_{i}(\mathcal{X}^+)=[a_{0,1},A_{0,1}]$ and $\Gamma_i(\mathcal{X}^-)=[a_{0,2},A_{0,2}]\sqcup\dots\sqcup [a_{0,t_0},b_{0,t_0}]$. Thus $t(\mathcal{X}^+)=t-t_0+1<t$ and $t(\mathcal{X}^-)=t-t_1-t_2-1<t$.

By induction $\mathcal{X}^+$ and $\mathcal{X}^-$ are collapsible regions of $\Delta$, so $\beta^+\in C_1(\mathcal{X}^+,\mathbb{X})+\varepsilon_1(C_2(\Delta,\mathbb{X}))$ and $\beta^-\in C_1(\mathcal{X}^-,\mathbb{X})+\varepsilon_1(C_2(\Delta,\mathbb{X}))$ by Proposition \ref{propn: loose end}.\\

\noindent Finally, $\beta^++\beta^-$ is a shift of $\beta$, so $\beta-\beta^+-\beta^-\in\varepsilon_1(C_2(\Delta,\mathbb{X}))$. So since $C_1(\mathcal{X}^+,\mathbb{X}),C_1(\mathcal{X}^-,\mathbb{X})\subseteq C_1(\mathcal{X},\mathbb{X})$ it follows that $\beta\in C_1(\mathcal{X},\mathbb{X})+\varepsilon_1(C_2(\Delta,\mathbb{X})$ as required.\end{proof}

\noindent We can now complete the proof of our main theorem:\\


\noindent\emph{Proof of Theorem \ref{letterthm: exact for star}.} We only need to prove Conjecture \ref{conj: exactness 1} for $G=SL_3(K)$, and the result follows from Theorem \ref{letterthm: implies sur}. In other words, we only need to show that if \textbf{(A)} $\mathcal{X}$ is an $I$-invariant complete region of $\Delta$ or \textbf{(B)} $\mathcal{X}$ consists of a single face of $C$, then the local oriented chain complex $$0\to C_2(\mathcal{X},\mathbb{X})\to C_1(\mathcal{X},\mathbb{X})\to C_0(\mathcal{X},\mathbb{X})\to\mathbb{X}$$ is exact.\\

\noindent In case \textbf{(B)}, this follows from Lemma \ref{lem: (A) implies (B)}, so we may assume that $\mathcal{X}$ is complete and $I$-invariant.\\

\noindent The assumption that the local oriented chain complex is exact for any star bounded region implies that any star bounded region $\mathcal{Y}\subseteq\text{ Star}(v_0)$ is Star$(v_0)$-collapsible by Lemma \ref{lem: one way}, i.e. Conjecture \ref{conj: did not get it} holds. 

Therefore, it follows from Theorem \ref{thm: final step} that $\mathcal{X}$ is collapsible. So since $0\to C_2(\mathcal{X},\mathbb{X})\to C_1(\mathcal{X},\mathbb{X})\to C_0(\Delta,\mathbb{X})$ is exact by Proposition \ref{propn: exactness1}, it follows from Corollary \ref{cor: extended regions} that the local oriented chain complex is exact.\qed

\section{Chains on Star$(v_0)$: Technical results}\label{sec: small cases}

In this section, we will complete the proof of the remaining technical results that we stated in section \ref{sec: orbits}, and also explore some avenues which we hope will ultimately yield a proof of Conjecture \ref{conj: did not get it}. These arguments require us to consider chains on small regions of $\Delta$ more closely. 

Throughout, we will assume that $G=SL_3(K)$. Initially, we make no assumptions on the extension $K/\mathbb{Q}_p$, but later we will assert that it is totally ramified.

\subsection{Notation}\label{subsec: notation}

We will now fix a notation for every vertex, edge and chamber in Star$(v_0)$ that we will refer to until the end of the paper. Again, we let $v_0,v_1,v_2$ be the vertices of $C$, and we now write $s_i$ for the edge $\{v_{i-1},v_{i+1}\}$ (subscripts modulo 3). Let $[q]:=\{1,2,\dots,q\}$)

\begin{itemize}
    \item Label by $P_1,\dots,P_q$ the summits of $\Delta_1$ based at $\{v_0,v_1\}$, $Q_1,\dots,Q_q$ the summits based at $\{v_0,v_2\}$. So Crown$(X_{1,1})=\{P_1,\dots,P_q\}$ and Crown$(X_{2,1})=\{Q_1,\dots,Q_q\}$.
    
    \item For each $i=1,\dots,q$, let $u_i$ be the peak of $P_i$, $w_i$ the peak $Q_i$. Let $e_i$ (resp. $d_i$) be the edge connecting $u_i$ (resp $w_i$) to $v_1$ (resp. $v_2$). Let $h_i$ (resp. $k_i$) be the edge connecting $u_i$ (resp $w_i$) to $v_0$.

    \item For $i,j=1,\dots,q$, label by $P_{j,i}$ (resp. $Q_{j,i}$) the summits of $\Delta_2$ which define $S_1^{(2)}$ (resp. $S_2^{2)}$). Let $u_{j,i}$ (resp. $w_{j,i}$) be their peaks. We can assume that $h_i$ (resp. $k_i$) forns the base of $P_{j,i}$ (resp. $Q_{j,i}$).

\item  Let $e_{j,i}$ (resp. $d_{j,i}$) be the edge connecting $u_{j,i}$ to $u_i$ (resp. $w_{j,i}$ to $w_i$). Let $h_{j,i}$ (resp. $k_{j,i}$) be the edge connecting $u_{j,i}$ (resp. $w_{j,i})$ to $v_0$.

    \item For each pair $(j,i)\in [q]^2$, there are precisely $q$ chambers $D_{(1,j,i)},\dots,D_{(q,j,i)}$ adjacent to $P_{j,i}$ via $h_{j,i}$, and each adjacent to a chamber in $S_2^{(2)}$. 

    \item $r_{k,j,i}$ is the edge of $D_{k,j,i}$ which joins the peak $u_{j,i}\in P_{1}^{(2)}$ to a peak in $P_2^{(2)}$. 

\item For each $(k,j,i)\in [q]^3$, let $Q(k,j,i)$ be the chamber in $S_2^{(2)}$ which is adjacent to $D_{k,j,i}$. Note that $Q(k,j,i)=Q_{\ell,m}$ for some $1\leq\ell,m\leq q$. By Theorem \ref{thm: d-partite}(2), $Q(k,j,i)$ and $Q(k',j,i)$ do not share a base if $k\neq k'$.
\end{itemize}

\noindent\textbf{Note:} In the definition of the chambers $D_{k,j,i}$, we index them via their adjacent chambers in $S_1^{(2)}$, but by symmetry, we could define them (and the edges $r_{k,j,i}$) using their adjacent chambers in $S_2^{(2)}$, it would only amount to a change of indexing.\\

\noindent The diagrams below illustrate this structure in the case where $q=2$. Figure 8 gives an illustration of all the chambers in Star$(v_0)$, labelled by the data above (though we do not include all the chambers $D_{k,j,i}$, and we do not label all edges, as this would become cumbersome). The colours of each chamber indicate their distance from the hyperspecial chamber $C$, and note that the chambers in blue comprise the crown ofStar$(v_0)$.

Figure 9 describes the bipartite graph defined by the peaks of $\Delta_2$ in Star$(v_0)$, and Figure 10 illustrates the chambers of Star$(v_0)$ that lie in a single apartment of $\Delta$, which we will assume to be the standard apartment.

\

\begin{center}
\tikzset{every picture/.style={line width=0.75pt}} 

\begin{tikzpicture}[x=0.75pt,y=0.75pt,yscale=-1,xscale=1]

\draw   (331.42,264.97) -- (462.76,455.97) -- (200.08,455.97) -- cycle ;
\draw  [color={rgb, 255:red, 245; green, 166; blue, 35 }  ,draw opacity=1 ] (78.68,276.09) -- (331.42,264.97) -- (199.78,455.73) -- cycle ;
\draw  [color={rgb, 255:red, 245; green, 166; blue, 35 }  ,draw opacity=1 ] (585.81,267.68) -- (465.24,458.93) -- (331.75,265.21) -- cycle ;
\draw  [color={rgb, 255:red, 245; green, 152; blue, 35 }  ,draw opacity=1 ][dash pattern={on 4.5pt off 4.5pt}] (142.11,184.87) -- (331.74,265.22) -- (200.1,455.97) -- cycle ;
\draw  [color={rgb, 255:red, 245; green, 166; blue, 35 }  ,draw opacity=1 ][dash pattern={on 4.5pt off 4.5pt}] (537.83,198.05) -- (465.24,458.93) -- (331.75,265.21) -- cycle ;
\draw  [color={rgb, 255:red, 74; green, 144; blue, 226 }  ,draw opacity=1 ] (90.1,98.6) -- (331.46,263.37) -- (78.68,276.89) -- cycle ;
\draw  [color={rgb, 255:red, 74; green, 144; blue, 226 }  ,draw opacity=1 ] (570.63,93.35) -- (585.79,267.67) -- (331.44,263.36) -- cycle ;
\draw [color={rgb, 255:red, 74; green, 144; blue, 226 }  ,draw opacity=1 ] [dash pattern={on 0.84pt off 2.51pt}]  (117.89,120.79) -- (78.71,276.85) ;
\draw [color={rgb, 255:red, 74; green, 144; blue, 226 }  ,draw opacity=1 ]   (128.46,81.41) -- (117.89,120.79) ;
\draw [color={rgb, 255:red, 74; green, 144; blue, 226 }  ,draw opacity=1 ]   (128.46,81.41) -- (331.46,263.37) ;
\draw [color={rgb, 255:red, 74; green, 144; blue, 226 }  ,draw opacity=1 ] [dash pattern={on 0.84pt off 2.51pt}]  (172.14,122.66) -- (142.12,184.85) ;
\draw [color={rgb, 255:red, 74; green, 144; blue, 226 }  ,draw opacity=1 ]   (202.94,58.44) -- (172.14,122.66) ;
\draw [color={rgb, 255:red, 74; green, 144; blue, 226 }  ,draw opacity=1 ]   (202.94,58.44) -- (217.5,81.47) ;
\draw [color={rgb, 255:red, 74; green, 144; blue, 226 }  ,draw opacity=1 ] [dash pattern={on 0.84pt off 2.51pt}]  (217.5,81.47) -- (330.45,262.04) ;
\draw [color={rgb, 255:red, 74; green, 144; blue, 226 }  ,draw opacity=1 ] [dash pattern={on 0.84pt off 2.51pt}]  (182.81,129.22) -- (142.96,184.86) ;
\draw [color={rgb, 255:red, 74; green, 144; blue, 226 }  ,draw opacity=1 ]   (182.81,129.22) -- (232.84,61.54) ;
\draw [color={rgb, 255:red, 74; green, 144; blue, 226 }  ,draw opacity=1 ]   (232.84,61.54) -- (330.45,262.04) ;
\draw [color={rgb, 255:red, 74; green, 144; blue, 226 }  ,draw opacity=1 ] [dash pattern={on 0.84pt off 2.51pt}]  (540.35,115.17) -- (585.79,267.67) ;
\draw [color={rgb, 255:red, 74; green, 144; blue, 226 }  ,draw opacity=1 ]   (535.01,101) -- (540.35,115.17) ;
\draw [color={rgb, 255:red, 74; green, 144; blue, 226 }  ,draw opacity=1 ] [dash pattern={on 0.84pt off 2.51pt}]  (478.09,143.26) -- (537.5,196.97) ;
\draw [color={rgb, 255:red, 74; green, 144; blue, 226 }  ,draw opacity=1 ]   (413.78,83.57) -- (478.09,143.26) ;
\draw [color={rgb, 255:red, 74; green, 144; blue, 226 }  ,draw opacity=1 ]   (331.46,263.37) -- (413.78,83.57) ;
\draw [color={rgb, 255:red, 74; green, 144; blue, 226 }  ,draw opacity=1 ] [dash pattern={on 0.84pt off 2.51pt}]  (486.98,133.9) -- (537.5,196.97) ;
\draw [color={rgb, 255:red, 74; green, 144; blue, 226 }  ,draw opacity=1 ]   (439.49,75.04) -- (489.65,137.53) ;
\draw [color={rgb, 255:red, 74; green, 144; blue, 226 }  ,draw opacity=1 ]   (427.39,98.32) -- (439.49,75.04) ;
\draw [color={rgb, 255:red, 74; green, 144; blue, 226 }  ,draw opacity=1 ] [dash pattern={on 0.84pt off 2.51pt}]  (331.46,263.37) -- (427.39,98.32) ;
\draw [color={rgb, 255:red, 208; green, 2; blue, 27 }  ,draw opacity=1 ]   (90.1,98.6) -- (570.63,93.36) ;
\draw [color={rgb, 255:red, 208; green, 2; blue, 2 }  ,draw opacity=1 ]   (90.1,98.6) -- (439.49,75.04) ;
\draw [color={rgb, 255:red, 74; green, 144; blue, 226 }  ,draw opacity=1 ]   (331.45,263.36) -- (535.01,101) ;

\draw (322,280.68) node [anchor=north west][inner sep=0.75pt]   [align=left] {$\displaystyle v_{0}$};
\draw (188.52,467.78) node [anchor=north west][inner sep=0.75pt]   [align=left] {$\displaystyle v_{1}$};
\draw (462.32,465.95) node [anchor=north west][inner sep=0.75pt]   [align=left] {$\displaystyle v_{2}$};
\draw (60.59,277.88) node [anchor=north west][inner sep=0.75pt]  [color={rgb, 255:red, 245; green, 152; blue, 35 }  ,opacity=1 ] [align=left] {$\displaystyle u_{1}$};
\draw (589.27,262.29) node [anchor=north west][inner sep=0.75pt]  [color={rgb, 255:red, 245; green, 152; blue, 35 }  ,opacity=1 ] [align=left] {$\displaystyle w_{1}$};
\draw (320.63,374.66) node [anchor=north west][inner sep=0.75pt]   [align=left] {$\displaystyle C$};
\draw (116.87,288.42) node [anchor=north west][inner sep=0.75pt]  [color={rgb, 255:red, 245; green, 152; blue, 35 }  ,opacity=1 ] [align=left] {$\displaystyle P_{1}$};
\draw (168.35,216.22) node [anchor=north west][inner sep=0.75pt]  [color={rgb, 255:red, 245; green, 152; blue, 35 }  ,opacity=1 ] [align=left] {$\displaystyle P_{2}$};
\draw (531.64,289.97) node [anchor=north west][inner sep=0.75pt]  [color={rgb, 255:red, 245; green, 152; blue, 35 }  ,opacity=1 ] [align=left] {$\displaystyle Q_{1}$};
\draw (490.5,227.55) node [anchor=north west][inner sep=0.75pt]  [color={rgb, 255:red, 245; green, 152; blue, 35 }  ,opacity=1 ] [align=left] {$\displaystyle Q_{2}$};
\draw (117.78,353.5) node [anchor=north west][inner sep=0.75pt]  [color={rgb, 255:red, 245; green, 152; blue, 35 }  ,opacity=1 ] [align=left] {$\displaystyle e_{1}$};
\draw (71.65,76.61) node [anchor=north west][inner sep=0.75pt]  [color={rgb, 255:red, 74; green, 144; blue, 226 }  ,opacity=1 ] [align=left] {$\displaystyle u_{1,1}$};
\draw (116.12,62.57) node [anchor=north west][inner sep=0.75pt]  [color={rgb, 255:red, 74; green, 144; blue, 226 }  ,opacity=1 ] [align=left] {$\displaystyle u_{1,2}$};
\draw (123.73,175.84) node [anchor=north west][inner sep=0.75pt]  [color={rgb, 255:red, 245; green, 152; blue, 35 }  ,opacity=1 ] [align=left] {$\displaystyle u_{2}$};
\draw (227.29,41.79) node [anchor=north west][inner sep=0.75pt]  [color={rgb, 255:red, 74; green, 144; blue, 226 }  ,opacity=1 ] [align=left] {$\displaystyle u_{2,1}$};
\draw (189.94,36.36) node [anchor=north west][inner sep=0.75pt]  [color={rgb, 255:red, 74; green, 144; blue, 226 }  ,opacity=1 ] [align=left] {$\displaystyle u_{2,2}$};
\draw (539.18,189.27) node [anchor=north west][inner sep=0.75pt]  [color={rgb, 255:red, 245; green, 152; blue, 35 }  ,opacity=1 ] [align=left] {$\displaystyle w_{2}$};
\draw (571.59,77.87) node [anchor=north west][inner sep=0.75pt]  [color={rgb, 255:red, 74; green, 144; blue, 226 }  ,opacity=1 ] [align=left] {$\displaystyle w_{1,1}$};
\draw (520.01,71.32) node [anchor=north west][inner sep=0.75pt]  [color={rgb, 255:red, 74; green, 144; blue, 226 }  ,opacity=1 ] [align=left] {$\displaystyle w_{1,2}$};
\draw (396.38,60.09) node [anchor=north west][inner sep=0.75pt]  [color={rgb, 255:red, 74; green, 144; blue, 226 }  ,opacity=1 ] [align=left] {$\displaystyle w_{2,1}$};
\draw (429.29,51.66) node [anchor=north west][inner sep=0.75pt]  [color={rgb, 255:red, 74; green, 144; blue, 226 }  ,opacity=1 ] [align=left] {$\displaystyle w_{2,2}$};
\draw (168.48,288.9) node [anchor=north west][inner sep=0.75pt]  [color={rgb, 255:red, 245; green, 152; blue, 35 }  ,opacity=1 ] [align=left] {$\displaystyle e_{2}$};
\draw (532.13,353.5) node [anchor=north west][inner sep=0.75pt]  [color={rgb, 255:red, 245; green, 152; blue, 35 }  ,opacity=1 ] [align=left] {$\displaystyle d_{1}$};
\draw (482.32,308.56) node [anchor=north west][inner sep=0.75pt]  [color={rgb, 255:red, 245; green, 152; blue, 35 }  ,opacity=1 ] [align=left] {$\displaystyle d_{2}$};
\draw (311.47,99.05) node [anchor=north west][inner sep=0.75pt]  [color={rgb, 255:red, 208; green, 2; blue, 27 }  ,opacity=1 ] [align=left] {$\displaystyle r_{1,1,1}$};
\draw (87.03,124.6) node [anchor=north west][inner sep=0.75pt]  [color={rgb, 255:red, 74; green, 144; blue, 226 }  ,opacity=1 ] [align=left] {$\displaystyle P_{1,1}$};
\draw (127.94,101.19) node [anchor=north west][inner sep=0.75pt]  [color={rgb, 255:red, 74; green, 144; blue, 226 }  ,opacity=1 ] [align=left] {$\displaystyle P_{1,2}$};
\draw (60.47,162.39) node [anchor=north west][inner sep=0.75pt]  [color={rgb, 255:red, 74; green, 144; blue, 226 }  ,opacity=1 ] [align=left] {$\displaystyle e_{1,1}$};
\draw (97.83,193.58) node [anchor=north west][inner sep=0.75pt]  [color={rgb, 255:red, 74; green, 144; blue, 226 }  ,opacity=1 ] [align=left] {$\displaystyle e_{1,2}$};
\draw (429.27,84.5) node [anchor=north west][inner sep=0.75pt]  [color={rgb, 255:red, 74; green, 144; blue, 226 }  ,opacity=1 ] [align=left] {$\displaystyle Q_{2,2}$};
\draw (402.59,95.85) node [anchor=north west][inner sep=0.75pt]  [color={rgb, 255:red, 74; green, 144; blue, 226 }  ,opacity=1 ] [align=left] {$\displaystyle Q_{2,1}$};
\draw (307.28,61.54) node [anchor=north west][inner sep=0.75pt]  [color={rgb, 255:red, 208; green, 2; blue, 27 }  ,opacity=1 ] [align=left] {$\displaystyle r_{1,1,2}$};
\draw (470.42,100.02) node [anchor=north west][inner sep=0.75pt]  [color={rgb, 255:red, 74; green, 144; blue, 226 }  ,opacity=1 ] [align=left] {$\displaystyle d_{2,2}$};
\draw (441.08,124.08) node [anchor=north west][inner sep=0.75pt]  [color={rgb, 255:red, 74; green, 144; blue, 226 }  ,opacity=1 ] [align=left] {$\displaystyle d_{2,1}$};
\draw (165.67,491.54) node [anchor=north west][inner sep=0.75pt]   [align=left] {\textbf{Figure 8}: The region $\displaystyle \text{Star}(v_0)$ in the $\displaystyle \tilde{A}_{2}(\mathbb{Q}_{2})$ building};
\draw (319.78,457.34) node [anchor=north west][inner sep=0.75pt]   [align=left] {$\displaystyle s_{0}$};
\draw (252.19,339.72) node [anchor=north west][inner sep=0.75pt]   [align=left] {$\displaystyle s_{2}$};
\draw (395.38,338.82) node [anchor=north west][inner sep=0.75pt]   [align=left] {$\displaystyle s_{1}$};

\end{tikzpicture}
\end{center}

\begin{center}

\tikzset{every picture/.style={line width=0.75pt}} 

\begin{tikzpicture}[x=0.75pt,y=0.75pt,yscale=-1,xscale=1]

\draw  [color={rgb, 255:red, 0; green, 0; blue, 0 }  ,draw opacity=1 ][fill={rgb, 255:red, 0; green, 0; blue, 0 }  ,fill opacity=1 ] (65,62.79) .. controls (65,60.7) and (66.54,59) .. (68.43,59) .. controls (70.33,59) and (71.87,60.7) .. (71.87,62.79) .. controls (71.87,64.89) and (70.33,66.58) .. (68.43,66.58) .. controls (66.54,66.58) and (65,64.89) .. (65,62.79) -- cycle ;
\draw  [color={rgb, 255:red, 0; green, 0; blue, 0 }  ,draw opacity=1 ][fill={rgb, 255:red, 0; green, 0; blue, 0 }  ,fill opacity=1 ] (176.22,64.31) .. controls (176.22,62.21) and (177.76,60.52) .. (179.66,60.52) .. controls (181.55,60.52) and (183.09,62.21) .. (183.09,64.31) .. controls (183.09,66.4) and (181.55,68.1) .. (179.66,68.1) .. controls (177.76,68.1) and (176.22,66.4) .. (176.22,64.31) -- cycle ;
\draw  [color={rgb, 255:red, 0; green, 0; blue, 0 }  ,draw opacity=1 ][fill={rgb, 255:red, 0; green, 0; blue, 0 }  ,fill opacity=1 ] (313.54,62.79) .. controls (313.54,60.7) and (315.07,59) .. (316.97,59) .. controls (318.87,59) and (320.4,60.7) .. (320.4,62.79) .. controls (320.4,64.89) and (318.87,66.58) .. (316.97,66.58) .. controls (315.07,66.58) and (313.54,64.89) .. (313.54,62.79) -- cycle ;
\draw  [color={rgb, 255:red, 0; green, 0; blue, 0 }  ,draw opacity=1 ][fill={rgb, 255:red, 0; green, 0; blue, 0 }  ,fill opacity=1 ] (424.76,64.31) .. controls (424.76,62.21) and (426.3,60.52) .. (428.19,60.52) .. controls (430.09,60.52) and (431.63,62.21) .. (431.63,64.31) .. controls (431.63,66.4) and (430.09,68.1) .. (428.19,68.1) .. controls (426.3,68.1) and (424.76,66.4) .. (424.76,64.31) -- cycle ;
\draw  [color={rgb, 255:red, 0; green, 0; blue, 0 }  ,draw opacity=1 ][fill={rgb, 255:red, 0; green, 0; blue, 0 }  ,fill opacity=1 ] (66.37,191.69) .. controls (66.37,189.6) and (67.91,187.9) .. (69.81,187.9) .. controls (71.7,187.9) and (73.24,189.6) .. (73.24,191.69) .. controls (73.24,193.79) and (71.7,195.48) .. (69.81,195.48) .. controls (67.91,195.48) and (66.37,193.79) .. (66.37,191.69) -- cycle ;
\draw  [color={rgb, 255:red, 0; green, 0; blue, 0 }  ,draw opacity=1 ][fill={rgb, 255:red, 0; green, 0; blue, 0 }  ,fill opacity=1 ] (177.6,193.21) .. controls (177.6,191.11) and (179.13,189.42) .. (181.03,189.42) .. controls (182.93,189.42) and (184.46,191.11) .. (184.46,193.21) .. controls (184.46,195.3) and (182.93,197) .. (181.03,197) .. controls (179.13,197) and (177.6,195.3) .. (177.6,193.21) -- cycle ;
\draw  [color={rgb, 255:red, 0; green, 0; blue, 0 }  ,draw opacity=1 ][fill={rgb, 255:red, 0; green, 0; blue, 0 }  ,fill opacity=1 ] (314.91,191.69) .. controls (314.91,189.6) and (316.45,187.9) .. (318.34,187.9) .. controls (320.24,187.9) and (321.78,189.6) .. (321.78,191.69) .. controls (321.78,193.79) and (320.24,195.48) .. (318.34,195.48) .. controls (316.45,195.48) and (314.91,193.79) .. (314.91,191.69) -- cycle ;
\draw  [color={rgb, 255:red, 0; green, 0; blue, 0 }  ,draw opacity=1 ][fill={rgb, 255:red, 0; green, 0; blue, 0 }  ,fill opacity=1 ] (426.13,193.21) .. controls (426.13,191.11) and (427.67,189.42) .. (429.57,189.42) .. controls (431.46,189.42) and (433,191.11) .. (433,193.21) .. controls (433,195.3) and (431.46,197) .. (429.57,197) .. controls (427.67,197) and (426.13,195.3) .. (426.13,193.21) -- cycle ;
\draw [color={rgb, 255:red, 208; green, 2; blue, 2 }  ,draw opacity=1 ]   (68.43,62.79) -- (69.81,187.9) ;
\draw [color={rgb, 255:red, 208; green, 2; blue, 2 }  ,draw opacity=1 ]   (68.43,62.79) -- (429.57,193.21) ;
\draw    (69.81,191.69) -- (313.54,62.79) ;
\draw    (179.66,64.31) -- (181.03,193.21) ;
\draw    (179.66,64.31) -- (318.34,191.69) ;
\draw    (316.97,62.79) -- (318.34,191.69) ;
\draw    (184.46,193.21) -- (428.19,64.31) ;
\draw    (428.19,64.31) -- (429.57,193.21) ;

\draw (487,51) node [anchor=north west][inner sep=0.75pt]   [align=left] {$\displaystyle P_{1}^{(2)}$};
\draw (484,189) node [anchor=north west][inner sep=0.75pt]   [align=left] {$\displaystyle P_{2}^{(2)}$};
\draw (51,34) node [anchor=north west][inner sep=0.75pt]   [align=left] {$\displaystyle u_{1,1}$};
\draw (166,36) node [anchor=north west][inner sep=0.75pt]   [align=left] {$\displaystyle u_{1,2}$};
\draw (303,37) node [anchor=north west][inner sep=0.75pt]   [align=left] {$\displaystyle u_{2,1}$};
\draw (415,38) node [anchor=north west][inner sep=0.75pt]   [align=left] {$\displaystyle u_{2,2}$};
\draw (53,202) node [anchor=north west][inner sep=0.75pt]   [align=left] {$\displaystyle w_{1,1}$};
\draw (167,196) node [anchor=north west][inner sep=0.75pt]   [align=left] {$\displaystyle w_{1,2}$};
\draw (307,201) node [anchor=north west][inner sep=0.75pt]   [align=left] {$\displaystyle w_{2,1}$};
\draw (417,201) node [anchor=north west][inner sep=0.75pt]   [align=left] {$\displaystyle w_{2,2}$};
\draw (35,110) node [anchor=north west][inner sep=0.75pt]  [color={rgb, 255:red, 208; green, 2; blue, 2 }  ,opacity=1 ] [align=left] {$\displaystyle r_{1,1,1}$};
\draw (349.33,141.76) node [anchor=north west][inner sep=0.75pt]  [color={rgb, 255:red, 208; green, 2; blue, 2 }  ,opacity=1 ] [align=left] {$\displaystyle r_{1,1,2}$};
\draw (150,239) node [anchor=north west][inner sep=0.75pt]   [align=left] {\textbf{Figure 9}: The peaks of $\displaystyle \text{Star}(v_0)$ when $q=2$};

\end{tikzpicture}

\

\tikzset{every picture/.style={line width=0.75pt}} 

\begin{tikzpicture}[x=0.75pt,y=0.75pt,yscale=-1,xscale=1]

\draw   (320,210) -- (435,379) -- (205,379) -- cycle ;
\draw  [color={rgb, 255:red, 245; green, 152; blue, 35 }  ,draw opacity=1 ] (81.27,209.45) -- (319.81,209.36) -- (204.36,378.04) -- cycle ;
\draw  [color={rgb, 255:red, 245; green, 152; blue, 35 }  ,draw opacity=1 ] (551.39,209.35) -- (434.8,379.13) -- (319.81,209.36) -- cycle ;
\draw  [color={rgb, 255:red, 74; green, 144; blue, 226 }  ,draw opacity=1 ] (200.54,40.45) -- (319.81,209.36) -- (81.27,209.36) -- cycle ;
\draw  [color={rgb, 255:red, 74; green, 144; blue, 226 }  ,draw opacity=1 ] (435.6,40.45) -- (551.39,209.35) -- (319.81,209.35) -- cycle ;
\draw [color={rgb, 255:red, 208; green, 2; blue, 2 }  ,draw opacity=1 ]   (200.54,40.45) -- (435.6,40.45) ;

\draw (312,300) node [anchor=north west][inner sep=0.75pt]   [align=left] {$\displaystyle C$};
\draw (312,225) node [anchor=north west][inner sep=0.75pt]   [align=left] {$\displaystyle v_{0}$};
\draw (189,383) node [anchor=north west][inner sep=0.75pt]   [align=left] {$\displaystyle v_{1}$};
\draw (436.8,382.13) node [anchor=north west][inner sep=0.75pt]   [align=left] {$\displaystyle v_{2}$};
\draw (186,256) node [anchor=north west][inner sep=0.75pt]  [color={rgb, 255:red, 245; green, 152; blue, 35 }  ,opacity=1 ] [align=left] {$\displaystyle P_{1}$};
\draw (442,246) node [anchor=north west][inner sep=0.75pt]  [color={rgb, 255:red, 245; green, 152; blue, 35 }  ,opacity=1 ] [align=left] {$\displaystyle Q_{1}$};
\draw (119,283) node [anchor=north west][inner sep=0.75pt]  [color={rgb, 255:red, 245; green, 152; blue, 35 }  ,opacity=1 ] [align=left] {$\displaystyle e_{1}$};
\draw (510,285) node [anchor=north west][inner sep=0.75pt]  [color={rgb, 255:red, 245; green, 152; blue, 35 }  ,opacity=1 ] [align=left] {$\displaystyle d_{1}$};
\draw (54,198.19) node [anchor=north west][inner sep=0.75pt]  [color={rgb, 255:red, 245; green, 152; blue, 35 }  ,opacity=1 ] [align=left] {$\displaystyle u_{1}$};
\draw (560,199.19) node [anchor=north west][inner sep=0.75pt]  [color={rgb, 255:red, 245; green, 152; blue, 35 }  ,opacity=1 ] [align=left] {$\displaystyle w_{1}$};
\draw (181,133) node [anchor=north west][inner sep=0.75pt]  [color={rgb, 255:red, 74; green, 144; blue, 226 }  ,opacity=1 ] [align=left] {$\displaystyle P_{1,1}$};
\draw (423,131) node [anchor=north west][inner sep=0.75pt]  [color={rgb, 255:red, 74; green, 144; blue, 226 }  ,opacity=1 ] [align=left] {$\displaystyle Q_{1,1}$};
\draw (185,15.19) node [anchor=north west][inner sep=0.75pt]  [color={rgb, 255:red, 74; green, 144; blue, 226 }  ,opacity=1 ] [align=left] {$\displaystyle u_{1,1}$};
\draw (426,14.19) node [anchor=north west][inner sep=0.75pt]  [color={rgb, 255:red, 74; green, 144; blue, 226 }  ,opacity=1 ] [align=left] {$\displaystyle w_{1,1}$};
\draw (111,98) node [anchor=north west][inner sep=0.75pt]  [color={rgb, 255:red, 74; green, 144; blue, 226 }  ,opacity=1 ] [align=left] {$\displaystyle e_{1,1}$};
\draw (499,101) node [anchor=north west][inner sep=0.75pt]  [color={rgb, 255:red, 74; green, 144; blue, 226 }  ,opacity=1 ] [align=left] {$\displaystyle d_{1,1}$};
\draw (301,12) node [anchor=north west][inner sep=0.75pt]  [color={rgb, 255:red, 208; green, 2; blue, 2 }  ,opacity=1 ] [align=left] {$\displaystyle r_{1,1,1}$};
\draw (301,86) node [anchor=north west][inner sep=0.75pt]  [color={rgb, 255:red, 208; green, 2; blue, 2 }  ,opacity=1 ] [align=left] {$\displaystyle D_{1,1,1}$};
\draw (101,414) node [anchor=north west][inner sep=0.75pt]   [align=left] {\textbf{Figure 10}: The intersection of $\displaystyle \text{Star}(v_0)$ with the standard apartment};
\draw (379,277) node [anchor=north west][inner sep=0.75pt]   [align=left] {$\displaystyle s_{1}$};
\draw (243,278) node [anchor=north west][inner sep=0.75pt]   [align=left] {$\displaystyle s_{2}$};
\draw (307,383) node [anchor=north west][inner sep=0.75pt]   [align=left] {$\displaystyle s_{0}$};
\draw (225,104) node [anchor=north west][inner sep=0.75pt]  [color={rgb, 255:red, 74; green, 144; blue, 226 }  ,opacity=1 ] [align=left] {$\displaystyle h_{1,1}$};
\draw (388,112) node [anchor=north west][inner sep=0.75pt]  [color={rgb, 255:red, 74; green, 144; blue, 226 }  ,opacity=1 ] [align=left] {$\displaystyle k_{1,1}$};
\draw (189,211) node [anchor=north west][inner sep=0.75pt]  [color={rgb, 255:red, 245; green, 152; blue, 35 }  ,opacity=1 ] [align=left] {$\displaystyle h_{1}$};
\draw (432,212) node [anchor=north west][inner sep=0.75pt]  [color={rgb, 255:red, 245; green, 152; blue, 35 }  ,opacity=1 ] [align=left] {$\displaystyle k_{1}$};

\end{tikzpicture}

\end{center}

\noindent\textbf{Remark:} 1. It is a relatively straightforward exercise to show that the peaks of Star$(v_0)$ form the bipartite graph given in Figure 9 when $q=2$, since any other possible graph which agrees with our conditions would result in a cycle of length 4 in the chambers $D_{k,j,i}$, contradicting Lemma \ref{lem: no-cycles}. Describing the graph when $q>2$ becomes difficult to achieve by hand.\\

\noindent 2. Figure 9 reveals that the chambers $D_{k,j,i}$ form an octagonal arrangement, centred at $v_0$, when $q=2$. This arrangement cannot lie in a single apartment of $\Delta$, and it demonstrates that cycles exist in $\Delta$ of a very different nature to those that exist in the standard $\widetilde{A}_2$ Coxeter complex.\\

\noindent For the purposes of calculation, it is helpful for us to realise the vertices of $\Delta$ as the set of rank 3 $\mathcal{O}$-lattices in $K^3$ modulo scaling. Setting $\{e_1,e_2,e_3\}$ as the standard basis for $K^3$, the standard apartment $\mathcal{A}_0$ is the lattices of the form $\langle\alpha_1e_1,\alpha_2e_2,\alpha_3 e_3\rangle$ for some $\alpha_1,\alpha_2,\alpha_3\in K\backslash\{0\}$.

We can realise the hyperspecial vertex as $v_0=\mathcal{O}^3=\langle e_1,e_2,e_3\rangle$, and the hyperspecial chamber $C$ consists of $v_0$ together with $v_1=\langle e_1,e_2,\pi e_3\rangle$, $v_2=\langle e_1,\pi e_2,\pi e_3\rangle$, all of which lie in $\mathcal{A}_0$. We can also take $u_1=\langle\pi e_1,e_2,\pi e_3\rangle$, $w_1=\langle e_1,\pi e_2,e_3\rangle$, $u_{1,1}=\langle\pi e_1,e_2,e_3\rangle$, $w_{1,1}=\langle\pi e_1,\pi e_2,e_3\rangle$.\\

\noindent When considering \emph{oriented} edges, we will adopt the following convention.\\

\noindent\textbf{Convention:} Given any edge $e$ in Star$(v_0)$, we will denote by $\overset{\rightarrow}{e}$ and $\overset{\leftarrow}{e}$ the two corresponding oriented edges, so $\overset{\leftarrow}{e}=\sigma(\overset{\rightarrow}{e})$. When we realise the region pictorially, as in Figure 8, we write $\overset{\rightarrow}{e}$ when we are considering the orientation of $e$ where the origin is to the left of the target, and $\overset{\leftarrow}{e}$ when the origin is to the right of the target, e.g. $o(\overset{\to}{s_0})=v_1$ and $o(\overset{\leftarrow}{s_0})=v_1$.

\subsection{$G$-orbits in Star$(v_0)$}\label{subsec: X_2 properties}

\noindent It is clear that $\text{Stab}_G(v_0)=SL_3(\mathcal{O})$. Moreover, since $SL_3$ is semisimple, we know by \cite[Proposition 4.6.32]{Bruhat-Tits2} that an edge in $\Delta$ is fixed by an element of $G=SL_3(K)$ if and only if both its adjacent vertices are, thus we can realise the stabilisers of the edges $s_1=\{v_0,v_2\},s_2=\{v_0,v_2\}$ as

$$\text{Stab}_G(s_1)=\text{Stab}_G(v_0)\cap \text{Stab}_G(v_2)=\left\{\left(\begin{array}{ccc}
    a &   b &   c\\
    \pi d & e & f \\
    \pi g & h & i
\end{array}\right)\in SL_3(K):a,b,c,d,e,f,g,i\in \mathcal{O}\right\}$$ 

$$\text{Stab}_G(s_2)=\text{Stab}_G(v_0)\cap \text{Stab}_G(v_1)=\left\{\left(\begin{array}{ccc}
    a &  b &  c\\
    d & e &  f \\
    \pi g & \pi h & i
\end{array}\right)\in SL_3(\mathcal{O}):a,b,c,d,e,f,g,i\in \mathcal{O}\right\}$$

\noindent For convenience, from now on we will write sets of this form as a single matrices, with variable entries in $\mathcal{O}$. The subgroups $I_{s_1}$ and $I_{s_2}$ arise as the preimage of the unipotent radical of these stabilisers modulo $\pi$, so they can be realised as 

$$I_{s_1}=\left(\begin{array}{ccc}
    1+\pi a &  b &  c\\
    \pi d & 1+\pi e &  \pi f \\
    \pi g & \pi h & 1+\pi i
\end{array}\right)$$

$$I_{s_2}=\left(\begin{array}{ccc}
    1+\pi a &  \pi b &  c\\
    \pi d & 1+\pi e &   f \\
    \pi g & \pi h & 1+\pi i
\end{array}\right)$$

\noindent Note that for any facet $F$ in $\Delta$ there exists a matrix $g\in GL_3(K)$ such that $F=g\cdot F'$ for some face $F'$ of $C$, and we can caluculate $I_F=gI_{F'}g^{-1}$. For example, let $g:=\left(\begin{array}{ccc}
    \pi &  0 &  0\\
    0 & 1 &  0 \\
    0 & 0 & 1
\end{array}\right)$ and $k:=\left(\begin{array}{ccc}
    \pi &  0 &  0\\
    0 & 1 &  0 \\
    0 & 0 & \pi
\end{array}\right)$, then $e_1=h\cdot s_1$ and $e_{1,1}=g\cdot s_2$, so we can calculate


$$I_{e_1}=kI_{s_1}k^{-1}=\left(\begin{array}{ccc}
    1+\pi a &  \pi b &  c\\
    d & 1+\pi e &  f \\
    \pi g & \pi^2 h & 1+\pi i
\end{array}\right)$$


$$I_{e_{1,1}}=gI_{s_2}g^{-1}=\left(\begin{array}{ccc}
    1+\pi a &  \pi^2 b &  \pi c\\
    d & 1+\pi e &  f \\
    g & \pi h & 1+\pi i
\end{array}\right)$$


\noindent Also, let $I:=I_C$ be the pro-$p$ Iwahori subgroup of $G$, and let $K_1=I_{v_0}$ be the first congruence kernel of $SL_3(\mathcal{O})$. We can realise these subgroups explicitly as $$I=\left(\begin{array}{ccc}
    1+\pi a &  b &  c\\
    \pi d & 1+\pi e &  f \\
    \pi g & \pi h & 1+\pi i
\end{array}\right),K_1=\left(\begin{array}{ccc}
    1+\pi a &  \pi b &  \pi c\\
    \pi d & 1+\pi e &  \pi f \\
    \pi g & \pi h & 1+\pi i
\end{array}\right)$$ Moreover, $I$ is the unique Sylow $p$-subgroup of $\text{Stab}_G(C)$, while $K_1$ is precisely the stabiliser in $G$ of all vertices in Star$(v_0)$. 

\begin{lemma}\label{propn: order p}
If $D$ is a chamber of $\Delta$ and $e$ is an edge of $D$, then $I_D/I_e\cong (\mathbb{F}_q,+)$. In particular, if $q=p$ then $I_D/I_e$ is a cyclic group of order $p$.
\end{lemma}

\begin{proof}
We may assume without loss of generality that $D=C$ and $e=s_2=(v_0,v_1)$, so $I_e=\left(\begin{array}{ccc}
    1+\pi a &  \pi b &  c\\
    \pi d & 1+\pi e &   f \\
    \pi g & \pi h & 1+\pi i
\end{array}\right)$ and $I_D=\left(\begin{array}{ccc}
    1+\pi a &  b &   c\\
    \pi d & 1+\pi e &    f \\
    \pi g & \pi h & 1+\pi i
\end{array}\right)$, so clearly $I_D/I_e$ is isomorphic to $\left\{\left(\begin{array}{ccc}
    1 &  b &  0\\
    0 & 1 &   0 \\
     0 & 0 & 1
\end{array}\right):b\in\mathcal{O}/\pi\mathcal{O}\right\}\cong (\mathbb{F}_q,+)$ as required.\end{proof}

\noindent From now on, we will assume that $q=p$, i.e. the residue field of $K$ is $\mathbb{F}_p$ (so $K/\mathbb{Q}_p$ is totally ramified). We can now prove the following technical results regarding the action of $I$ on Star$(v_0)$.

\begin{lemma}\label{lem: transitivity}
$I/K_1$ acts faithfully and transitively on $S_1^{(2)}$ and $S_2^{(2)}$.
\end{lemma}

\begin{proof}

By Lemma \ref{lem: preservation}, we know that $I=I_C$ acts on $S_1^{(2)}$ and $S_2^{(2)}$, and we know that $K_1$ fixes all vertices in these sets. We will prove the statement for $S_2^{(2)}$, the result follows for $S_1^{(2)}$ by symmetry.\\

\noindent To prove faithfulness, suppose $g\in I$ and $g$ fixes the chambers in $S_1^{(2)}$. Fix a pair $(r,s)$ with $1\leq r,s\leq p$, then again by Lemma \ref{lem: preservation}, $g\cdot P_{r,s}=P_{r',s'}$ for some $r',s'$, and assume for contradiction that $P_{r,s}\neq P_{r',s'}$.\\

\noindent We know that $P_{r,s}$ is adjacent to all chambers in $\{D_{r,s,k}:k=1,\dots,p\}$, and similarly $P_{r',s'}$ is adjacent to all chambers in $\{D_{r',s',k'}:k'=1,\dots,p\}$. Fix $k,k'$ with $1\leq k,k'\leq p$, and there are unique chambers $Q,Q'\in S_2^{(2)}$ such that $D_{r,s,k}$ is adjacent to $Q$ and $D_{r',s',k'}$ is adjacent to $Q'$. 

Moreover, we can assume that $Q\neq Q'$, since our choice of $k,k'$ was arbitrary, and distinct elements of $\{D_{r,s,k}:k=1,\dots,p\}$ are adjacent to distinct chambers in $S_2^{(2)}$.\\

\noindent But $g\cdot Q=Q$, $g\cdot Q'=Q'$, $g\cdot D_{r,s,k}$ is adjacent to $g\cdot P_{r,s}=P_{r',s'}$ and $g^{-1}\cdot D_{r',s',k'}$ is adjacent to $P_{r,s}$. Moreover, $D_{r,s,k}$ and $g^{-1}\cdot D_{r',s',k'}$ share an edge, so they are adjacent chambers, as are $D_{r',s',k'}$ and $g\cdot D_{r,s,k}$. 

But $D_{r,s,k}$ and $g\cdot D_{r,s,k}$ are both adjacent to $Q=g\cdot Q$ by the same edge, so this gives us a cycle $$D_{r,s,k}\sim g\cdot D_{r,s,k}\sim D_{r',s',k'}\sim g^{-1}\cdot D_{r',s',k'}$$ of length 4 in $\Delta$, contradicting Lemma \ref{lem: no-cycles}.\\

\noindent So we conclude that $g\cdot P_{r,s}=P_{r,s}$. Since our choice of $r,s$ was arbitrary, it follows that $g$ fixes all chambers $P_{i,j},Q_{i,j}$, i.e. all chambers in Star$(v_0)$ of distance 2 from $C$. By Corollary \ref{cor: adjacent to facet}, it follows that $g$ fixes all chambers of Star$(v_0)$, and hence all vertices adjacent to $v_0$, which implies that $g\in K_1$ as required.\\

\noindent To prove transitivity, for any two chambers $Q_{j,i}$, $Q_{k,\ell}\in S_2^{(2)}$, we want to show that there exists $g\in I$ such that $g\cdot Q_{j,i}=Q_{k,\ell}$. Let us first suppose that $i=\ell$, i.e. $Q_{j,i},Q_{k,i}$ are both adjacent to $Q_i$. Since the action of the pro-$p$ group $I$ permutes the $p$ chambers $\{Q_{i,s}:1\leq s\leq p\}$ non-trivially, and the size of each orbit divides $p$, it follows that the action is transitive.

If $i\neq \ell$, then similarly there exists $g\in I$ such that $g\cdot Q_k=Q_i$, so replacing $Q_{k,\ell}$ with $g\cdot Q_{k,\ell}$, we can apply the same argument.\end{proof}

\begin{lemma}\label{lem: adjacent transitive}
If for all $i,j=1,\dots,p$, $\text{Stab}_I(w_{j,i})$ acts transitively on the set of vertices in $P_1^{(2)}$ adjacent to $w_{j,i}$
\end{lemma}

\begin{proof}
There are precisely $p$ vertices in $P_1^{(2)}$ adjacent to $w_{j,i}$, and since $\text{Stab}_I(w_{j,i})\leq I$ is a pro-$p$ group,  the size of the orbit divides $p$. So either $\text{Stab}_I(w_{j,i})$ acts transitively, or it fixes every vertex adjacent to $w_{j,i}$ in $P_1^{(2)}$.

Assume for contradiction that $\text{Stab}_I(w_{j,i})$ fixes every vertex in $P_1^{(2)}$ adjacent to $w_{j,i}$. Fix such a vertex $u_{\ell,k}$, and it follows that $\text{Stab}_I(w_{j,i})$ must permute the $p$ vertices in $S_2^{(2)}$ adjacent to $u_{k,\ell}$. So again, either it permutes them transitively or fixes all of them. But we know it fixes $w_{j,i}$, so it cannot act transitively, so it must fix them all, so applying this reasoning inductively, we deduce that $\text{Stab}_I(w_{j,i})$ fixes all vertices in $S_1^{(1)}\sqcup S_2^{(1)}$.\\ 

\noindent Using Lemma \ref{lem: transitivity}, it follows that $\text{Stab}_I(w_{j,i})\subseteq K_1$, and hence $I_{d_{j,i}}\subseteq K_1$. But from the matrix descriptions of $I_{d_{1,1}}$ and $K_1$, we know that $I_{d_{1,1}}$ is not contained in $K_1$. So since there exists $h\in I$ such that $h\cdot Q_{1,1}=Q_{j,i}$ by Lemma \ref{lem: transitivity}, it follows that $I_{d_{j,i}}=hI_{d_{1,1}}h^{-1}\not\subseteq K_1$, a contradiction.\end{proof}

\begin{proposition}\label{propn: non-trivial intersection}
For all $i,j$, $I_{P_{i}}=\langle I_{h_i},I_{e_i}\cap I\cap SL_3(\mathbb{Q}_p\rangle$ and $I_{P_{j,i}}=\langle I_{h_{j,i}},I_{e_{j,i}}\cap I\cap SL_3(\mathbb{Q}_p)\rangle$.
\end{proposition}

\begin{proof}

We will use the description of vertices in the standard apartment to prove that $I_{e_{1}}\cap I\cap SL_3(\mathbb{Q}_p)$ and $I_{e_{1,1}}\cap I\cap SL_3(\mathbb{Q}_p)$ generate $I_{P_1}/I_{h_1}$ and $I_{P_{1,1}}/I_{h_{1,1}}$ respectively. Since the sets $\{P_i:i\leq p\}$ and $\{P_{i,j}:i,j\leq p\}$ each form a single orbit under the action of $GL_3(\mathbb{Z}_p)$, the result will follow for all $i,j$.\\

\noindent Using Lemma \ref{propn: order p}, we see that $I_{P_1}/I_{h_1}$ and $I_{P_{1,1}}/I_{h_{1,1}}$ have order $p$, so it remains only to prove that $I_{e_{1,1}}\cap I\cap SL_3(\mathbb{Q}_p)\not\subseteq I_{h_1}$ and $I_{e_{1,1}}\cap I\cap SL_3(\mathbb{Q}_p)\not\subseteq I_{h_{1,1}}$.\\

\noindent We have seen that $$I_{e_1}=kI_{s_1}k^{-1}=\left(\begin{array}{ccc}
    1+\pi a &  \pi b &  c\\
    d & 1+\pi e &  f \\
    \pi g & \pi^2 h & 1+\pi i
\end{array}\right)$$ and $$I_{e_{1,1}}=gI_{s_2}g^{-1}\left(\begin{array}{ccc}
    1+\pi a &  \pi^2 b &  \pi c\\
    d & 1+\pi e &  f \\
     g & \pi h & 1+\pi i
\end{array}\right)$$ where $g:=\left(\begin{array}{ccc}
    \pi &  0 &  0\\
    0 & 1 &  0 \\
    0 & 0 & 1
\end{array}\right)$ and $k:=\left(\begin{array}{ccc}
    \pi &  0 &  0\\
    0 & 1 &  0 \\
    0 & 0 & \pi
\end{array}\right)$. We also see immediately that $h_{1,1}=g\cdot s_1$, so $$I_{h_{1,1}}=gI_{s_1}g^{-1}=\left(\begin{array}{ccc}
    1+\pi a &  \pi b &  \pi c\\
    d & 1+\pi e &  \pi f \\
     g & \pi h & 1+\pi i
\end{array}\right)$$ so $I_{e_{1,1}}\cap I\cap SL_3(\mathbb{Q}_p)$ contains $\left(\begin{array}{ccc}
    1 &  0 &  0\\
    0 & 1 &   1 \\
    0 & 0 & 1
\end{array}\right)$, which does not lie in $I_{h_{1,1}}$.\\

\noindent Finally, we calculate $I_{h_1}$ directly. We know that $$\text{Stab}_G(h_1)=\text{Stab}_G(v_0)\cap\text{Stab}_G(u_1)=\left(\begin{array}{ccc}
    a &  \pi b &   c\\
    d &  e &   f \\
     g & \pi h &  i
\end{array}\right)$$ and we deduce immediately that $$I_{h_1}=\left(\begin{array}{ccc}
    1+\pi a &  \pi b &   \pi c\\
    d &  1+\pi e &   f \\
     \pi g & \pi h &  1+\pi i
\end{array}\right)$$ so $I_{e_1}\cap I\cap SL_3(\mathbb{Q}_p)$ contains $\left(\begin{array}{ccc}
    1 &  0 &  1\\
    0 & 1 &   0 \\
    0 & 0 & 1
\end{array}\right)$, which does not lie in $I_{h_{1}}$.\end{proof}

\begin{lemma}\label{lem: key ramified}
If $K\neq\mathbb{Q}_p$, $g\in I\cap SL_3(\mathbb{Q}_p)$ and $g$ stabilises all vertices adjacent to $v_0$, then $g\in K_2$. In particular, $g\in I_e$ for all edges $e$ in Star$(v_0)$.
\end{lemma}

\begin{proof}
We are assuming that $g\in K_1$, so $$g=\left(\begin{array}{ccc}
    1+\pi a_{1,1} &  \pi a_{1,2} &  \pi a_{1,3}\\
    \pi a_{2,1} & 1+\pi a_{2,2} &   \pi a_{2,3} \\
    \pi a_{3,1} & \pi a_{3,2} & 1+\pi a_{3,3}
\end{array}\right)$$ for some $a_{i,j}\in\mathcal{O}$. But we are also assuming that $g\in SL_3(\mathbb{Q}_p)$, so $\pi a_{i,j}\in\mathbb{Q}_p\cap\mathcal{O}=\mathbb{Z}_p$ for each $i,j$.\\ 

\noindent But $v_{\pi}(\pi a_{i,j})>0$, so $\pi a_{i,j}$ is not a unit in $\mathbb{Z}_p$, which implies that $v_p(\pi a_{i,j})\geq 1$, and thus $v_{\pi}(\pi a_{i,j})\geq 2$ since the extension is ramified, and hence $v_{\pi}(a_{i,j})\geq 1$, i.e. $a_{i,j}\in\pi\mathcal{O}$.

Write $b_{i,j}:=\pi^{-1}a_{i,j}$, and we see that $g=\left(\begin{array}{ccc}
    1+\pi^2 b_{1,1} &  \pi^2 b_{1,2} &  \pi^2 b_{1,3}\\
    \pi^2 b_{2,1} & 1+\pi^2 b_{2,2} &   \pi^2 b_{2,3} \\
    \pi^2 b_{3,1} & \pi^2 b_{3,2} & 1+\pi^2 b_{3,3}
\end{array}\right)\in K_2$.\\

\noindent But if $e$ is an edge in Star$(v_0)$, then for any chamber $D$ containing $e$, all vertices of $D$ have distance no more than 2 from $v_0$, hence they are fixed by every element of $K_2$, so in particular by $g$, thus $g\in I_e$ by Proposition \ref{propn: adjacent chambers facet subgroup}.
\end{proof}




\subsection{The subgroups $H_1$ and $H_2$}

Recall from Definition \ref{defn: H_e and H_f} how we defined the subgroups $H_e\subseteq I_e$ and $H_f\subseteq I_f$, where $e,f$ are oriented edges in a summit $D$ of $\Delta_n$, whose target is the peak $v\in D$. We will now explore these subgroups in more detail when $n=0$, $v$ is the hyperspecial vertex, $e=\overset{\to}{s_2}$ and $f=\overset{\leftarrow}{s_1}$.\\

\noindent Firstly for each $i,j,k,\ell=1,\dots,p$, let $$H_{1,j,i}:=\{g\in I_{s_2}:g\cdot P_{j,i}=P_{j,i}\}$$ and $$H_{2,\ell,k}:=\{g\in I_{s_1}:g\cdot Q_{k,\ell}=Q_{k,\ell}\}$$ Then $H_{1,j,i}$ and $H_{2,j,i}$ can be realised as $H_{s_2}$ and $H_{s_1}$ respectively if $u_{j,i}$ is joined by an edge to $w_{k,\ell}$.

\begin{lemma}\label{lem: order p}
For all $i,j,k,\ell,n,m=1,\dots,p$, 
\begin{itemize}
\item $\text{\emph{Stab}}_{G}(Q_n)\cap H_{1,j,i}=\text{\emph{Stab}}_{G}(P_m)\cap H_{2,k,\ell}=K_1$

\item $H_{1,j,i}/K_1$ and $H_{2,k,\ell}/K_1$ have order $p$.
\end{itemize}
\end{lemma}

\begin{proof}

Let $T:=\text{Stab}_{G}(Q_n)\cap H_{1,j,i}$. Then since $H_{1,j,i}$ is a pro-$p$ group which permutes $Q_1,\dots,Q_p$, it follows that $T$ must fix $Q_1,\dots,Q_p$.\\ 

\noindent But if $h\in T$ then $h\cdot P_{j,i}=P_{j,i}$ by the definition of $H_{1,j,i}$. So since $u_{j,i}$ is adjacent to $p$ vertices in $S_2^{(1)}$, and no two neighbours of $u_{j,i}$ have summits with the same base by Theorem \ref{thm: d-partite}, it follows that every chamber $Q_1,\dots,Q_p$ is adjacent via $k_1,\dots,k_p$ respectively to a summit in $S_2^{(2)}$ whose peak is connected to $u_{j,i}$.

In other words, $D_{j,i,1},\dots,D_{j,i,p}$ are adjacent only to $Q_{1,k_1},\dots, Q_{p,k_p}$ with $k_r\neq k_s$ if $r\neq s$, and $h$ fixes the base of $Q_{r,k_r}$ for each $r$.\\

\noindent Since $h\cdot P_{j,i}=P_{j,i}$, it follows that $h$ permutes $D_{j,i,1},\dots,D_{j,i,p}$, and also $Q_{1,k_1},\dots, Q_{p,k_p}$. But since $h$ fixes the base of every $Q_{r,k_r}$, this implies that $h$ fixes each $Q_{r,k_r}$.

But again, for each $r$, $T$ fixes $Q_r$, so it must permute $Q_{r,1},\dots,Q_{r,p}$. This once again implies that $h$ fixes all of them, i.e. it fixes every chamber in $S_2^{(1)}$, and hence $h\in K_1$ by Lemma \ref{lem: transitivity}.\\

\noindent So $T\subseteq K_1$, and since $K_1\subseteq H_{1,j,i}$ and $K_1$ fixes every vertex in $B(v_0,1)$, it follows that $K_1=T$. A symmetric argument shows that $K_1=\text{Stab}_{G}(P_m)\cap H_{2,k,\ell}$.\\

\noindent Moreover, since $H_{1,j,i}/T$ is a $p$-group acting on $Q_1,\dots,Q_p$, it can have size either 1 or $p$. But the action of $H_{1,j,i}$ on $Q_1,\dots,Q_p$ is non-trivial by Lemma \ref{lem: transitivity}, so it follows that $H_{1,j,i}/T=H_{1,j,i}/K_1$ has order $p$. Again, a symmetric argument shows that $H_{2,\ell,k}/K_1$ also has order $p$.\end{proof}

\noindent With this lemma, and the following result, we can now complete the proof of Proposition \ref{H-properties}.

\begin{proposition}\label{propn: generating_subgroup1}
If $u_{j,i}$ is joined to $w_{\ell,k}$ then $H_{1,j,i}\cap H_{2,\ell,k}=K_1$, and $I=\langle H_{1,j,i},H_{2,\ell,k}\rangle$.
\end{proposition}

\begin{proof}

Let us first assume that $i=j=k=\ell=1$. Since every element of $I_{s_2}$ fixes $u_1$ by Proposition \ref{propn: adjacent chambers facet subgroup}, $H_{1,1,1}$ is the set of all $g\in I_{s_2}$ that fix $u_{1,1}$. Similarly, $H_{2,1,1}$ is the set of all $g\in I_{s_1}$ that fixes $w_{1,1}$, so we can write them explicitly. $$H_{1,1,1}=\left(\begin{array}{ccc}
    1+\pi a &  \pi b &  c\\
    \pi d & 1+\pi e &   f \\
    \pi g & \pi h & 1+\pi i
\end{array}\right)\cap Stab(u_{1,1})=\left(\begin{array}{ccc}
    1+\pi a &  \pi b &  \pi c\\
    \pi d & 1+\pi e &   f \\
    \pi g & \pi h & 1+\pi i
\end{array}\right)$$

$$H_{2,1,1}=\left(\begin{array}{ccc}
    1+\pi a &  b &  c\\
    \pi d & 1+\pi e &  \pi f \\
    \pi g & \pi h & 1+\pi i
\end{array}\right)\cap Stab(w_{1,1})=\left(\begin{array}{ccc}
    1+\pi a &  b &  \pi c\\
    \pi d & 1+\pi e &  \pi f \\
    \pi g & \pi h & 1+\pi i
\end{array}\right)$$

\noindent It is clear that the intersection of these two subgroups is $\left(\begin{array}{ccc}
    1+\pi a &  \pi b &  \pi c\\
    \pi d & 1+\pi e &  \pi f \\
    \pi g & \pi h & 1+\pi i
\end{array}\right)=K_1$, and it is straightforward to see that any matrix in $SL_3(\mathcal{O})$ that is unipotent upper triangular modulo $\pi$ can be written as a product of matrices in these subgroups, so it follows that $I=\langle H_{1,1,1},H_{2,1,1}\rangle$.\\

\noindent In the general case, we can apply Lemma \ref{lem: transitivity} to find an element $h_0\in I$ such that $h_0\cdot Q_{1,1}=Q_{\ell,k}$, and applying Lemma \ref{lem: adjacent transitive} we can choose $h_1\in I$ such that $$h_1\cdot Q_{\ell,k}=Q_{\ell,k}\text{ and }h_1\cdot h_0P_{1,1}=P_{j,i}$$ Let $h:=h_1h_0$. Then for any $g\in G$, $g\cdot P_{j,i}=P_{j,i}$ if and only if $h^{-1}gh\cdot P_{1,1}=P_{1,1}$, so $T_{1,j,i}=hT_{1,1,1}h^{-1}$.\\ 

\noindent On the other hand, since $Q_{\ell,k}=h_0\cdot Q_{1,1}$ and $h_1\cdot Q_{\ell,k}=Q_{\ell,k}$, it follows that $h\cdot Q_{1,1}=Q_{\ell,k}$, so we similarly deduce that $H_{2,\ell,k}=hH_{2,1,1}h^{-1}$, so $H_{1,j,i}\cap H_{2,\ell,k}=h(H_{1,1,1}\cap H_{2,1,1})h^{-1}=K_1$ and $\langle H_{1,j,i},H_{2,\ell,k}\rangle=h\langle H_{1,1,1},H_{2,1,1}\rangle h^{-1}=I$.\end{proof}

\noindent We can also use the subgroups $H_1$ and $H_2$ to prove the following result, which we used in the proof of Lemma \ref{lem: (A) implies (B)}.

\begin{proposition}\label{propn: generating_subgroup2}
If $s=s_0=(v_1,v_2)$, then $I_{s}$ is generated $I_{v_1}$ and $I_{v_2}$.
\end{proposition}

\begin{proof}
For convenience, let $A:=I_{v_1}$, $B:=I_{v_2}$. Then $A$ and $B$ are both $GL_3(K)$-conjugate to $I_{v_0}=K_1$, so they are both pro-$p$ subgroups of $SL_3(K)$, normal in the stabiliser of $v_1$ and $v_2$ respectively. Thus $A\subseteq I_{e_i}$ and $B\subseteq I_{d_i}$ for each $i$, and $A,B\subseteq I_s$ by Proposition \ref{propn: adjacent chambers facet subgroup}. It remains to prove that $A$ and $B$ generate $I_s$.\\

\noindent Firstly, note that $B$ acts non-trivially on $e_1,\dots,e_p$, since $K_1\cong B$ does not fix any edge outside Star$(v_0)$. So since $B\subseteq I_s$, $I_s$ acts non-trivially on $e_1,\dots,e_p$. Again, since $I_s$ is a pro-$p$ group, it follows that  $I_s/\text{Stab}_{I_s}(e_i)$ has order $p$ for each $i$.

Therefore, fixing $i=1$, $T:=\text{Stab}_{I_s}(e_1)$, $I_s/T=B/T$, so it remains to prove that $T$ is generated by $A\cap T$ and $B\cap T$.\\

\noindent Since $A$ acts trivially on $e_1,\dots,e_p$, we know that $A\subseteq T$, so it suffices to show that $T/A$ has order $p$, and that $B\cap T$ is not contained in $A$. For the former statement, note that we can realise $T$ as $$T=\{g\in I_{s}:g\cdot P_1=P_1\}$$ If we perform an isometry of the building which sends $C$ to $Q_1$, fixing $v_2$, sending $v_1$ to $v_0$ and $v_0$ to $u_1$, then this subgroup coincides with $H_{2,\ell,k}$, where $Q_{k,\ell}$ is the image of $P_1$ under this isometry. Furthermore, $A$ coincides with $K_1$ (the stabiliser of all vertices of distance 1 from $v_0$), so it follows from Lemma \ref{lem: order p} that $T/A$ has order $p$.\\

\noindent To show that $B\cap T$ is not contained in $A$, by applying the same isometry, this is equivalent to showing that $\{g\in H_{2,\ell,k}:g\cdot v=v$ if $d(v,v_2)\leq 1\}$ is not contained in $K_1$.\\

\noindent Without loss of generality, we may assume that $k=\ell=1$, and as in the proof of Proposition \ref{propn: generating_subgroup1}, we see that $H_{2,\ell,k}=\left(\begin{array}{ccc}
    1+\pi a &  b &  \pi c\\
    \pi d & 1+\pi e &   \pi f \\
    \pi g & \pi h & 1+\pi i
\end{array}\right)$.\\ 

\noindent Moreover, since the matrix $h=\left(\begin{array}{ccc}
    1 &  0 &  0\\
    0 & \pi &  0 \\
    0 & 0 & \pi
\end{array}\right)\in GL_3(K)$ sends $v_0$ to $v_2$, we see that $$\left\{g\in G:g\cdot v=v\text{ if }d(v,v_2)\leq 1\right\}=hK_1h^{-1}=\left(\begin{array}{ccc}
    1+\pi a &  b &  c\\
    \pi^2 d & 1+\pi e &   \pi f \\
    \pi^2 g & \pi h & 1+\pi i
\end{array}\right)$$

\noindent And thus 
\begin{align*}
\left\{g\in H_{2,\ell,k}:g\cdot v=v\text{ if }d(v,v_2)\leq 1\right\}&=\left(\begin{array}{ccc}
    1+\pi a &  b &  \pi c\\
    \pi d & 1+\pi e &   \pi f \\
    \pi g & \pi h & 1+\pi i
\end{array}\right)\cap \left(\begin{array}{ccc}
    1+\pi a &  b &  c\\
    \pi^2 d & 1+\pi e &   \pi f \\
    \pi^2 g & \pi h & 1+\pi i
\end{array}\right)\\
&=\left(\begin{array}{ccc}
    1+\pi a &   b &  \pi c\\
    \pi^2 d & 1+\pi e &   \pi f \\
    \pi^2 g & \pi h & 1+\pi i
\end{array}\right)
\end{align*}

\noindent and this is not contained in $K_1=\left(\begin{array}{ccc}
    1+\pi a &  \pi b &  \pi c\\
    \pi d & 1+\pi e &   \pi f \\
    \pi g & \pi h & 1+\pi i
\end{array}\right)$\end{proof}

\subsection{Approaching Conjecture \ref{conj: did not get it}.}\label{subsec: isolation}

To complete the proof of exactness of the local oriented chain complex, it remains only to prove Conjecture \ref{conj: did not get it}, i.e. to prove that any $I$-invariant region $\mathcal{X}$ of Star$(v_0)$ is Star$(v_0)$-collapsible.

We will conclude with a proof of this in the simplest case, when $\mathcal{X}=\Delta_0$, i.e. $\mathcal{X}$ consists only of the hyperspecial chamber $C$.\\ 

\noindent Fix a chain $\beta\in C_1(\text{Star}(v_0),\mathbb{X})$ which is $(I,\Delta_0)$ shift invariant and $\varepsilon_0(\beta)\in C_0(\Delta_0,\mathbb{X})$, i.e. $\varepsilon_0(\beta)$ is non-zero only on the vertices of $C$. As per Definition \ref{defn: collapsible}, we want to prove that $\beta\in C_1(\Delta_0,\mathbb{X})+\varepsilon_1(C_2(\Delta,\mathbb{X}))$, i.e. $\beta$ has a shift $\beta'$ which is non-zero only on the edges of $C$.

\begin{definition}\label{defn: weak isolation}
We say that $\beta$ satisfies the \emph{isolation property} if $\beta(e)=0$ for all oriented edges $e$ of $X_{0,2}$ which contain $v_0$, but are not contained in $C$. The diagram below illustrates this when $p=2$:

\begin{center}

\tikzset{every picture/.style={line width=0.75pt}} 

\begin{tikzpicture}[x=0.75pt,y=0.75pt,yscale=-1,xscale=1]

\draw   (323,209) -- (438,378) -- (208,378) -- cycle ;
\draw [color={rgb, 255:red, 245; green, 152; blue, 35 }  ,draw opacity=1 ]   (61,242) -- (208,378) ;
\draw [color={rgb, 255:red, 245; green, 152; blue, 35 }  ,draw opacity=1 ]   (159,198) -- (208,378) ;
\draw [color={rgb, 255:red, 245; green, 152; blue, 35 }  ,draw opacity=1 ]   (438,378) -- (476,202) ;
\draw [color={rgb, 255:red, 245; green, 152; blue, 35 }  ,draw opacity=1 ]   (438,378) -- (598,235) ;
\draw [color={rgb, 255:red, 74; green, 144; blue, 226 }  ,draw opacity=1 ]   (145,56) -- (59,244) ;
\draw [color={rgb, 255:red, 74; green, 144; blue, 226 }  ,draw opacity=1 ]   (210,77) -- (61,242) ;
\draw [color={rgb, 255:red, 74; green, 144; blue, 226 }  ,draw opacity=1 ]   (253,34) -- (159,198) ;
\draw [color={rgb, 255:red, 74; green, 144; blue, 226 }  ,draw opacity=1 ]   (296,82) -- (159,198) ;
\draw [color={rgb, 255:red, 74; green, 144; blue, 226 }  ,draw opacity=1 ]   (484,92) -- (599,234) ;
\draw [color={rgb, 255:red, 74; green, 144; blue, 226 }  ,draw opacity=1 ]   (457,136) -- (599,234) ;
\draw [color={rgb, 255:red, 74; green, 144; blue, 226 }  ,draw opacity=1 ]   (362,55) -- (476,202) ;
\draw [color={rgb, 255:red, 74; green, 144; blue, 226 }  ,draw opacity=1 ]   (476,202) -- (378,27) ;
\draw [color={rgb, 255:red, 208; green, 2; blue, 2 }  ,draw opacity=1 ]   (210,77) -- (457,136) ;
\draw [color={rgb, 255:red, 208; green, 2; blue, 2 }  ,draw opacity=1 ]   (210,77) -- (378,27) ;
\draw [color={rgb, 255:red, 208; green, 2; blue, 2 }  ,draw opacity=1 ]   (145,56) -- (484,92) ;
\draw [color={rgb, 255:red, 208; green, 2; blue, 2 }  ,draw opacity=1 ]   (145,56) -- (362,55) ;
\draw [color={rgb, 255:red, 208; green, 2; blue, 2 }  ,draw opacity=1 ]   (296,82) -- (457,136) ;
\draw [color={rgb, 255:red, 208; green, 2; blue, 2 }  ,draw opacity=1 ]   (296,82) -- (362,55) ;
\draw [color={rgb, 255:red, 208; green, 2; blue, 2 }  ,draw opacity=1 ]   (253,34) -- (378,27) ;
\draw [color={rgb, 255:red, 208; green, 2; blue, 2 }  ,draw opacity=1 ]   (253,34) -- (484,92) ;

\draw (315,299) node [anchor=north west][inner sep=0.75pt]   [align=left] {$\displaystyle C$};
\draw (315,223) node [anchor=north west][inner sep=0.75pt]   [align=left] {$\displaystyle v_{0}$};
\draw (122,282) node [anchor=north west][inner sep=0.75pt]  [color={rgb, 255:red, 245; green, 152; blue, 35 }  ,opacity=1 ] [align=left] {$ $};
\draw (407,355) node [anchor=north west][inner sep=0.75pt]   [align=left] {$\displaystyle v_{2}$};
\draw (223,354) node [anchor=north west][inner sep=0.75pt]   [align=left] {$\displaystyle v_{1}$};
\draw (45,402) node [anchor=north west][inner sep=0.75pt]   [align=left] {\textbf{Figure 11:} The isolation property implies that $\displaystyle \beta $ is non-zero only on the visible edges.};
\draw (382,278) node [anchor=north west][inner sep=0.75pt]   [align=left] {$\displaystyle s_{1}$};
\draw (311,378) node [anchor=north west][inner sep=0.75pt]   [align=left] {$\displaystyle s_{0}$};
\draw (248,278) node [anchor=north west][inner sep=0.75pt]   [align=left] {$\displaystyle s_{2}$};

\end{tikzpicture}

\end{center}

\end{definition}

\begin{proposition}\label{propn: final step}

Suppose $\beta$ satisfies the isolation property. Then there exists a shift of $\beta$ in $C_1(\Delta_0,\mathbb{X})$.
    
\end{proposition}

\begin{proof}

First, let $\beta_1$ be the chain defined by $$\beta_1(e):=\begin{cases}\beta(e) & e\in\Delta_0\\ 0 & \text{otherwise}\end{cases}$$ and let $\beta_2:=\beta-\beta_1$. Then $\beta_2$ is non-zero only on edges outside $\Delta_0$, and its images on these edges agree with the images of $\beta$. In particular, $\beta_2$ also satisfies the isolation property.

It suffices to show that we can find a shift $\beta_2'\in C_1(\Delta_0,\mathbb{X})$ of $\beta_2$, and since $\beta_1\in C_1(\Delta_0,\mathbb{X})$, it will  follow that $\beta':=\beta_1+\beta_2'$ is a shift of $\beta$ in $C_1(\Delta_0,\mathbb{X})$.\\ 

\noindent Replace $\beta$ with $\beta_2$, and we can now assume (in light of the isolation property) that $\beta$ is non-zero only on the edges $\{e_i,d_i,e_{j,i},d_{j,i},r_{k,j,i}:1\leq i,j,k\leq p\}$. In particular, for each $i,j=1,\dots,p$, $$\varepsilon_0(\beta)(v_1)=\underset{1\leq i\leq p}{\sum}{\beta(\overset{\to}{e_i})}$$ 

$$\varepsilon_0(\beta)(u_i)=\beta(\overset{\leftarrow}{e_i})+\underset{1\leq j\leq p}{\sum}{\beta(\overset{\leftarrow}{e_{j,i}})}$$ and $$\varepsilon_0(\beta)(u_{j,i})=\beta(\overset{\to}{e_{j,i}})+\underset{1\leq k\leq p}{\sum}{\beta(\overset{\leftarrow}{r_{k,j,i}})}$$ Moreover, we know that $\varepsilon_0(\beta)$ is zero outside of $\Delta_0$, so we know that $\varepsilon_0(u_{i})=\varepsilon_0(u_{j,i})=\varepsilon_0(w_{j,i})=0$ for all $i,j\leq p$, so we can write $\beta(\overset{\leftarrow}{e_i})=-\underset{1\leq j\leq p}{\sum}{\beta(\overset{\leftarrow}{e_{j,i}})}$ and $\beta(\overset{\to}{e_{j,i}})=-\underset{1\leq k\leq p}{\sum}{\beta(\overset{\leftarrow}{r_{k,j,i}})}$ and hence 
\begin{equation}\label{eqn: weak isolation equation}
\varepsilon_0(\beta)(v_1)=\underset{1\leq i\leq p}{\sum}{\beta(\overset{\to}{e_i})}=-\underset{1\leq i\leq p}{\sum}{\beta(\overset{\leftarrow}{e_i})}=-\underset{1\leq j\leq p}{\sum}{\beta(\overset{\leftarrow}{e_{j,i}})}=\underset{1\leq j\leq p}{\sum}{\beta(\overset{\to}{e_{j,i}})}=\underset{1\leq j,k\leq p}{\sum}{\beta(\overset{\leftarrow}{r_{k,j,i}})}
\end{equation}

\noindent Symmetrically, we deduce that $\varepsilon_0(\beta)(v_2)=\underset{1\leq j,k\leq p}{\sum}{\beta(\overset{\to}{r_{k,j,i}})}=-\varepsilon_0(\beta)(v_1)$, thus $\varepsilon_0(\beta)(v_1)\in\mathbb{X}^{I_{v_1}}\cap\mathbb{X}^{I_{v_2}}=\mathbb{X}^{\langle I_{v_1},I_{v_2}\rangle}=\mathbb{X}^{I_{s_0}}$ by Proposition \ref{propn: generating_subgroup2}.\\

\noindent Finally, define $\beta'\in C_1(\Delta_0,\mathbb{X})$ by $$\beta'(e)=\begin{cases}
    \varepsilon_0(\beta)(v_1) & e=\overset{\leftarrow}{s_0}\\\varepsilon_0(\beta)(v_2)=-\varepsilon_0(\beta)(v_1) & e=\overset{\to}{s_0}\\ 0 & \text{otherwise}
\end{cases}$$ and we see immediately that $\varepsilon_0(\beta')(v)=\varepsilon_0(\beta)(v)$ for all vertices $v$ of $C$, and hence $\varepsilon_0(\beta)=\varepsilon_0(\beta')$, and thus $\beta'$ is a shift of $\beta$ as required.\end{proof}

\noindent So it remains to prove that there exists a shift of $\beta$ which satisfies the isolation property. From now on, we will make the further assumption that $K\neq\mathbb{Q}_p$, and we will set $A:=I\cap SL_3(\mathbb{Q}_p)$. We will consider the action of $A$ on Star$(v_0)$.

\begin{lemma}\label{lem: ramified}
Let $\mathcal{Y}:=\{D_{k,j,i}:1\leq i,j,k\leq p\}$, then there exists a $\mathcal{Y}$-shift $\beta'$ of $\beta$ such that $\beta'(g\cdot \overset{\to}{r_{k,j,i}})=g\beta'(\overset{\to}{r_{k,j,i}})$ for all $g\in A$, $i,j,k\leq p$.
\end{lemma}

\begin{proof}

Let $Y:=\{r_{k,j,i}:1\leq i,j,k\leq p\}$, and since the edges in $Y$ all lie on the border on $\Delta_2$, and outside of $\mathcal{X}$, it remains to prove that the action of $A$ on these edges satisfies all the hypotheses of Lemma \ref{lem: crucial} to deduce the existence of such a shift $\beta'$.\\

\noindent We will first prove hypothesis 2, i.e. that the action of $A$ on $Y$ is transitive. We saw in section \ref{subsec: X_2 properties} that we can realise $I/K_1$ as the group of unipotent, upper triangular matrices in $M_3\left(\mathcal{O}/\pi\mathcal{O}\right)=M_3(\mathbb{F}_p)$, and this agrees with $(SL_3(\mathbb{Q}_p)\cap I)K_1/K_1$, so we only need to prove that $I/K_1$ acts transitively on $Y=\{r_{k,j,i}:1\leq i,j,k\leq p\}$.

But $Y$ contains precisely $p^3$ edges, and $I/K_1$ has order $p^3$, so it suffices to show that the stabiliser of any edge in $Y$ under $I/K_1$ is trivial. But if $g\in I$ fixes $r_{k,j,i}$, then $g\cdot D_{k,j,i}=D_{k,j,i}$, and applying Corollary \ref{cor: gallery fixing} we see that $g\cdot P_{j,i}=P_{j,i}$ and $g\cdot P_i=P_i$. Since $g$ must also stabilise all the chambers in $\{Q_{k,\ell}:1\leq k,\ell\leq p\}$ that are adjacent to $D_{k,j,i}$, it follows from Lemma \ref{lem: transitivity} that $g\in K_1$ as required.\\

\noindent This also proves that Stab$_A(r_{k,j,i})=K_1$ for all $1\leq i,j,k\leq p$, which implies hypothesis 3 of Lemma \ref{lem: crucial}. Moreover, since $A\subseteq I$, it is clear that $A/K_1=(A\cap I)/K_1$, which is precisely hypothesis 4, so it remains only to prove hypothesis 1, i.e. that if $g\in A$ and $g\cdot r_{k,j,i}=r_{k,j,i}$ for some $i,j,k$, then $g\in I_{r_{k,j,i}}$. 

But we know that if $g\cdot r_{k,j,i}=r_{k,j,i}$ then $g\in K_1\cap SL_3(\mathbb{Q}_p)$, i.e. $g$ stabilises all vertices adjacent to $v_0$. So since $K\neq\mathbb{Q}_p$ it follows from Lemma \ref{lem: key ramified} that $g\in I_{r_{k,j,i}}$ for all $i,j,k$.\end{proof}

\noindent In light of this result, we replace $\beta$ with $\beta'$ and now assume that $\beta'(g\cdot \overset{\to}{r_{k,j,i}})=g\beta'(\overset{\to}{r_{k,j,i}})$ for all $g\in A$, and hence $\{\beta(\overset{\to}{r_{k,j,i}}):1\leq i,j,k\leq p\}$ forms a single $A$-orbit.

\begin{proposition}\label{propn: top and bottom}
There exists a Star$(v_0)$-shift $\beta'$ of $\beta$ which satisfies the isolation property.
\end{proposition}

\begin{proof}

\noindent First, fix $i,j\leq p$, and let $A_{j,i}:=$ Stab$_A e_{j,i}$. Then $A_{j,i}$ permutes $Y_{j,i}:=\{r_{k,j,i}:1\leq k\leq p\}$, so $\{\beta(\overset{\leftarrow}{r_{k,j,i}}):1\leq k\leq p\}$ must form a single $A_{j,i}$-orbit, and thus the sum $\underset{1\leq k\leq p}{\sum}{\beta(\overset{\leftarrow}{r_{k,j,i}})}$ is $A_{j,i}$-invariant.\\

\noindent But we know that 
\begin{equation}\label{eqn: isolation1}
0=\varepsilon_0(\beta)(u_{j,i})=\beta(\overset{\rightarrow}{e_{j,i}})+\underset{1\leq k\leq p}{\sum}{\beta(\overset{\leftarrow}{r_{k,j,i}})}+\beta(\overset{\leftarrow}{h_{j,i}})
\end{equation}
\noindent and since $\beta(\overset{\rightarrow}{e_{j,i}})$ is $I_{e_{j,i}}$-invariant, clearly it is $A_{j,i}\cap I_{e_{j,i}}$-invariant.  So we conclude that $\beta(\overset{\leftarrow}{h_{j,i}})$ is  $A_{j,i}\cap I_{e_{j,i}}$-invariant.

But $\beta(\overset{\leftarrow}{h_{j,i}})\in\mathbb{X}^{I_{h_{j,i}}}$, so it is invariant under the subgroup generated by $A_{j,i}\cap I_{e_{j,i}}$ and $I_{h_{j,i}}$. But we know that $I_{h_{j,i}}$ and $A_{j,i}\cap I_{e_{j,i}}=I\cap SL_3(\mathbb{Q}_p)\cap I_{e_{j,i}}$ generate $I_{P_{j,i}}$ by Proposition \ref{propn: non-trivial intersection}, so it follows that $\beta(\overset{\leftarrow}{h_{j,i}})\in\mathbb{X}^{I_{P_{j,i}}}$\\

\noindent Therefore, define $\gamma_1\in C_2(\Delta,\mathbb{X})$ by $$\gamma_1(D,c):=\begin{cases}
    -\beta(\overset{\leftarrow}{h_{j,i}}) & (D,c)=(P_{j,i},c_{j,i})\text{ for some }1\leq i,j\leq p\\
    \beta(\overset{\leftarrow}{h_{j,i}}) & (D,c)=(P_{j,i},-c_{j,i})\text{ for some }1\leq i,j\leq p\\
    0 & \text{otherwise}
\end{cases}$$ where $c_{j,i}$ is the orientation of $P_{j,i}$ which agrees with the orientation of $\overset{\leftarrow}{h_{j,i}}$. Let $\beta_1:=\beta+\varepsilon_1(\gamma_2)$, so that $\beta_1(\overset{\leftarrow}{h_{j,i}})=\beta(\overset{\leftarrow}{h_{j,i}})-\beta(\overset{\leftarrow}{h_{j,i}})=0$ for all $1\leq i,j\leq p$.\\

\noindent Moreover, $\beta_1$ is a Star$(v_0)$ shift of $\beta$, and $\beta_1$ agrees with $\beta$ on the edges $\{r_{k,j,i}:i,j,k\leq p\}$, so we can replace $\beta$ with $\beta_1$ and assume from now on that $\beta(\overset{\to}{h_{j,i}})=0$ for all $i,j$.

In particular, using (\ref{eqn: isolation1}), we know that $\beta(\overset{\to}{e_{j,i}})=-\underset{1\leq k\leq p}{\sum}{\beta(\overset{\leftarrow}{r_{k,j,i}})}$, and we also calculate \begin{equation}\label{eqn: isolation2}
0=\varepsilon_0(\beta)(u_{i})=\beta(\overset{\leftarrow}{e_{i}})+\underset{1\leq j\leq p}{\sum}{\beta(\overset{\leftarrow}{e_{j,i}})}+\beta(\overset{\leftarrow}{h_{i}})
\end{equation}

\noindent But $\underset{1\leq j\leq p}{\sum}{\beta(\overset{\leftarrow}{e_{j,i}})}=\underset{1\leq j,k\leq p}{\sum}{\beta(\overset{\leftarrow}{r_{k,j,i}})}$ is invariant under $A_i:=$ Stab$_A(e_i)$, so since $\beta(\overset{\leftarrow}{e_{i}})$ is invariant under $I_{e_i}$ we deduce that $\beta(\overset{\leftarrow}{h_{i}})$ is invariant under $A_i\cap I_{e_i}=I_{e_i}\cap I\cap SL_3(\mathbb{Q}_p)$.

But we know that $\beta(\overset{\leftarrow}{h_{i}})\in\mathbb{X}^{I_{h_i}}$, so it follows that it is invariant under $\langle I_{h_i},I_{e_i}\cap I\cap SL_3(\mathbb{Q}_p)\rangle$, which is equal to $I_{P_i}$ by Proposition \ref{propn: non-trivial intersection} again. So $\beta(\overset{\leftarrow}{h_{i}})\in\mathbb{X}^{I_{P_i}}$ for all $i\leq p$.\\

\noindent Therefore, define $\gamma_2\in C_2(\Delta,\mathbb{X})$ by $$\gamma_2(D,c):=\begin{cases}
    -\beta(\overset{\leftarrow}{h_{i}}) & (D,c)=(P_{i},c_{i})\text{ for some }1\leq i\leq p\\
    \beta(\overset{\leftarrow}{h_{i}}) & (D,c)=(P_{i},-c_{i})\text{ for some }1\leq i\leq p\\
    0 & \text{otherwise}
\end{cases}$$ where $c_{i}$ is the orientation of $P_{i}$ which agrees with the orientation of $\overset{\leftarrow}{h_{i}}$. Let $\beta_2:=\beta+\varepsilon_1(\gamma_2)$, so that $\beta_2(\overset{\leftarrow}{h_{i}})=\beta(\overset{\leftarrow}{h_{i}})-\beta(\overset{\leftarrow}{h_{i}})=0$ for all $1\leq i\leq p$.\\

\noindent So we have now defined a Star$(v_0)$ shift $\beta_2$ of $\beta$ that is zero on the edges $h_i$ and $h_{j,i}$, but which agrees with $\beta$ on the edges $r_{k,j,i},k_{j,i}$ and $k_i$. So applying a symmetric argument to $\beta_2$, we find a Star$(v_0)$-shift $\beta'$ of $\beta_2$ that is zero on $h_i,k_i,h_{j,i}$ and $k_{j,i}$ for all $i,j$, i.e. $\beta'$ satisfies the isolation property.\end{proof}

\noindent So, finally, we deduce the following case of Conjecture \ref{conj: did not get it}

\begin{theorem}\label{thm: final}
$\Delta_0$ is a \emph{Star}$(v_0)$-collapsible region of $\Delta$.    
\end{theorem}

\begin{proof}

This follows immediately from Proposition \ref{propn: final step} and Proposition \ref{propn: top and bottom}.\end{proof}

\noindent To complete a full proof of Conjecture \ref{conj: did not get it} (and hence of Conjecutre \ref{sur} for $SL_3(K)$), we would need to establish that the remaining star-bounded regions are collapsible. In fact, it suffices to prove this for the following two regions.

\begin{itemize}
    \item $\mathcal{X}_1=\{C,P_1,\dots,P_p,Q_1,\dots,Q_p\}$.

    \item $\mathcal{X}_2=\mathcal{X}_1\cup\{P_{j,i}:1\leq i,j\leq p\}$.
\end{itemize}

\noindent We can prove that $\mathcal{X}_1$ is Star$(v_0)$ collapsible when $p=2$ (and using the argument in section \ref{sec: orbits}, this is enough to show that $\Delta_n$ is collapsible for all $n\in\mathbb{N}$), but we have not yet been able to establish a similar result for the asymmetric region $\mathcal{X}_2$.

It is very possible that a completely new idea is needed to generalise this argument. In which case, we hope these preliminary results will reignite interest in this project within the community, and new ideas may be presented which will lift them to a full proof for $G$ of type $\widetilde{A}_2$, and perhaps generalise them to arbitrary types.

\bibliographystyle{abbrv}

\begin{small}

\end{small}

\end{document}